\documentclass{article}

\usepackage{arxiv}

\usepackage[utf8]{inputenc} 
\usepackage[T1]{fontenc}    
\usepackage{hyperref}       
\usepackage{url}            
\usepackage{booktabs}       
\usepackage{amsfonts}       
\usepackage{nicefrac}       
\usepackage{microtype}      
\usepackage{graphicx}
\usepackage{amsmath}
\usepackage{array}
\usepackage{doi}
\usepackage{amsthm}
\usepackage{dsfont}
\usepackage{mathtools}
\usepackage{tcolorbox}
\usepackage{amsmath}
\usepackage{physics}
\usepackage{bbm}
\usepackage{wasysym}
\usepackage{caption}
\usepackage{subcaption}
\usepackage{multirow}
\usepackage{algorithm}
\usepackage{algpseudocode}
\usepackage{mathtools}
\usepackage{amssymb}
\usepackage{soul}
\DeclarePairedDelimiter{\ceil}{\lceil}{\rceil}

\newtheorem{theorem}{Theorem}[section]
\newtheorem{assumption}[theorem]{Assumption}
\newtheorem{lemma}[theorem]{Lemma}
\newtheorem{proposition}[theorem]{Proposition}
\newtheorem{corollary}[theorem]{Corollary}
\newtheorem{remark}[theorem]{Remark}
\newtheorem{definition}[theorem]{Definition}

\newtheorem{example}[theorem]{Example}
\newtheorem{property}[theorem]{Property}

\newcommand{\R}{\mathbb{R}}
\newcommand{\N}{\mathbb{N}}

\usepackage[normalem]{ulem}

\title{Symplectic QTT-FEM solution of the one-dimensional acoustic wave equation in the time domain}

\author{ \href{https://orcid.org/0009-0005-7783-4438}{\includegraphics[scale=0.06]{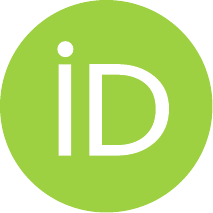}\hspace{1mm} Sara Fraschini} \\
	Faculty of Mathematics \\
    University of Vienna \\
    1090 Vienna, Austria \\
	\texttt{sara.fraschini@univie.ac.at} \\
	\And
	\href{https://orcid.org/0000-0002-9445-3267}
        {\includegraphics[scale=0.06]{orcid.pdf}\hspace{1mm}Vladimir Kazeev}\\
        Faculty of Mathematics \\
    University of Vienna \\
    and Research Network Data Science\\
    1090 Vienna, Austria \\
	\texttt{vladimir.kazeev@univie.ac.at} \\
	 \And
    \href{https://orcid.org/0000-0003-1368-2883}
        {\includegraphics[scale=0.06]{orcid.pdf}\hspace{1mm}Ilaria Perugia}\\
        Faculty of Mathematics \\
    University of Vienna \\
    1090 Vienna, Austria \\
	\texttt{ilaria.perugia@univie.ac.at} \\
}

\date{}

\renewcommand{\undertitle}{}
\renewcommand{\shorttitle}{}

\hypersetup{
pdftitle={symplectic QTT-FEM wave equation},
pdfsubject={q-bio.NC, q-bio.QM},
pdfauthor={Fraschini S.~Kazeev V.~Perugia I.},
pdfkeywords={symplectic QTT-FEM, wave equation},
}

\begin{document}
\maketitle

\begin{abstract}
Structured Finite Element Methods (FEMs) based on low-rank approximation in the form of the so-called \emph{Quantized Tensor Train (QTT)} decomposition (QTT-FEM) have been proposed and extensively studied in the case of elliptic equations. In this work, we design a QTT-FE method for time-domain acoustic wave equations, combining stable low-rank approximation in space with a suitable  conservative discretization in time. For the acoustic wave equation with a homogeneous source term in a single space dimension as a model problem, we consider its reformulation as a first-order system in time. In space, we employ a low-rank QTT-FEM discretization based on continuous piecewise linear finite elements corresponding to uniformly refined nested meshes. Time integration is performed using symplectic high-order Gauss--Legendre Runge--Kutta methods.
In our numerical experiments, we investigate the energy conservation and exponential convergence
of the proposed method.
\end{abstract}


\section{Introduction}\label{sec:introduction}

In this paper, we leverage quantized tensor decompositions to propose a novel scalable numerical method for the acoustic wave equation in the time domain, focusing on the setting of a single space dimension. The numerical scheme combines symplectic Gauss--Legendre time integrators with a low-rank truncated piecewise linear Finite Element Method (FEM) in space. The finite element spaces are induced by uniform nested meshes. The refinement of such discretizations is limited by their unfavorable tradeoff between accuracy and complexity that stems from the conventional approach of independently storing and directly accessing their degrees of freedom. Instead, these large
finite element spaces can be reparametrized in terms of a much lower number of parameters by the so-called \emph{Quantized Tensor Train} (QTT) decomposition. This QTT-structured FEM is meant to benefit from an approximation accuracy comparable to the one provided by the finest finite element mesh while requiring significantly fewer parameters than the corresponding finite element space. The resulting computational complexity reduction allows us to use finite element mesh size equal to $2^{-15}$ at each step of our time-stepping scheme. To the best of our knowledge, this work represents the first attempt in the literature to provide a QTT-structured discretization of the time-dependent acoustic wave equation. 

\subsection{Tensor-structured approximation}
In scientific computing, column vectors, matrices, and higher-dimensional arrays are commonly referred to as tensors.
Specifically, a real tensor of order $d$, or a $d$-dimensional tensor, is an element of $\R^{n_1 \times \ldots \times n_d}$, for positive sizes $n_1,\ldots,n_d \in \N$. Low-rank approximation can be generalized from the standard matrix setting to the case of tensors in various ways~\cite{KB2009}, which results in respective decompositions of tensors as generalizations of the low-rank factorization of matrices.
Just as low-rank representations of matrices can significantly reduce the cost of storage and computations, the low-rank tensor approximation achieves the same for higher-dimensional arrays, mitigating thereby the otherwise prohibitive complexity of storing and manipulating such data. For an overview of tensor decompositions and their computational applications, we refer to the surveys~\cite{KB2009,GKT2013,H2014} and to the comprehensive books~\cite{K2018,H2019}.

In the context of Partial Differential Equations (PDEs), low-rank tensor approximation has been initially developed to reduce the computational complexity of the numerical solution of PDEs on high-dimensional domains, which is a typical setting in quantum-physical computations or stochastic modeling; see the survey~\cite{B2023} and references therein. In this context, suitable low-rank tensor-structured representations compress the high-dimensional coefficient arrays representing the discrete solutions, understood as high-order tensors. These representations are essentially generalizations of the classical separation of variables, in which complex multivariate functions are decomposed in terms of much simpler, univariate functions. 

Even though the development of tensor approximation has historically been centered around high-dimensional problems and the so-called curse of dimension~\cite{Bellman:1961}, many PDE problems posed on domains of moderate dimensionality are still extremely difficult to solve numerically.
 Capturing specific features of solutions, such as singularities or high-frequency oscillations, may require an impractically large number of degrees of freedom when using discretization methods that are not tailored to the problem. In this context, standard solvers, such as uniform low-order FEMs, become computationally prohibitive, even for problems posed on one-, two-, or three-dimensional domains. As a result, more efficient discretization methods, such as \emph{hp}-FEMs~\cite{GB1986a,GB1986b} and homogenization techniques~\cite{AHP2021,AEEV2012}, have been developed. An alternative emerging approach is the low-rank tensor-structured approximation of high-order tensors representing the coefficients with respect to generic, low-order finite element bases
 associated with extremely fine grids. For the relation between the higher-order tensors and the FEM coefficients, there is a natural choice that associates the dimensions of the former with the scales of the latter, as opposed to its original, ``physical'' variables. Such interpretation of scales as independent variables subject to the separation of variables has been referred to as \emph{quantization}, or \emph{tensorization}. For example,
  a vector of length $2^L$ is reshaped into an $L^{\text{th}}$-order tensor of size $2\times \cdots \times 2$, and the low-rank approximation of the latter yields the approximation of the former; see, e.g.,~\cite{K2011,G2010,O2009b,O2010,KK2012}. This efficient low-parametric tensor-structured representation of vectors and matrices allows for the use of extremely fine uniform grids in the discretization. This results in a general approach for the approximation of PDEs with singularities or multiscale structures, which would otherwise require highly specialized methods. In this setting, low-order FEMs prove to be impressively accurate, using up to $L = 50$ discretization levels~\cite{BK2020,KORS2022}, and therefore grid sizes as fine as $2^{-50}$.

\subsection{Previous work on QTT}
The Tensor Train (TT) decomposition, introduced in \cite{OT2009a,O2011,O2009a,OT2009b}, is a multilinear low-parametric representation of multi-dimensional arrays based on the separation of variables.
This tensor decomposition traces back to the field of quantum physics~\cite{W1993,V2003,Verstraete:2006:MPS,Schollwoeck:2011:DMRG-MPS}, where it is known as \emph{Matrix Product States} (MPS) and has been used to model quantum spin systems. 

Compared to other low-rank tensor representations, the TT decomposition has several crucial properties a combination of which is necessary for its use for the numerical solution of PDEs. First, the TT decomposition can drastically reduce the storage requirements associated with higher-order tensors: the number of parameters grows quadratically with the ranks and linearly, rather than exponentially, with the number of dimensions it separates. Additionally, a quasi-optimal TT approximation of a given tensor with given ranks can be computed from the SVDs of a sequence of auxiliary matrices~\cite[Algo. 1]{O2011}. Moreover, many basic operations of linear algebra can be efficiently implemented within the TT decomposition, and a stable and fast rounding procedure for reducing the ranks of intermediate decompositions produced by such operations is available, see~\cite{Verstraete:2006:MPS} and~\cite[Algo. 2]{O2011}. This combination of properties makes the TT decomposition particularly suitable for high-dimensional computations.

When combined with the quantization, the TT decomposition is referred to as \emph{Quantized} Tensor Train or \emph{Quantics} Tensor Train (QTT) decomposition~\cite{O2009b,O2010,K2011,G2010}. Despite being a relatively recent low-rank representation technique, the QTT decomposition has already proven effective in numerical linear algebra and scientific computing fields. 

As shown in~\cite{O2010,KK2012,KKT2013,KRS2013}, various discrete operators can be represented or approximated with low rank in the QTT decomposition. This is crucial for the low-rank tensor-based approximation of PDE solutions. In this context, low-rank representations are needed for the corresponding problem data. When applied for the representation of matrices and vectors arising from uniform low-order discretizations, the QTT decomposition captures and exploits their self-similarity across different scales.

In PDE settings, the QTT-FEM, a combination of the QTT decomposition with simple low-order discretization methods based extremely fine grids, can provide approximations of solutions that have high accuracy and low complexity. Many problems for differential operators have been considered in the literature~\cite{DKO2012,KKNS2014,K2015,KS2018,BK2020,MRS2022,KORS2017,KORS2022,MRU2022,MTO2021}. In particular, exponential convergence rates with respect to the total number of representation parameters, similar to those shown for \emph{hp}-FEMs for singular solutions, have been established for solutions of elliptic PDEs with singularities or high-frequency oscillations in~\cite{KS2018,BK2020,MRS2022,KORS2017,KORS2022}.

When combined with suitable preconditioning techniques, such as the tensor-structured BPX preconditioner introduced in~\cite{BK2020} for uniformly elliptic problems and applied to one-dimensional singularly perturbed problems in~\cite{MRU2022}, low-rank QTT representations of discrete solutions on grids whose sizes approach machine precision can be used, making adaptive mesh refinement generally unnecessary in this setting. 

The power of the QTT decomposition lies in its ability to handle computationally challenging PDEs without relying on problem-specific \emph{a priori} information, such as high-order regularity or efficient approximation spaces. For example, for the approximation of singular solutions of PDE problems, the QTT-FEM approach does not require prior knowledge of the type or the exact location of the singularity. The discrete problem approximating the PDE problem is built using low-order finite element functions associated with a highly refined uniform grid, and the degrees of freedom of the discrete problem are never accessed directly but are manipulate exclusively via QTT approximations of FEM expansion coefficients. As a result, the approximation accuracy comparable to that of the underlying generic, low-order FEM discretization is achieved, while the low-rank compression afforded by low-rank QTT approximation drastically reduces the number of effective parameters (provided that the solution admits such approximation). For the QTT-FEM, the complexity of storage and computations is governed by the ranks of the QTT decompositions involved.
The QTT decomposition is effective for solving PDEs only when those ranks can be set to low values without compromising the accuracy of the method.
By theoretical analysis and numerical evidence, that approximability assumption was shown to hold for the solutions of several classes of elliptic PDE problems; we refer to~\cite{KS2018,MRS2022,KORS2022,MRU2022} for rank bounds that are polylogarithmic with respect to the Sobolev-norm accuracy.

QTT-structured numerical methods have been extensively explored in the literature for various classes of PDEs. Among the others, we refer to~\cite{K2015,KS2018,MRS2022,KORS2017,KORS2022,MRU2022,BK2020,MTO2021} for second-order elliptic PDEs,~\cite{KORS2017} for the one-dimensional Helmholtz equation,~\cite{KKNS2014,KS2015} for the chemical master equation,~\cite{KO2010} for the molecular Schr{\"o}dinger equation, and~\cite{DKO2012} for the Fokker--Planck equation. However, to the best of our knowledge, the use of QTT-structured numerical schemes for the acoustic wave equation in the time domain has not yet been addressed.

Building a QTT-structured numerical scheme for the time-dependent acoustic wave equation presents several challenges. If a space--time formulation is considered, desirable is uniform well-conditioning for the block system encompassing both space and time, achieving which is closely related to developing an asymptotically optimal \emph{space--time preconditioner}. Additionally, treating time as an additional one of the dimensions separated by tensor approximation means imposing, in particular, low-rank structure with respect to the spatial and temporal degrees of freedom. This can pose difficulties in problems with moving fronts in the space--time domain, where the spatial and temporal variables do not separate with low rank. If a QTT-structured discretization in space is combined with time stepping, the absence of a Courant--Friedrichs--Lewy (CFL) condition is essential for the feasibility of the combined solver. Furthermore, uniform well-conditioning is desirable for the family of linear systems solved during the time-stepping iterations, and the conservation properties of the continuous problem should be preserved with sufficient accuracy under low-rank approximation throughout time stepping.

\subsection{Contribution and outline}
We start from the first-order-in-time reformulation of the acoustic wave equation proposed in~\cite{BL1994,FP1996}. We discretize it by combining implicit Gauss--Legendre Runge--Kutta (GLRK) time integrators with QTT-FEM in space. As a consequence
of the GLRK stability, the resulting method is unconditionally stable, i.e., no CFL condition is required for stability. Moreover, the GLRK symplecticity~\cite[Ch. VI, Th. 4.2]{HLW2006} ensures the energy conservation in the absence of errors due to low-rank approximation. As demonstrated by our numerical tests, the energy preservation is still satisfied in the presence of such approximation. Additionally, for the single-stage GLRK scheme (i.e., the midpoint method), we observe numerically that the linear system at each time step is uniformly well-conditioned. For higher-order GLRK methods, we propose a stabilization technique based on the BPX preconditioner from~\cite{BK2020}. While the preconditioned system is not uniformly well-conditioned, the ill-conditioning is mitigated by the action of the proposed operator, allowing for favorable results.

The paper is structured as follows. In Section~\ref{sec:model problem}, we present the variational framework for the model problem under consideration. The construction of the finite element space underlying the QTT-FEM in space is described in Section~\ref{sec:semidiscrete}. Section~\ref{sec:time_discr} provides a brief summary of GLRK time integrators and the rationale for employing these schemes in our setting. In Section~\ref{sec::fully_discrete}, we present a fully-discrete method in a way that is amenable to parametrization by the QTT decomposition. The QTT decomposition itself and related notions are reviewed in Section~\ref{sec:qtt}. Section~\ref{sec:tt in space} illustrates how to represent and solve the devised time-stepping method in the QTT decomposition. In Section~\ref{subsec:initdata}, we describe how to obtain QTT-structured discrete initial data. Section~\ref{subsec:QTT_GLRK-FEM} is devoted to the combination of the QTT-FEM with the time stepping. The issue of ill-conditioning in the case of high-order GLRK schemes, along with a stabilization technique, is discussed in Section~\ref{subsec:high-order_GLRK-FEM}. Finally, in Section~\ref{sec:numerical_experiments}, we present
numerical experiments investigating the energy conservation and exponential accuracy of the QTT-structured discretization. 

\subsection{Notation}
Boldface letters are used to denote arrays. Capitalized letters are used for arrays of order $d>1$, and non-capitalized letters are used for vectors. The array entries are denoted by square brackets, i.e., $\boldsymbol{A}[i_1,\ldots,i_d]$ for arrays of order $d>1$, and $\boldsymbol{a}[i]$ for vectors.

For scalar quantities $A$ and $B$, $A \lesssim B$ is used to denote $A \leq C B$ with a constant $C>0$ that is independent of any parameters explicitly appearing in the expressions $A$ and $B$.

Given a time interval $(a,b) \subset \R$ and a function $w: (a,b) \rightarrow \R$, we denote with $\dot{w}$ 
its derivative in time. Instead, for a space domain $S \subset \R$ and $w: S \rightarrow \R$, we denote the first derivative with $w'$. For a bounded domain $D \subset \R$, $L^2(D)$ is the Lebesgue space of (classes of) real-valued square-integrable functions defined on $D$, endowed with the usual Hilbertian norm $\|\cdot\|_{L^2(D)}$. With the usual notation, we denote with $H^1(D)$ the Hilbert space of (classes of) square-integrable real-valued functions with square-integrable weak derivatives. We also consider $H^1_0(D)$ and its dual $H^{-1}(D)$, the former with the norm $\|w\|_{H^1_0(D)}=\|w'\|_{L^2(D)}$. Additionally, for $n \in \N$, $\mathbb{P}_n(D)$ is the space of polynomials of degree $n$ defined on $D$.

For functions depending on space and time, with possibly different regularities in space and time variables, we employ the usual notion of Bochner spaces of vector-valued functions. Specifically, let $[a,b] \subset \R$ be a time interval and $H$ be separable real Hilbert space. We denote with $C^0([a,b];H)$ and $C^1([a,b];H)$ the spaces of vector-valued functions $w: [a,b] \rightarrow H$ that are, respectively, strongly continuous and continuously differentiable. Additionally, $L^2(a,b;H)$ is the Hilbert space of (classes of) measurable functions $w: (a,b) \rightarrow H$ such that 
$$
 \int_{(a,b)} \|w(t)\|^2_H \ dt  < \infty.
$$
With this notation, and for an integer $k \in \N$, the Bochner Sobolev space $H^k(a,b;H)$ is defined as
$$
H^k(a,b;H) := \{ w \in L^2(a,b;H): \partial^{\boldsymbol{j}}_t w \in L^2(a,b;H) \text{ for all } |\boldsymbol{j}|\leq k\}.
$$

\section{Model problem}\label{sec:model problem}
Given the space domain $\Omega := (0,1)$, and a finite final time $T>0$, the problem under consideration is the one-dimensional acoustic wave problem
\begin{equation}\label{eq_wave}
\begin{cases}
\partial_{tt}u(x,t)-\partial_{xx} u(x,t)=0, \quad &(x,t) \in \Omega \times (0,T),\\
u(x,t)=0, \quad &(x,t) \in \partial\Omega \times (0,T),\\
u(x,0)=u_0(x), \quad
\partial_tu(x,t)_{|t=0}=v_0(x), \quad &x \in \Omega,
\end{cases}
\end{equation}
with initial conditions satisfying $u_0\in H^1_0(\Omega)$ and $v_0\in L^2(\Omega)$. 

As it is well known (see e.g.,~\cite[Ch.~V, \S29]{W1987},~\cite[Ch.~3, \S8]{LM1972}, and~\cite[Ch.~XVIII, \S6.1]{DL1992}), there exists a unique solution of the variational formulation:
\begin{equation*}\label{var_wave_2ndorder}
\begin{aligned}
&\text{Find } u \in 
L^2\left(0,T;H^1_0(\Omega)\right)\cap H^1\left(0,T;L^2(\Omega)\right) \text{ such that:}\\
&\displaystyle 
{}_{H^{-1}(\Omega)}\!\left\langle \partial_{tt}u(\cdot,t),\lambda\right\rangle_{H^1_0(\Omega)}
+ \int_\Omega \partial_x u(x,t) \hspace{0.05cm}  \lambda'(x) \ dx = 0 \quad \text{for all } \lambda \in H^1_0(\Omega),
\end{aligned}
\end{equation*}
for almost all $t \in (0,T)$, together with the initial conditions
\begin{equation*}
u(0) = u_0 \quad \text{in } H^1_0(\Omega), \qquad
\partial_tu(0) = v_0 \quad \text{in}\quad L^2(\Omega).
\end{equation*}
Furthermore, as observed in~\cite[Ch.~XVIII, \S6.1]{DL1992} and~\cite[Rmk.~8.2]{LM1972}, $u \in C^0\left([0,T];H^1_0(\Omega)\right)\cap C^1\left([0,T];L^2(\Omega)\right)
\cap H^2\left(0,T;H^{-1}(\Omega)\right)$.

In what follows, we consider the position--velocity reformulation of problem~\eqref{eq_wave} that has been proposed in~\cite{BL1994,FP1996}, namely,
\begin{equation}\label{eq_wave_ham}
\begin{cases}
\partial_{t}u(x,t) -v(x,t) = 0, \quad &(x,t) \in \Omega \times (0,T),\\
\partial_{t}v(x,t)-\partial_{xx} u(x,t)=0, \quad &(x,t) \in \Omega \times (0,T),\\
u(x,t)=0, \quad &(x,t) \in \partial{\Omega} \times (0,T),\\
u(x,0)=u_0(x), \quad
v(x,0)=v_0(x), \quad &x \in \Omega.
\end{cases}
\end{equation}
The pointwise in time variational formulation of problem~\eqref{eq_wave_ham} reads as:
\begin{equation}\label{var_wave}
    \begin{aligned}
        &\text{Find } u \in L^2\left(0,T;H^1_0(\Omega)\right) \cap H^1\left(0,T;L^2(\Omega)\right) \text{ and } v \in L^2\left(0,T;L^2(\Omega)\right) \text{ such that }\\
        &\begin{cases}
\displaystyle \int_\Omega \partial_tu(x,t)\hspace{0.05cm}\xi(x) \ dx -\int_\Omega v(x,t) \hspace{0.05cm}\xi(x) \ dx = 0 \quad &\text{for all } \xi \in L^2(\Omega),\\
\displaystyle 
{}_{H^{-1}(\Omega)}\!\left\langle \partial_{t}v(\cdot,t),\lambda\right\rangle_{H^1_0(\Omega)}
+ \int_\Omega \partial_x u(x,t) \hspace{0.05cm}  \lambda'(x) \ dx = 0 \quad &\text{for all } \lambda \in H^1_0(\Omega),
\end{cases}
    \end{aligned}
\end{equation}
for almost all $t \in (0,T)$, together with the initial conditions
\begin{equation}\label{in_data}
u(0) = u_0 \quad \text{in } H^1_0(\Omega), \qquad
v(0) = v_0 \quad \text{in } L^2(\Omega).
\end{equation}
Moreover~$u \in 
C^0\left([0,T];H^1_0(\Omega)\right)\cap C^1\left([0,T];L^2(\Omega)\right)$ and $v \in C^0\left([0,T];L^2(\Omega)\right)\cap
H^1\left(0,T;H^{-1}(\Omega)\right)$.

The energy of the Hamiltonian system~\eqref{var_wave} reads as
\begin{equation}\label{Energy}
E(t) := \int_\Omega \left( \frac{1}{2} |\partial_x u(x,t)|^2 + \frac{1}{2} |v(x,t)|^2  \right) dx, \quad t \in [0,T].
\end{equation}
A direct differentiation with respect to time gives that the total energy satisfies 
\begin{equation*}
\dot{E}(t) = 0 \quad \text{for all } t \in [0,T],
\end{equation*}
which implies
\begin{equation*}
E(t) = E(0) = \int_\Omega \left( \frac{1}{2} |u'_0(x)|^2 + \frac{1}{2} |v_0(x)|^2  \right) dx \geq 0 \quad \text{for all } t \in [0,T].
\end{equation*}
Motivated by the loss of accuracy for long-time computations suffered by numerical schemes that increase or dissipate energy, we shall devise a conservative discretization of formulation~\eqref{var_wave}. 

\section{Semi-discrete formulation in space}\label{sec:semidiscrete}
Intending to exploit tensor-structured techniques to offer accurate numerical approximations of the wave equation based on straightforward discretizations, the space discretization of the variational formulation~\eqref{var_wave} is based on a conforming piecewise linear FEM. 
\subsection{Low-order 
discretization in space}\label{sec:var_semidiscrete}
From now on, let $L \in \mathbb{N}$ be the fixed number of levels associated with
the iterative uniform binary partitioning of the spatial domain $\Omega$:
$$
\overline{\Omega} = \bigcup_{i = 1}^{2^L} K_{L,i},\quad\text{with}\quad
K_{L,i}:=[x_{L,i-1},x_{L,i}] \subset \mathbb{R}, \quad i \in \{1,\ldots,2^L\},
$$
where
$
    x_{L,i}=2^{-L}i$
    for $i \in \{0,\ldots,2^L\}.
$
Then,
\begin{equation*}
    \mathcal{T}_L:=\{K_{L,1},\ldots,K_{L,2^L}\}
\end{equation*}
is a uniform mesh of~$\Omega$ with mesh size $2^{-L}$.
As a discretization space, we employ the space $\mathcal{S}_1(\mathcal{T}_L)$ of functions that are continuous on $\overline{\Omega}$ and affine on each $K_{L,i}$, namely,
\begin{equation*}
    \mathcal{S}_1(\mathcal{T}_L) :=\{w_L \in H^1(\Omega): {w_L}_{|K_{L,i}} \in \mathbb{P}_1(K_{L,i}) \ \forall K_{L,i} \in \mathcal{T}_L\}.
\end{equation*}
To accommodate the homogeneous essential boundary conditions, we employ the finite element space
\begin{equation*}
    V_L := \mathcal{S}_1(\mathcal{T}_L) \cap H^1_0(\Omega).
\end{equation*}

The semi-discrete conforming FEM for formulation~\eqref{var_wave} reads as: 
\begin{equation}\label{semidis_wave}
\begin{aligned}
&\text{For any~$t\in(0,T)$, find } u_L(\cdot,t) \in V_L 
\text{ and } v_L(\cdot,t) \in V_L \text{ such that }\\
&\begin{cases}
\displaystyle \int_\Omega \partial_t{u}_L(x,t)\hspace{0.05cm}\xi_L(x) \ dx -\int_\Omega v_L(x,t) \hspace{0.05cm}\xi_L(x) \ dx = 0 \quad &\text{for all } \xi_L \in V_L,\\
\displaystyle \int_\Omega \partial_t{v}_L(x,t) \hspace{0.05cm}\lambda_L(x) \ dx + \int_\Omega \partial_x u_L(x,t) \hspace{0.05cm} \lambda_L'(x) \ dx = 0 \quad &\text{for all } \lambda_L \in V_L,
\end{cases}
\end{aligned}
\end{equation}
together with initial conditions
\begin{equation}\label{in_data_semidis}
u_L(0) = P_L u_0, \quad 
v_L(0) = \Pi_L v_0, 
\end{equation}
where $P_L: H^1_0(\Omega) \rightarrow V_L$ is the elliptic projection 
\begin{equation}\label{ell_proj}
    \int_\Omega (P_Lw)'(x) \hspace{0.05cm}\chi'_L(x) \ dx = \int_\Omega w'(x) \hspace{0.05cm}\chi'_L(x) \ dx \quad \text{for all } \chi_L \in V_L,
\end{equation}
and $\Pi_L: L^2(\Omega) \rightarrow V_L$ is the $L^2$ projection 
\begin{equation}\label{l2_proj}
    \int_\Omega \Pi_Lw(x) \hspace{0.05cm}\chi_L(x) \ dx = \int_\Omega w(x) \hspace{0.05cm}\chi_L(x) \ dx \quad \text{for all } \chi_L\in V_L.
\end{equation}

The semi-discrete energy of the conforming FEM~\eqref{semidis_wave} is
\begin{equation}\label{E_semidis}
E_L(t) := \int_\Omega \left(\frac{1}{2} |\partial_x u_L(x,t)|^2 + \frac{1}{2} |v_L(x,t)|^2 \right) dx, \quad  t \in [0,T].
\end{equation}
As for the continuous setting, a direct differentiation with respect to time gives the global conservation of the semi-discrete energy:
\begin{equation*}
E_L(t) = E_L(0) = \int_\Omega \left(\frac{1}{2} \left|\left(P_Lu_0\right)'(x)\right|^2 + \frac{1}{2} \left|\Pi_Lv_0(x)\right|^2 \right) dx \geq 0 \quad \text{for all } t \in [0,T].
\end{equation*}

\subsection{Matrix form of the semi-discrete problem}\label{sec:matrix_form}
Let $N_L := 2^L-1$ be the dimension of the space~$V_L$.
We denote with $\{\widehat{\varphi}_{L,j}\}_{j=1}^{N_L}\subset V_L$ the $L^2$\emph{-normalized nodal basis of hat functions}, i.e., 
the set of basis functions of $V_L$ that are affine on each $K_{L,i} \in \mathcal{T}_L$, continuous on $\overline{\Omega}$, and such that
\begin{equation}\label{basis}
    \widehat{\varphi}_{L,j}(x_{L,i}) := 2^{\frac{L}{2}}\delta_{ij},
\end{equation}
for $i\in \{0,\ldots,N_L+1\}$ and $j \in \{1,\ldots,N_L\}$. The $L$-dependent normalization factor in the right-hand side of~\eqref{basis} implies the spectral equivalence between the mass matrix defined by this basis and the identity matrix, uniformly with respect to $L$.

The finite element spaces are nested: for all $L,\ell \in \mathbb{N}$ such that $\ell \leq L$, we have 
\begin{equation*}
    \widehat{\varphi}_{\ell,k} = \sum_{j=1}^{N_L} \widehat{\boldsymbol{P}}_{\ell,L}[j,k] \ \widehat{\varphi}_{L,j}, \quad k \in \{1,\ldots,N_\ell\},
\end{equation*}
where $\widehat{\boldsymbol{P}}_{\ell,L} \in \mathbb{R}^{N_L \times N_\ell}$ is the matrix representation in the nodal basis~\eqref{basis} of the prolongation operator from $V_\ell$ to $V_L$. The matrix representation of the symmetric version of the classical BPX preconditioner~\cite{BPX1990} has been introduced in~\cite{BK2020} for mixed Dirichlet--Neumann boundary conditions and considered here for Dirichlet boundary conditions. This two-sided version of the classical BPX preconditioner is given by the symmetric positive definite operator
\begin{equation}\label{BPX}
    \boldsymbol{C}_L := \sum_{\ell = 1}^{L} 2^{-\ell} \widehat{\boldsymbol{P}}_{\ell,L}\widehat{\boldsymbol{P}}^T_{\ell,L}.
\end{equation}
Let $\{\varphi_{L,i}\}_{i =1}^{N_L} \subset V_L$ be the hierarchical basis of $V_L$ defined as follows:
\begin{equation}\label{bpx_basis}
    \varphi_{L,i}:=\sum_{j=1}^{N_L} \boldsymbol{C}_L[j,i] \ \widehat{\varphi}_{L,j}, \quad  i \in \{1,\ldots,N_L\},
\end{equation}
which we call \emph{BPX-basis}. We define the degrees of freedom $\boldsymbol{u}_L=(\boldsymbol{u}_{L}[i])_{i}: [0,T]\rightarrow\mathbb{R}^{N_L}$ and $\widehat{\boldsymbol{v}}_L = (\widehat{\boldsymbol{v}}_L[j])_{j}: [0,T]\rightarrow\mathbb{R}^{N_L}$ of the approximations $u_L$ and $v_L$, respectively, as the coefficients of the representations
\begin{equation}\label{semidis_sol}
    u_L(x,t)=\sum_{i=1}^{N_L} (\boldsymbol{u}_{L}[i])(t) \ \varphi_{L,i}(x), \quad v_L(x,t)=\sum_{j=1}^{N_L} (\widehat{\boldsymbol{v}}_L[j])(t) \ \widehat{\varphi}_{L,j}(x).
\end{equation}

Let us denote with $\widehat{\boldsymbol{M}}_L \in \mathbb{R}^{N_L \times N_L}$ and $\widehat{\boldsymbol{A}}_L \in \mathbb{R}^{N_L \times N_L}$, respectively, the mass matrix and stiffness matrix with respect to the $L^2$-normalized nodal basis functions~\eqref{basis}, and $\boldsymbol{u}_{L,0} \in \mathbb{R}^{N_L}$ and $\widehat{\boldsymbol{v}}_{L,0} \in \mathbb{R}^{N_L}$, respectively, the vector representations of the coefficients of $P_Lu_0$ and $\Pi_Lv_0$  in the prescribed bases~\eqref{bpx_basis} and~\eqref{basis}. The matrix form of the space-semidiscrete formulation~\eqref{semidis_wave}
reads as
\begin{equation}\label{ode_wave}
    \begin{cases}
        \widehat{\boldsymbol{M}}_L\boldsymbol{C}_L \dot{\boldsymbol{u}}_L(t)= \widehat{\boldsymbol{M}}_L \widehat{\boldsymbol{v}}_L(t), \quad & t \in (0,T),\\
        \boldsymbol{C}_L\widehat{\boldsymbol{M}}_L \dot{\widehat{\boldsymbol{v}}}_L(t)= - \boldsymbol{C}_L\widehat{\boldsymbol{A}}_L \boldsymbol{C}_L \boldsymbol{u}_L(t), \quad &t \in (0,T),\\
        \boldsymbol{u}_L(0) = \boldsymbol{u}_{L,0},\\
        \widehat{\boldsymbol{v}}_L(0) = \widehat{\boldsymbol{v}}_{L,0}.
    \end{cases}
\end{equation}
In the next section, we devise an energy-conservative numerical integrator of this differential system.

\section{Symplectic time integrators}\label{sec:time_discr}
To provide an energy-preserving time-marching scheme for~\eqref{ode_wave}, we exploit the structure of the semi-discrete energy. 
\begin{definition}[Quadratic invariant~\protect{\cite[Ch. IV, \S2]{HLW2006}}]\label{def:quad_inv}
    Let $\boldsymbol{S}\in \mathbb{R}^{m\times m}$ be a symmetric matrix, and let $\mathcal{Q}:\mathbb{R}^m\to \mathbb{R}$ be the following quadratic function defined by $\boldsymbol{S}$:
    \begin{equation}\label{quad_inv}
        \mathcal{Q}(\boldsymbol{y}):=\boldsymbol{y}^T\boldsymbol{S}\boldsymbol{y}.
    \end{equation}
    $\mathcal{Q}$ is called a quadratic invariant of the ODE system
    \begin{equation}\label{ode_quad_inv}
        \dot{\boldsymbol{y}} = \boldsymbol{f}(\boldsymbol{y})
    \end{equation}
    if $\boldsymbol{S}$ satisfies
    \begin{equation*}
        \boldsymbol{y}^T\boldsymbol{S}\boldsymbol{f}(\boldsymbol{y}) = 0 \quad \text{for all } \boldsymbol{y}\in\mathbb{R}^m.
    \end{equation*}
\end{definition}

\begin{remark}
    The gradient of~$\mathcal{Q}$ defined in \eqref{quad_inv} is given by $\nabla\mathcal{Q}(\boldsymbol{y})=2\boldsymbol{y}^T\boldsymbol{S}$. Hence, being a quadratic invariant of an ODE system implies being constant along every solution of the given ODE.
\end{remark}

\begin{lemma}\label{lemma:quad_inv}
    The semi-discrete energy~\eqref{E_semidis} is a quadratic invariant of the ODE system~\eqref{ode_wave}.
\end{lemma}
\begin{proof}
    The representation of the semi-discrete energy in the prescribed bases \eqref{bpx_basis}, \eqref{basis} of the finite element space $V_L$ reads as
    \begin{equation*}
        E_L(t) = \begin{pmatrix}
        {\boldsymbol{u}_L(t)}^T & {\widehat{\boldsymbol{v}}_L(t)}^T
        \end{pmatrix}
        \begin{pmatrix}
        \tfrac{1}{2}\boldsymbol{C}_L\widehat{\boldsymbol{A}}_L\boldsymbol{C}_L &  \\
         & \tfrac{1}{2}\widehat{\boldsymbol{M}}_L
        \end{pmatrix}
        \begin{pmatrix}
        \boldsymbol{u}_L(t)\\
        \widehat{\boldsymbol{v}}_L(t)
        \end{pmatrix},
        \quad t \in [0,T].
    \end{equation*}
Then, the thesis follows from
the definition of quadratic invariant and the fact that $
\begin{pmatrix}
    \boldsymbol{u}_L\\
    \widehat{\boldsymbol{v}}_L
\end{pmatrix}$ solves~\eqref{ode_wave}.
\end{proof}

As a result of Lemma~\ref{lemma:quad_inv}, our discretization in space is well-suited for integration in time with Runge--Kutta methods that preserve quadratic invariants. With this in mind, and being interested in providing a time discretization of arbitrary high-order, GLRK integration is a natural choice. GLRK methods with $q$ stages are known to be implicit, A-stable, collocation schemes of order $2q$, whose collocation points are the nodes of Gauss--Legendre quadrature in $[0,1]$. Furthermore, they are proven to preserve quadratic invariants~\cite[Ch. IV, \S2.1]{HLW2006}, which implies their symplecticity~\cite[Ch. VI, Th. 4.2]{HLW2006}.

Let us recall the standard notation for the Butcher tableau that defines Runge--Kutta methods with $q$ stages, i.e.,\\
\begin{center}
\begin{tabular}[b]{c|c}
$\boldsymbol{c}^{(q)}[1]$ & $\boldsymbol{A}^{(q)}[1,1]$ $\boldsymbol{A}^{(q)}[1,2]$ $\ldots$ $\boldsymbol{A}^{(q)}[1,q]$ \\
$\boldsymbol{c}^{(q)}[2]$ & $\boldsymbol{A}^{(q)}[2,1]$ $\boldsymbol{A}^{(q)}[2,2]$ $\ldots$ $\boldsymbol{A}^{(q)}[2,q]$ \\
$\vdots$ & \hspace{-0.3cm}{$\vdots$} \hspace{1.5cm}{$\vdots$} \hspace{0.6cm}{$\ddots$} \hspace{0.6cm}{$\vdots$} \\
$\boldsymbol{c}^{(q)}[q]$ & $\boldsymbol{A}^{(q)}[q,1]$ $\boldsymbol{A}^{(q)}[q,2]$ $\ldots$ $\boldsymbol{A}^{(q)}[q,q]$ \\
\hline
 & \hspace{0.1cm}{$\boldsymbol{b}^{(q)}[1]$} \hspace{0.3cm}{$\boldsymbol{b}^{(q)}[2]$} \hspace{0.1cm}{$\cdots$} \hspace{0.1cm}{$\boldsymbol{b}^{(q)}[q]$}
\end{tabular}  
\hspace{0.2cm} \raisebox{1.8ex}{=} \hspace{0.2cm}
\begin{tabular}[b]{ c|c } 
    $\boldsymbol{c}^{(q)}$ & $\boldsymbol{A}^{(q)}$ \\
     \hline
           & ${\boldsymbol{b}^{(q)}}^T$
\end{tabular}
\qquad \raisebox{1.8ex}{with $\boldsymbol{A}^{(q)} \in \R^{q \times q}$, and $\boldsymbol{b}^{(q)}, \boldsymbol{c}^{(q)} \in \R^{q}$.}
\end{center}
In the GLRK setting, the collocation points $\mathbf{c}^{(q)} = (\boldsymbol{c}^{(q)}[i])_{i} \in \R^{q}$ are the zeros of the $q^{\text{th}}$ shifted Legendre polynomial 
$$
\frac{d^{q}}{ds^{q}} \left(s^{q}(s-1)^{q}\right), \quad s \in [0,1],
$$
and the Runge--Kutta coefficients $\boldsymbol{A}^{(q)} =(\boldsymbol{A}^{(q)}[i,j])_{i,j} \in \R^{q \times q}$ and $\boldsymbol{b}^{(q)} = (\boldsymbol{b}^{(q)}[i])_{i} \in \R^{q}$ read as
$$
\boldsymbol{A}^{(q)}[i,j] := \int_{0}^{\boldsymbol{c}^{(q)}[i]} \ell^{(q)}_j(s) \ ds, \qquad \boldsymbol{b}^{(q)}[i] := \int_0^1 \ell^{(q)}_i(s) \ ds,
$$
for $i,j \in \{1,\ldots,q\}$, where $\ell^{(q)}_i$ is the Lagrange polynomial $\ell^{(q)}_i(s) := \Pi_{j \ne i} \tfrac{s - \boldsymbol{c}^{(q)}[j]}{\boldsymbol{c}^{(q)}[i] - \boldsymbol{c}^{(q)}[j]}$.

\begin{example}\label{ex:GLRK}
GLRK methods of order $2$ (implicit midpoint rule), $4$ (the method of Hammer \& Hollingsworth~\cite{HH1955}) and $6$, respectively, read as
\begin{center}
\setlength{\extrarowheight}{5pt} 
\begin{tabular}{c@{\qquad}c@{\qquad}c} 
    \begin{tabular}[b]{ c|c } 
        $\frac{1}{2}$ & $\frac{1}{2}$ \\
        \hline
        & $1$
    \end{tabular}
    &
    \begin{tabular}[b]{ c|c } 
        $\frac{1}{2} - \frac{\sqrt{3}}{6}$ & \hspace{0.2cm}{$\frac{1}{4}$} \hspace{0.8cm}{$\frac{1}{4} -  \frac{\sqrt{3}}{6}$}\\
        $\frac{1}{2} + \frac{\sqrt{3}}{6}$ & \hspace{-0.7cm}{$\frac{1}{4} +  \frac{\sqrt{3}}{6}$} \hspace{0.7cm}{$\frac{1}{4}$} \\
        \hline
        & \hspace{-0.35cm}{$\frac{1}{2}$} \hspace{1.25cm}{$\frac{1}{2}$}
    \end{tabular}
    &
    \begin{tabular}[b]{ c|c }  
        $\frac{1}{2} - \frac{\sqrt{15}}{10}$ & \hspace{0.5cm}{$\frac{5}{36}$} \hspace{1cm}{$\frac{2}{9} -  \frac{\sqrt{15}}{15}$} \hspace{0.8cm}{$\frac{5}{36} -  \frac{\sqrt{15}}{30}$} \\
        $\frac{1}{2}$ & \hspace{0.1cm}{$\frac{5}{36} +  \frac{\sqrt{15}}{24}$} \hspace{0.9cm}{$\frac{2}{9}$} \hspace{1.3cm}{$\frac{5}{36} -  \frac{\sqrt{15}}{24}$}\\
        $\frac{1}{2} + \frac{\sqrt{15}}{10}$ & \hspace{-0.5cm}{$\frac{5}{36} +  \frac{\sqrt{15}}{30}$} \hspace{0.45cm}{$\frac{2}{9} +  \frac{\sqrt{15}}{15}$} \hspace{1.2cm}{$\frac{5}{36}$}\\ 
        \hline
        & \hspace{-0.2cm}{$\frac{5}{18}$} \hspace{1.7cm}{$\frac{4}{9}$} \hspace{1.6cm}{$\frac{5}{18}$}
    \end{tabular}
\end{tabular}
\end{center}

\end{example}

\section{Fully-discrete formulation}\label{sec::fully_discrete}
In this section, we combine 
GLRK schemes with the space FEM discretization of Section~\ref{sec:semidiscrete} (GLRK--FEM), providing a family of unconditionally stable, energy-preserving, fully-discrete methods for the wave equation~\eqref{eq_wave_ham}. The unconditional stability is related to the A-stability of Gauss--Legendre integrators, the energy conservation is implied by their preservation of quadratic invariants. 

The ODE-system~\eqref{ode_wave} 
can be written as follows
\begin{equation}\label{ode_wave_compact}
\begin{cases}
    \begin{pmatrix}
        \dot{\boldsymbol{u}}_L(t)\\
        \dot{\widehat{\boldsymbol{v}}}_L(t)
    \end{pmatrix}
    = \boldsymbol{B}_{2L}
    \begin{pmatrix}
        \boldsymbol{u}_L(t)\\
        \widehat{\boldsymbol{v}}_L(t)
    \end{pmatrix},
    \quad  t\in (0,T),\\
    \\
    \begin{pmatrix}
        \boldsymbol{u}_L(0)\\
        \widehat{\boldsymbol{v}}_L(0)
    \end{pmatrix}
    =
    \begin{pmatrix}
        \mathbf{u}_{L,0}\\
        \widehat{\mathbf{v}}_{L,0}
    \end{pmatrix},
\end{cases}
\end{equation}
where we have defined the matrix $\boldsymbol{B}_{2L}$ as
\begin{equation}\label{matrix_midp}
    \boldsymbol{B}_{2L}:=
\begin{pmatrix}
 & \boldsymbol{C}_L^{-1}\\
-\widehat{\boldsymbol{M}}_L^{-1}\widehat{\boldsymbol{A}}_L\boldsymbol{C}_L & 
\end{pmatrix}
\in \R^{2N_L\times2N_L}.
\end{equation}

Let us consider a uniform partition of the time interval $[0,T]$ into $n_t$ subintervals with step size $\tau:=\tfrac{T}{n_t}$, which gives the following discrete times:
\begin{equation}\label{discrete_times}
    0 = t_0 < t_1 < \ldots < t_{n_t} = T, \quad t_n := n\tau \quad \text{for all } n \in \{0,\ldots,n_t\}.
\end{equation}
The GLRK integration of system~\eqref{ode_wave_compact} with step size $\tau$, $q \in \N$ stages, and coefficients $\boldsymbol{A}^{(q)} =(\boldsymbol{A}^{(q)}[i,j])_{i,j} \in \R^{q \times q}$, $\boldsymbol{b}^{(q)} = (\boldsymbol{b}^{(q)}[i])_{i} \in \R^{q}$ reads as: For $n=0,1,\ldots,n_t-1$,
\begin{equation}\label{GLRK_FEM}
    \begin{aligned}
    \begin{pmatrix}
        \boldsymbol{u}_L^{(n+1)}\\
        \widehat{\boldsymbol{v}}_L^{(n+1)}
    \end{pmatrix}
    &= 
    \begin{pmatrix}
        \boldsymbol{u}_L^{(n)}\\
        \widehat{\boldsymbol{v}}_L^{(n)}
    \end{pmatrix}
    + \tau \sum_{i=1}^{q} \boldsymbol{b}^{(q)}[i] \ \boldsymbol{k}^{(n)}_{2L,i},\\
    \text{with} \quad
        \boldsymbol{k}^{(n)}_{2L,i} &= \mathbf{B}_{2L}\left(
        \begin{pmatrix}
         \boldsymbol{u}_L^{(n)}\\
         \widehat{\boldsymbol{v}}_L^{(n)}
    \end{pmatrix}
    + \tau \sum_{j=1}^{q} \boldsymbol{A}^{(q)}[i,j] \ \boldsymbol{k}^{(n)}_{2L,j} \right), \qquad i \in \{1,\ldots,q\},
    \end{aligned}
\end{equation}
where we have assumed that, for each $n \in \{1,\ldots,n_t\}$, $(\boldsymbol{u}_L^{(n)},\widehat{\boldsymbol{v}}_L^{(n)})$ denotes the numerical approximation of $(\boldsymbol{u}_L(t_n),\widehat{\boldsymbol{v}}_L(t_n))$, and $(\boldsymbol{u}_L^{(0)},\widehat{\boldsymbol{v}}_L^{(0)}) := (\mathbf{u}_{L,0}, \widehat{\mathbf{v}}_{L,0})$.

At each time $t_n = n\tau \in [0,T]$, the vectors $\boldsymbol{u}_L^{(n)} = (\boldsymbol{u}_L^{(n)}[i])_i  \in \R^{N_L}$ and $\widehat{\boldsymbol{v}}_L^{(n)} = (\widehat{\boldsymbol{v}}_L^{(n)}[j])_j \in \R^{N_L}$ provide numerical approximations of the time evaluations, respectively, of the semi-discrete position and semi-discrete velocity~\eqref{semidis_sol}, i.e., for $x \in \Omega$, the functions
\begin{equation}\label{fully_discr_sol}
    u^{(n)}_L(x):=\sum_{i=1}^{N_L} \boldsymbol{u}_L^{(n)}[i] \ \varphi_{L,i}(x) \qquad \text{and} \qquad v^{(n)}_L(x):=\sum_{j=1}^{N_L} \widehat{\boldsymbol{v}}_L^{(n)}[j] \ \widehat{\varphi}_{L,j}(x)
\end{equation}
approximate, respectively, 
\begin{equation}\label{semi_discr_sol_eval}
     u_L(x,t_n)=\sum_{i=1}^{N_L} (\boldsymbol{u}_{L}[i])(t_n) \ \varphi_{L,i}(x) \qquad \text{and} \qquad
     v_L(x,t_n)=\sum_{j=1}^{N_L} (\widehat{\boldsymbol{v}}_L[j])(t_n) \ \widehat{\varphi}_{L,j}(x).
\end{equation}

At each time $t_n = n\tau \in [0,T]$, with $n\in \{0,\ldots,n_t\}$, the fully-discrete energy is given by
\begin{equation}\label{E_dis}
\begin{aligned}
    E_L^{(n)} &:= \int_\Omega \left(\frac{1}{2} \left| \left(u^{(n)}_L\right)'(x)\right|^2 + \frac{1}{2} \left|v^{(n)}_L(x)\right|^2 \right) dx\\
    &= \begin{pmatrix}
        \boldsymbol{u}^{(n) T}_L& \widehat{\boldsymbol{v}}^{(n) T}_L
        \end{pmatrix}
        \begin{pmatrix}
        \tfrac{1}{2}\boldsymbol{C}_L\widehat{\boldsymbol{A}}_L\boldsymbol{C}_L &  \\
         & \tfrac{1}{2}\widehat{\boldsymbol{M}}_L
        \end{pmatrix}
        \begin{pmatrix}
        \boldsymbol{u}^{(n)}_L\\
        \widehat{\boldsymbol{v}}^{(n)}_L
        \end{pmatrix}.
\end{aligned}
\end{equation}
Due to the exact preservation of quadratic invariants by GLRK methods in exact arithmetic, the following equality between the fully-discrete and semi-discrete energy is satisfied:
\begin{equation*}
    E^{(n)}_L = E_L(t_n) \quad \text{for all} \quad t_n =n\tau \in [0,T]. 
\end{equation*}
Recalling the global conservation law of the semi-discrete energy, this equality implies
\begin{equation*}
    E^{(n)}_L = E_L(0) = \int_\Omega \left(\frac{1}{2} |(P_Lu_0)'(x)|^2 + \frac{1}{2} |\Pi_Lv_0(x)|^2 \right) dx \geq 0 \quad \text{for all } n \in \{0,\ldots,n_t\},
\end{equation*}
entailing that the error between the fully-discrete and exact energy is only related to the error in the projections~\eqref{ell_proj} and~\eqref{l2_proj} of the initial data when assuming that rounding errors are negligible.

\begin{remark}\label{FP}
It is noted in~\cite{FP1996} that the GLRK--FEM~\eqref{GLRK_FEM} is equivalent to a space--time discretization of the acoustic wave equation~\eqref{eq_wave_ham} analyzed in~\cite{FP1996}. Specifically,
\begin{itemize}
    \item As a result of~\cite[Th. 8]{FP1996}, if the continuous initial data $u_0$ and $v_0$ satisfy suitable compatibility conditions, the fully-discrete solutions~\eqref{fully_discr_sol} approximate the time evaluations of the exact solutions of the 
    variational formulation~\eqref{eq_wave_ham} with accuracy, respectively, $\mathcal{O}(2^{-2L}+\tau^{2q})$ and $\mathcal{O}(2^{-L}+\tau^{2q})$, for $L \rightarrow \infty$ and $q \rightarrow \infty$, in $\|\cdot\|_{L^2(\Omega)}$ and $\|\cdot\|_{H^1_0(\Omega)}$.
    \item For other optimal-order choices of the discrete initial data, the estimates for the position remain valid, while the error estimates for the velocity require $u_L(0) := P_Lu_0$.
\end{itemize}
\end{remark}

Let us reformulate the GLRK--FEM method~\eqref{GLRK_FEM} in a way that is amenable to tensor-structured parametrizations. From~\eqref{GLRK_FEM}, we note that, for each $n = 0,\ldots,n_t-1$, the updated position $\boldsymbol{u}_L^{(n+1)}$ and velocity $\widehat{\boldsymbol{v}}_L^{(n+1)}$ are defined, respectively, by the first $N_L$ components and the last $N_L$ components of  $\boldsymbol{k}^{(n)}_{2L,i} \in \R^{2N_L}$, $i \in \{1,\ldots, q\}$. Motivated by this fact and the block structure of $\boldsymbol{B}_{2L}$ defined in~\eqref{matrix_midp}, for each $n = 0,\ldots,n_t-1$, we introduce the vectors $\boldsymbol{k}_{u,L}^{(n,q)}\in \R^{qN_L}$ and $\boldsymbol{k}_{v,L}^{(n,q)} \in  \R^{qN_L}$ representing, respectively, these first and last $N_L$ components of $\boldsymbol{k}^{(n)}_{2L,i} \in \R^{2N_L}$ for each $i \in \{1,\ldots, q\}$. Namely, we define
\begin{equation*}
\boldsymbol{k}_{u,L}^{(n,q)} := 
\begin{pmatrix}
\boldsymbol{k}_{2L,1}^{(n)}[1,\ldots,N_L]\\
\vdots\\
\boldsymbol{k}_{2L,q}^{(n)}[1,\ldots,N_L]\\
\end{pmatrix}
 \in \R^{qN_L}, \qquad
\boldsymbol{k}_{v,L}^{(n,q)} := 
\begin{pmatrix}
\boldsymbol{k}_{2L,1}^{(n)}[N_L+1,\ldots,2N_L]\\
\vdots\\
\boldsymbol{k}_{2L,q}^{(n)}[N_L+1,\ldots,2N_L]\\
\end{pmatrix}
\in \R^{qN_L}.
\end{equation*}
Then, denoting with $\boldsymbol{I}^{(q)}$ the identity matrix of size $q\times q$, and $\boldsymbol{I}_L$ the identity matrix of size $N_L\times N_L$, the GLRK--FEM method~\eqref{GLRK_FEM} reads as: For $n=0,1,\ldots,n_t-1$,
\begin{subequations}\label{u-v_GLRK-FEM}
\begin{align}
        \widehat{\boldsymbol{u}}^{(n+1)}_L &= \boldsymbol{C}_L \boldsymbol{u}^{(n)}_L + \tau \left(\boldsymbol{b}^{(q)T }\otimes \boldsymbol{I}_L \right) \left(\boldsymbol{I}^{(q)} \otimes \boldsymbol{C}_L\right) \boldsymbol{k}_{u,L}^{(n,q)},\label{GLRK-FEM_uhat}\\
        \boldsymbol{C}_L \widehat{\boldsymbol{A}}_L\boldsymbol{C}_L \boldsymbol{u}^{(n+1)}_L &= \boldsymbol{C}_L \widehat{\boldsymbol{A}}_L \widehat{\boldsymbol{u}}^{(n+1)}_L, \label{GLRK-FEM_u}\\
        \widehat{\boldsymbol{v}}^{(n+1)}_L &= \widehat{\boldsymbol{v}}^{(n)}_L + \tau \left(\boldsymbol{b}^{(q)  T} \otimes \boldsymbol{I}_L \right) \boldsymbol{k}_{v,L}^{(n,q)},
    \end{align}
\end{subequations}
with
\begin{subequations}\label{k1-k2_GLRK-FEM}
\begin{align}
    \left(\boldsymbol{I}^{(q)} \otimes \boldsymbol{C}_L\right) \boldsymbol{k}_{u,L}^{(n,q)} &= \tau \left(\boldsymbol{A}^{(q)} \otimes \boldsymbol{I}_L\right) \boldsymbol{k}_{v,L}^{(n,q)} + (1,\ldots,1)^T \otimes \widehat{\boldsymbol{v}}^{(n)}_L,\label{GLRK-FEM_k1}\\ 
    \begin{split}\label{GLRK-FEM_k2}
    \left(\boldsymbol{I}^{(q)}\otimes\boldsymbol{I}_L + \tau^2 \boldsymbol{A}^{(q)}\boldsymbol{A}^{(q)} \otimes \widehat{\boldsymbol{M}}_L^{-1}\widehat{\boldsymbol{A}}_L\right)\boldsymbol{k}^{(n,q)}_{v,L} &= -\tau\left( \boldsymbol{A}^{(q)} \otimes \widehat{\boldsymbol{M}}_L^{-1}\widehat{\boldsymbol{A}}_L\right)\left((1,\ldots,1)^T \otimes \widehat{\boldsymbol{v}}^{(n)}_L\right)\\ 
    &\qquad- \left( \boldsymbol{I}^{(q)} \otimes \widehat{\boldsymbol{M}}_L^{-1}\widehat{\boldsymbol{A}}_L\right) \left((1,\ldots,1)^T \otimes \boldsymbol{C}_L \boldsymbol{u}^{(n)}_L\right), 
    \end{split} 
    \end{align}
\end{subequations}
where $(1,\ldots,1)^T \in \R^q$. Let us remark that the linear systems we need to solve are just the (asymptotically) uniformly well-conditioned system~\eqref{GLRK-FEM_u}, see Table~\ref{tab::cond_BPX-stiff}, and linear system~\eqref{GLRK-FEM_k2}, which defines the intermediate slopes for velocity. GLRK--FEM iterations for position and velocity~\eqref{u-v_GLRK-FEM}-\eqref{k1-k2_GLRK-FEM} can be implemented as shown in Algorithm~\ref{alg:GLRK-FEM}. These iterations are the ones we are interested in compressing and solving in the tensor-structured format. 

\begin{table}[h!]
    \centering
    \begin{tabular}{cccccc}
        \toprule
        $L$ & $h^{-2}_L$ & $\kappa_2( \widehat{\boldsymbol{A}}_L )$ &  $\kappa_2(\boldsymbol{C}_L \widehat{\boldsymbol{A}}_L\boldsymbol{C}_L )$\\
        \midrule
        1 & 4.0000000e0 & 1.0000000e0 & 1.0000000e0\\
        2 & 1.6000000e1 & 5.8284271e0& 2.5000000e0 \\
        3 & 6.4000000e1 & 2.5274142e1 & 4.1381100e0\\
        4 & 2.5600000e2 & 1.0308687e2 &  5.5799800e0\\
        5 & 1.0240000e3 & 4.1434506e2 &  6.8098300e0\\
        6 & 4.0960000e3 & 1.6593797e3 & 7.8445700e0\\
        7 & 1.6384000e4 & 6.6395184e3 & 8.7181000e0\\
        8 & 6.5536000e4 & 2.6560073e4 & 9.4564400e0\\
        9 & 2.6214400e5 & 1.0624229e5 & 1.0081800e1\\
        10 & 1.0485760e6 & 4.2497118e5 & 1.0613000e1\\
    \bottomrule
    \end{tabular}
    \caption{Comparison between the spectral condition number of the stiffness matrix $\kappa_2( \widehat{\boldsymbol{A}}_L)$ and the preconditioned stiffness matrix $\kappa_2(\boldsymbol{C}_L \widehat{\boldsymbol{A}}_L\boldsymbol{C}_L)$. We denote the uniform mesh size in space with $h_L:=2^{-L}$.}
    \label{tab::cond_BPX-stiff}
\end{table}

\begin{algorithm}
\caption{GLRK--FEM~\eqref{u-v_GLRK-FEM}-\eqref{k1-k2_GLRK-FEM}}\label{alg:GLRK-FEM}
\hspace*{\algorithmicindent} \textbf{Input:} $\left(\boldsymbol{u}^{(n)}_L,\widehat{\boldsymbol{v}}^{(n)}_L\right)$\\
\hspace*{\algorithmicindent} \textbf{Output:} $\left(\boldsymbol{u}^{(n+1)}_L,\widehat{\boldsymbol{v}}^{(n+1)}_L\right)$
\begin{algorithmic}[1]
\State $\widehat{\boldsymbol{u}}^{(n)}_L \gets \boldsymbol{C}_L \boldsymbol{u}^{(n)}_L$
\State $\widehat{\boldsymbol{u}}^{(n,q)}_L \gets (1,\ldots,1)^T \otimes \widehat{\boldsymbol{u}}^{(n)}_L$, with $(1,\ldots,1)^T \in \R^q$
\State $\widehat{\boldsymbol{v}}^{(n,q)}_L \gets (1,\ldots,1)^T \otimes \widehat{\boldsymbol{v}}^{(n)}_L$, with $(1,\ldots,1)^T \in \R^q$
\State $\boldsymbol{k}^{(n,q)}_{v,L} \gets $ Solve \eqref{GLRK-FEM_k2} given $\left(\widehat{\boldsymbol{u}}^{(n,q)}_L, \widehat{\boldsymbol{v}}^{(n,q)}_L\right)$
\State $\widehat{\boldsymbol{v}}^{(n+1)}_L \gets \widehat{\boldsymbol{v}}^{(n)}_L + \tau \left(\boldsymbol{b}^{(q)  T} \otimes \boldsymbol{I}_L \right) \boldsymbol{k}_{v,L}^{(n,q)}$
\State $\left(\boldsymbol{I}^{(q)} \otimes \boldsymbol{C}_L\right) \boldsymbol{k}_{u,L}^{(n,q)} \gets \tau \left(\boldsymbol{A}^{(q)} \otimes \boldsymbol{I}_L\right) \boldsymbol{k}_{v,L}^{(n,q)} + \widehat{\boldsymbol{v}}^{(n,q)}_L $
\State $\widehat{\boldsymbol{u}}^{(n+1)}_L \gets \widehat{\boldsymbol{u}}^{(n)}_L + \tau \left(\boldsymbol{b}^{(q)T }\otimes \boldsymbol{I}_L \right) \left(\boldsymbol{I}^{(q)} \otimes \boldsymbol{C}_L\right) \boldsymbol{k}_{u,L}^{(n,q)}$
\State $\boldsymbol{u}^{(n+1)}_L \gets$ Solve~\eqref{GLRK-FEM_u} given $\widehat{\boldsymbol{u}}^{(n+1)}_L$
\end{algorithmic}
\end{algorithm}

\section{QTT decomposition}\label{sec:qtt}
In view of the model problem formulated in Section~\ref{sec:model problem} and discretized by the GLRK--FEM~\eqref{u-v_GLRK-FEM}-\eqref{k1-k2_GLRK-FEM}, we review the Quantized Tensor Train (QTT) approximation of vectors from $\mathbb{R}^{2^L}$ with $L \in \N$.

\begin{definition}[TT decomposition~\protect{\cite{O2011}}]
    Let $\boldsymbol{W} \in \mathbb{R}^{n_1 \times \ldots \times n_L}$ be a tensor with $L \in \mathbb{N}$ dimensions and mode sizes $n_1,\ldots,n_L \in \mathbb{N}$. Let $r_0,\ldots,r_{L} \in \mathbb{N}$ with $r_0=1$ and $r_L=1$, and, for $\ell \in \{1,\ldots,L\}$, let $\boldsymbol{W}_\ell$ be three-dimensional arrays of size $r_{\ell-1}\times n_\ell\times r_\ell$. The $L^{\text{th}}$-order tensor $\boldsymbol{W}$ is said to be represented in the Tensor Train (TT) decomposition with TT ranks $r_1,\ldots,r_{L-1}$, and cores, or factors, $\boldsymbol{W}_1,\ldots,\boldsymbol{W}_L$ if
    \begin{equation}\label{TT}
    \begin{split}
        \boldsymbol{W}[j_1,\ldots,j_L] = \sum_{\alpha_1 = 1}^{r_1} \cdots \sum_{\alpha_{L-1}=1}^{r_{L-1}}\boldsymbol{W}_1[\alpha_0,&j_1,\alpha_1]\boldsymbol{W}_2[\alpha_1,j_2,\alpha_2]\cdots\\ &\boldsymbol{W}_{L-1}[\alpha_{L-2},j_{L-1},\alpha_{L-1}]\boldsymbol{W}_L[\alpha_{L-1},j_L,\alpha_L],
    \end{split}
    \end{equation}
for all $j_k \in \{1,\ldots,n_k\}$ with $k \in \{1,\ldots,L\}$, where $\alpha_0\equiv1$ and $\alpha_L\equiv 1$ are dummy indices.
\end{definition}

A representation of the form~\eqref{TT} allows to parametrize a tensor $\boldsymbol{W} \in \mathbb{R}^{n_1 \times \ldots \times n_L}$ by 
\begin{equation}\label{TT-param}
\sum_{\ell=1}^L r_{\ell-1}n_\ell r_\ell \leq L \hspace{0.02cm} n_{\max}r^2_{\max}
\end{equation}
parameters from $\mathbb{R}$, where $n_{\max}:=\max\{n_1,\ldots,n_L\}$ and $r_{\max}:=\max\{r_1,\ldots,r_{L-1}\}$, instead of storing
$
    \displaystyle \Pi_{\ell=1}^L n_\ell \leq n^L_{\max}
$
entries. The number of parameters of the TT representation, formally linear in $L$, is mainly
governed by the TT ranks, which can in principle still
scale exponentially with respect to $L$. Therefore, proper rank bounds are necessary for actually avoiding the so-called \textit{curse of dimensionality} that arises from the entrywise storage of high-dimensional arrays. As shown by theoretical considerations and numerical evidence, the boundedness of the TT ranks 
is fulfilled in many applications; see, e.g.,~\cite{GKT2013}. 

In the present work, we use the TT decomposition to represent a given coefficient vector $\boldsymbol{w} \in \mathbb{R}^{2^L}$ in a low-rank tensor format. Our approach combines 
two ingredients: \emph{i)} the so-called \textit{quantization}, which is an isometric isomorphism $ \mathbb{R}^{2^L}\simeq \mathbb{R}^{2\times\cdots\times2} $ reshaping the $2^L$-component vector into an $L$-dimensional tensor; \emph{ii)} the TT
decomposition~\eqref{TT} for the resulting $L$-dimensional tensor. Precisely, for $2^L$-component vectors, we work with representations of the form
\begin{subequations}\label{QTT}
\begin{equation}
\begin{split}
    \boldsymbol{w}[j] = \sum_{\alpha_1 = 1}^{r_1} \cdots &\sum_{\alpha_{L-1}=1}^{r_{L-1}}\boldsymbol{W}_1[1,j_1,\alpha_1]\boldsymbol{W}_2[\alpha_1,j_2,\alpha_2]\cdots\\
    &\boldsymbol{W}_{L-1}[\alpha_{L-2},j_{L-1},\alpha_{L-1}]\boldsymbol{W}_L[\alpha_{L-1},j_L] \quad \text{for all } j \in \{1,\ldots,2^L\},
\end{split}
\end{equation}
where the identification of the \textit{physical} index $j$ with the \textit{virtual}, or mode, indices $j_1,\ldots,j_L \in \{1,2\}$ is provided by the lexicographic ordering
\begin{equation}\label{QTT-quant}
j := \overline{j_1,\ldots,j_L} := 1 + \sum_{\ell=1}^L 2^{L-\ell}(j_\ell-1) \in \{1,\ldots,2^L\} \quad \longleftrightarrow \quad (j_1,\ldots,j_L) \in \{1,2\}^{\times L}.
\end{equation}
\end{subequations}
For a vector $\boldsymbol{w}\in \mathbb{R}^{2^L}$, a representation of
the form~\eqref{QTT} is called a \textit{Quantized Tensor Train} (QTT) decomposition with \textit{cores}, or \textit{factors}, $\boldsymbol{W}_1,\ldots,\boldsymbol{W}_L$, and \textit{QTT ranks} $r_1,\ldots,r_{L-1}$. As for~\eqref{TT}, whose number of parameters~\eqref{TT-param} can drastically reduce the storage requirements associated with a given tensor $\boldsymbol{W} \in \mathbb{R}^{n_1 \times \ldots \times n_L}$, a representation of the form~\eqref{QTT} allows us to parametrize a vector $\boldsymbol{w} \in \mathbb{R}^{2^L}$ by
\begin{equation*}
    2r_1+\sum_{\ell=2}^{L-1}2r_{\ell-1}r_\ell + 2r_{\ell-1} \leq 2Lr^2_{\max}
\end{equation*}
parameters from $\mathbb{R}$, where $r_{\max}:=\max\{r_1,\ldots,r_{L-1}\}$, instead of storing $2^L$ entries. Notably, if the QTT ranks have a polynomial dependence on $L$, the number of parameters is still polynomial in $L$. 

One can extend the definition of QTT decomposition to matrices of size $2^L\times 2^L$ in such a way that matrix-by-vector multiplications can be efficiently performed in the format by acting on the $L$ cores independently. Specifically, for a given matrix $\boldsymbol{A} \in \mathbb{R}^{2^L \times 2^L}$, we work with representations of the form 
\begin{equation}\label{QTT-mat}
\begin{split}
    \boldsymbol{A}[i,j] = \sum_{\beta_1 = 1}^{p_1} \cdots &\sum_{\beta_{L-1}=1}^{p_{L-1}} \boldsymbol{A}_1[1,i_1,j_1,\beta_1]\boldsymbol{A}_2[\beta_1,i_2,j_2,\beta_2]\cdots\\
    &\boldsymbol{A}_{L-1}[\beta_{L-2},i_{L-1},j_{L-1},\beta_{L-1}]\boldsymbol{A}_L[\beta_{L-1},i_L,j_L] \quad \text{for all } i,j \in \{1,\ldots,2^L\},
\end{split}
\end{equation}
where the physical indices $i$ and $j$ are mapped into the virtual indices $i_1,...,i_L \in \{1,2\}$
and $j_1,...,j_L \in \{1,2\}$ under the same binary quantization of~\eqref{QTT-quant}. As for $2^L$-component vectors, for a matrix $\boldsymbol{A}\in \mathbb{R}^{2^L \times 2^L}$, a representation of
the form~\eqref{QTT-mat} is called a QTT decomposition with cores, or factors, $\boldsymbol{A}_1,\ldots,\boldsymbol{A}_L$, and ranks $p_1,\ldots,p_{L-1}$, and its total number of parameters is bounded by $4Lp^2_{\max}$, where $p_{\max}:=\max\{p_1,\ldots,p_{L-1}\}$.

\subsection{Useful insights into the QTT decomposition}\label{subsec:insights_QTT}
In this section, we address
linear algebra operations in the QTT format and highlight 
that they preserve the QTT structure. For brevity, several properties are stated for matrices, as they can be easily specialized for vectors, the latter being seen as one-column matrices. A comprehensive
presentation of the following properties is provided in~\cite{O2011}, which is our main reference for this section.

\begin{property}[Addition and multiplication by a scalar]\label{prop:add_QTT}
    Let $\boldsymbol{A} \in \R^{2^L \times 2^L}$ and $\boldsymbol{B} \in \R^{2^L \times 2^L}$ with QTT decompositions of ranks $r_1,\ldots,r_{L-1}$ and $p_1,\ldots,p_{L-1}$, respectively. Then, for all $\alpha,\beta \in \mathbb{R}$, the linear combination $\alpha \boldsymbol{A}+\beta \boldsymbol{B}$ admits a QTT decomposition of ranks $r_1+p_1,\ldots,r_{L-1}+p_{L-1}$. Furthermore, the addition requires virtually no operations, and the multiplication by a scalar just needs to scale one of the cores by the same scalar.
\end{property}

\begin{property}[Diagonalization of a vector]\label{Prop::QTT-diag}
Let $\boldsymbol{w} \in \mathbb{R}^{2^L}$ be a vector represented in a QTT decomposition with certain ranks. Then, its diagonalization, which is the square matrix whose diagonal coincides with $\boldsymbol{w}$, admits a QTT decomposition with the same ranks.
\end{property}

\begin{property}[Matrix-by-matrix product]\label{Prop::QTT-mat-by-vec}
    Let $\boldsymbol{A} \in \R^{2^L \times 2^L}$ and $\boldsymbol{B} \in \R^{2^L \times 2^L}$ with QTT decompositions of ranks $r_1,\ldots,r_{L-1}$ and $p_1,\ldots,p_{L-1}$, respectively. Then, their product $\boldsymbol{A}\boldsymbol{B}$ can be represented in the QTT format with ranks $r_1p_1,\ldots,r_{L-1}p_{L-1}$. The total number of operations $N_{\textrm{op}}$ for such a matrix-by-matrix product satisfies $N_{\textrm{op}} \lesssim Lr^2_{\max}p^2_{\max}$, where $r_{\max}:=\max\{r_1,\ldots,r_{L-1}\}$, and $p_{\max}:=\max\{p_1,\ldots,p_{L-1}\}$.
\end{property}

\begin{property}[Hadamard product]
Let $\boldsymbol{A} \in \R^{2^L \times 2^L}$ and $\boldsymbol{B} \in \R^{2^L \times 2^L}$, given, respectively, in QTT decompositions of ranks $r_1,\ldots,r_{L-1}$ and $p_1,\ldots,p_{L-1}$. Then, the Hadamard, or elementwise, product $\boldsymbol{A}\odot\boldsymbol{B}$ admits a QTT decomposition of ranks $r_1p_1,\ldots,r_{L-1}p_{L-1}$.
\end{property}

\begin{property}[Frobenius norm]
    Let $\boldsymbol{A} \in \R^{2^L \times 2^L}$, given in a QTT decomposition with ranks $r_1,\ldots,r_{L-1}$. Then, the Frobenius norm $\|\boldsymbol{A}\|_F = \sqrt{\sum_{i,j}\boldsymbol{A}[i,j]^2}$ can be computed with total number of operations $N_{\textrm{op}} \lesssim Lr_{\max}^3$, where $r_{\max}:=\max\{r_1,\ldots,r_{L-1}\}$.
\end{property}

As observed above, performing basic linear algebra operations with QTT tensors yields results in the QTT format, but with increased QTT ranks. Thus, to avoid rank growth, one aims to reduce the ranks while maintaining accuracy. A powerful method for this task is the Schmidt decomposition~\cite[Algo. 2]{O2011}, a stable and fast (number of operations $N_{\textrm{op}} \lesssim Lr^3_{\max}$) rounding procedure acting directly on the cores of the given QTT representation.

\subsection{Core notation and strong Kronecker product}
In this section, we recall useful notation and definitions to work in the QTT format, see, e.g.,~\cite{BK2020,KK2012} and references therein. 

\begin{definition}[Core]
    For each $\alpha \in \{1,\ldots,p\}$ and $\beta \in \{1,\ldots,q\}$, let $\boldsymbol{W}^{[\alpha,\beta]} \in \R^{m\times n}$. We call a core of rank $p\times q$ and mode size $m\times n$ the $4$-dimensional array $\boldsymbol{W} \in \R^{p \times m \times n \times q}$ defined as
    \begin{equation*}
        \boldsymbol{W}[\alpha,i,j,\beta] := \boldsymbol{W}^{[\alpha,\beta]}[i,j],  
    \end{equation*}
    for each $\alpha \in \{1,\ldots,p\}$, $i \in \{1,\ldots,m\}$, $j \in \{1,\ldots,n\}$, $\beta \in \{1,\ldots,q\}$. Additionally, for any core $\boldsymbol{W}$ of rank $p \times q$ and mode size $m \times n$, we refer to each matrix $\boldsymbol{W}^{[\alpha,\beta]} \in \R^{m \times n}$, with $\alpha \in \{1,\ldots,p\}$ and $\beta \in \{1,\ldots,q\}$, as block $(\alpha,\beta)$ of the core $\boldsymbol{W}$.
\end{definition}

Any core $\boldsymbol{W} \in \R^{p \times m \times n \times q}$ can be explicitly defined through its blocks by representing the core as a matrix whose entries are the blocks, i.e., using the notation
\begin{equation*}
    \boldsymbol{W} = 
    \begin{bmatrix}
        \boldsymbol{W}^{[1,1]} & \ldots & \boldsymbol{W}^{[1,q]}\\
        \vdots & \ddots & \vdots\\
        \boldsymbol{W}^{[p,1]} & \ldots & \boldsymbol{W}^{[p,q]}
    \end{bmatrix},
\end{equation*}
where square brackets are used to distinguish cores from ordinary matrices. This block notation can be combined with the notion of \emph{strong Kronecker product}, introduced in~\cite{dLS1994}, to get a compact notation for the QTT decomposition.

\begin{definition}[Strong Kronecker product]
    Let $\boldsymbol{W}_1$ and $\boldsymbol{W}_2$ be cores of ranks $r_0 \times r_1$ and $r_1 \times r_2$, composed, respectively, of blocks $\boldsymbol{W}^{[\alpha_0,\alpha_1]}_1 \in \R^{m_1 \times n_1}$ and $\boldsymbol{W}^{[\alpha_1,\alpha_2]}_2 \in \R^{m_2 \times n_2}$, $\alpha_k \in \{1,\ldots,r_k\}$ for $k \in \{0,1,2\}$. The strong Kronecker product $\boldsymbol{W}_1 \Join \boldsymbol{W}_2$  is the core of rank $r_0 \times r_2$ and mode size $m_1m_2 \times n_1 n_2$ consisting of blocks
    \begin{equation*}
        (\boldsymbol{W}_1 \Join \boldsymbol{W}_2)^{[\alpha_0,\alpha_1]} := \sum_{\alpha_1 =1}^{r_1} \boldsymbol{W}^{[\alpha_0,\alpha_1]}_1 \otimes \boldsymbol{W}^{[\alpha_1,\alpha_2]}_2, \qquad \alpha_k \in \{1,\ldots,r_k\}, \ k \in \{0,2\}.
    \end{equation*}
\end{definition}
This definition means that we define $\boldsymbol{W}_1 \Join \boldsymbol{W}_2$ as the usual matrix-by-matrix product of the block forms of $\boldsymbol{W}_1$ and $\boldsymbol{W}_2$, where their entries (blocks) are multiplied using their Kronecker products. For example, for two cores of rank $2 \times 2$, we have
$$
\begin{bmatrix}
    \boldsymbol{U}^{[1,1]} & \boldsymbol{U}^{[1,2]}\\
    \boldsymbol{U}^{[2,1]} & \boldsymbol{U}^{[2,2]}
\end{bmatrix}
\Join
\begin{bmatrix}
    \boldsymbol{V}^{[1,1]} & \boldsymbol{V}^{[1,2]}\\
    \boldsymbol{V}^{[2,1]} & \boldsymbol{V}^{[2,2]}
\end{bmatrix}
=
\begin{bmatrix}
    \boldsymbol{U}^{[1,1]} \otimes \boldsymbol{V}^{[1,1]} + \boldsymbol{U}^{[1,2]}\otimes \boldsymbol{V}^{[2,1]} & \boldsymbol{U}^{[1,1]}\otimes \boldsymbol{V}^{[1,2]} + \boldsymbol{U}^{[1,2]} \otimes \boldsymbol{V}^{[2,2]}\\
    \boldsymbol{U}^{[2,1]} \otimes \boldsymbol{V}^{[1,1]} + \boldsymbol{U}^{[2,2]} \otimes \boldsymbol{V}^{[2,1]} & \boldsymbol{U}^{[2,1]} \otimes \boldsymbol{V}^{[1,2]} + \boldsymbol{U}^{[2,2]}\otimes \boldsymbol{V}^{[2,2]}
\end{bmatrix}.
$$

As one can easily verify, decompositions~\eqref{QTT} and \eqref{QTT-mat} can be recast by using the strong Kronecker product as follows:
\begin{equation*}
    \boldsymbol{w} = [\boldsymbol{w}] = \boldsymbol{W}_1 \Join \cdots \Join \boldsymbol{W}_L, \qquad \boldsymbol{A} = [\boldsymbol{A}] = \boldsymbol{A}_1 \Join \cdots \Join \boldsymbol{A}_L,
\end{equation*}
where, with the first equalities, we indicate that any array of size $m \times n$ can be identified with a core of rank $1 \times 1 $ and mode size $m \times n$. 

\begin{remark}
Writing the QTT decomposition using the strong Kronecker product of cores highlights the deep connection between QTT representations and low-rank Kronecker decompositions of matrices, as we will see in Example~\ref{ex:stiffness_Dir}. This link is highly exploited in our computations to provide exact and explicit low-rank QTT decompositions of wave operators, where by explicit we mean that all the cores involved are provided in closed form.
\end{remark}

We conclude this section by recalling useful properties for matrices and vectors represented in the QTT format, which follow immediately using the definition of strong Kronecker product.

\begin{property}\label{QTT-properties}
    Let $\boldsymbol{A} := \boldsymbol{A}_1 \Join \cdots \Join \boldsymbol{A}_L$ and $\boldsymbol{B}:= \boldsymbol{B}_1 \Join \cdots \Join \boldsymbol{B}_L$ be two matrices of size $2^L \times 2^L$. Furthermore, let $\boldsymbol{w} := \boldsymbol{W}_1 \Join \cdots \Join \boldsymbol{W}_L$ be a $2^L$-component vector.
    \begin{itemize}
        \item The transpose $\boldsymbol{A}^T$ of $\boldsymbol{A}$ is given by the strong Kronecker product of the cores definining $\boldsymbol{A}$ with transposed blocks. Namely, $\boldsymbol{A}^T = \boldsymbol{A}^T_1 \Join \cdots \Join \boldsymbol{A}^T_L$, where, for each $\ell \in \{1,\ldots,L\}$, the entries of  $\boldsymbol{A}^T$ are defined as $\boldsymbol{A}^T_\ell[\alpha_{\ell-1},i_\ell,j_\ell,\alpha_\ell]:=\boldsymbol{A}_\ell[\alpha_{\ell-1},j_\ell,i_\ell,\alpha_\ell]$.
        \item For each $\alpha,\beta \in \mathbb{R}$, the linear combination $\alpha \boldsymbol{A}+\beta \boldsymbol{B}$ admits the following QTT decomposition:
        $$
        \alpha \boldsymbol{A}+\beta \boldsymbol{B} = 
        \begin{bmatrix}
            \boldsymbol{A}_1 & \boldsymbol{B}_1 
        \end{bmatrix}
        \Join
        \begin{bmatrix}
            \boldsymbol{A}_2 & \\ & \boldsymbol{B}_2 
        \end{bmatrix}
        \Join
        \cdots
        \Join
        \begin{bmatrix}
            \boldsymbol{A}_{\ell-1} & \\ & \boldsymbol{B}_{\ell-1} 
        \end{bmatrix}
        \Join
        \begin{bmatrix}
            \alpha\boldsymbol{A}_\ell\\
            \beta\boldsymbol{B}_\ell 
        \end{bmatrix}.
        $$
        \item The tensor product $\boldsymbol{A} \otimes \boldsymbol{B}$ of $\boldsymbol{A}$ and $\boldsymbol{B}$ is given by $\boldsymbol{A} \otimes \boldsymbol{B} =  \boldsymbol{A}_1 \Join \cdots \Join \boldsymbol{A}_L \Join \boldsymbol{B}_1 \Join \cdots \Join \boldsymbol{B}_L$. 
        \item The diagonalization $\mathcal{D}(\boldsymbol{w})$ of $\boldsymbol{w}$ admits a QTT decomposition with blocks given by the diagonalization of the blocks of $\boldsymbol{w}$. Namely, we have the representation $\mathcal{D}(\boldsymbol{w}) = \boldsymbol{D}_1 \Join \cdots \Join \boldsymbol{D}_L$, where, for $\ell \in \{1,\ldots,L\}$, each core $\boldsymbol{D}_\ell$ is defined through the blocks of $\boldsymbol{W}_\ell$ as $\boldsymbol{D}_\ell[\alpha_{\ell-1},i_\ell,j_\ell,\alpha_\ell]:=\delta_{i_\ell j_\ell}\boldsymbol{W}^{[\alpha_{\ell-1},\alpha_\ell]}[i_\ell]$, where $\delta_{i_\ell j_\ell}$ is the Kronecker delta.
        \item The matrix-by-matrix product $\boldsymbol{A} \boldsymbol{B}$ of $\boldsymbol{A}$ and $\boldsymbol{B}$ admits a QTT decomposition $\boldsymbol{A} \boldsymbol{B}= \boldsymbol{C}_1 \Join \cdots \Join \boldsymbol{C}_L $ with blocks $\boldsymbol{C}^{[\overline{\alpha_{\ell-1}\beta_{\ell-1}},\overline{\alpha_{\ell}\beta_{\ell}}]}_\ell := \boldsymbol{A}^{[\alpha_{\ell-1},\alpha_\ell]}_\ell \boldsymbol{B}^{[\beta_{\ell-1},\beta_\ell]}_\ell$, where the lexicographic ordering is used for the indices identifying the blocks of the product $\boldsymbol{A} \boldsymbol{B}$.
        \item The Hadamard product $\boldsymbol{A}\odot\boldsymbol{B}$ of $\boldsymbol{A}$ and $\boldsymbol{B}$ admits a QTT decomposition $\boldsymbol{A}\odot\boldsymbol{B} = \boldsymbol{C}_1 \Join \cdots \Join \boldsymbol{C}_L $ with blocks $\boldsymbol{C}^{[\overline{\alpha_{\ell-1}\beta_{\ell-1}},\overline{\alpha_{\ell}\beta_{\ell}}]}_\ell := \boldsymbol{A}^{[\alpha_{\ell-1},\alpha_\ell]}_\ell \odot \boldsymbol{B}^{[\beta_{\ell-1},\beta_\ell]}_\ell$, where the lexicographic ordering is used for the indices identifying the blocks of the elementwise product $\boldsymbol{A} \odot \boldsymbol{B}$.
    \end{itemize} 
\end{property}

\subsection{Quasi-optimality and computation of QTT decompositions}
The link between the QTT decomposition and the low-rank factorization of matrices allows us to compute a TT representation of any array using an SVD-type algorithm~\cite[Algo. 1]{O2011}. Additionally, this TT representation is proved to be quasi-optimal~\cite[Cor. 2.4]{O2011}. Let us recall this quasi-optimality in the context of matrices.

\begin{proposition}[Quasi-optimality]
    Let $\boldsymbol{A} \in \R^{2^L \times 2^L}$ be a given matrix and $r_1,\ldots,r_{L-1}$ be given rank bounds. Then, the best QTT approximation $\boldsymbol{A}^{\text{best}}$ to $\boldsymbol{A}$ in the Frobenius norm with QTT ranks bounded by $r_1,\ldots,r_{L-1}$ always exists. Additionally, the TT-SVD algorithm~\cite[Algo. 1]{O2011} computes a QTT approximation $\widetilde{\boldsymbol{A}}$ to $\boldsymbol{A}$, with QTT ranks bounded by $r_1,\ldots,r_{L-1}$, satisfying the quasi-optimality estimate
    \begin{equation*}
        \|\boldsymbol{A}-\widetilde{\boldsymbol{A}}\|_F \leq \sqrt{L-1} \|\boldsymbol{A}-\boldsymbol{A}^{\text{best}}\|_F.
    \end{equation*}
\end{proposition}

Even if it is always possible to provide QTT approximations of matrices and vectors by using the TT-SVD algorithm, in our discretization we do not rely on this strategy, whose computational cost in terms of both floating point operations and memory requirements is exponential in $L$. Instead, we provide \emph{offline} explicit QTT decompositions of the finite element operators by exploiting their low-rank Kronecker structures. As such, we avoid storing their entrywise representations, drastically reducing the storage requirements, and we avoid an exponential number of floating point operations. 

We end this section by showing two examples of computations of QTT decompositions. In what follows, we use the following $2 \times 2$ matrices as QTT blocks of the QTT decompositions of the matrices under consideration:
\begin{equation*}
    \boldsymbol{I} := 
    \begin{pmatrix}
        1 & 0 \\
        0 & 1
    \end{pmatrix},
    \qquad
    \boldsymbol{I}_1 := 
    \begin{pmatrix}
        1 & 0\\
        0 & 0
    \end{pmatrix},
    \quad 
    \boldsymbol{I}_2 :=
    \begin{pmatrix}
        0 & 0\\
        0 & 1
    \end{pmatrix},
    \quad 
    \boldsymbol{J} := 
    \begin{pmatrix}
        0 & 1\\
        0 & 0
    \end{pmatrix}.
\end{equation*}
Furthermore, for a given matrix $\boldsymbol{A}$, and $k \in \N$, we denote by $\boldsymbol{A}^{\otimes k}$ 
the $k^{\text{th}}$ tensor power of $\boldsymbol{A}$. Similarly, 
we use this convention for the strong Kronecker product operation $\Join$.

\begin{example}\label{ex::Laplacian}
    Let $L \in \N$ be such that $L\geq 2$. We consider the negative Laplacian with homogeneous Dirichlet boundary conditions on $(0,1)$, discretized by piecewise linear finite elements on a uniform grid with $2^L$ interior nodes. Using the standard hat basis functions, and omitting the scaling factor depending on the uniform mesh size, the discrete Laplacian 
    \begin{equation*}
    \boldsymbol{\Delta}_{L} := 
    \begin{pmatrix}
        2 & -1 & & \\
        -1 & 2 & \ddots &  \\
           & \ddots & \ddots & -1 \\
           & & -1 & 2 \\
    \end{pmatrix}
    \end{equation*}
    admits the following QTT decomposition of $L$ factors and ranks $3$:
    \begin{equation}\label{QTT_Laplacian}
        \boldsymbol{\Delta}_L = 
        \begin{bmatrix}
            \boldsymbol{I} & \boldsymbol{J} & \boldsymbol{J}^T
        \end{bmatrix}
        \Join
        \begin{bmatrix}
            \boldsymbol{I} & \boldsymbol{J} & \boldsymbol{J}^T \\
            & \boldsymbol{J}^T & \\
            & & \boldsymbol{J}
        \end{bmatrix}^{\Join (L-2)}
        \Join 
        \begin{bmatrix}
            2\boldsymbol{I} - \boldsymbol{J} - \boldsymbol{J}^T\\
            - \boldsymbol{J}^T\\
            - \boldsymbol{J}
        \end{bmatrix}.
    \end{equation}
This decomposition can be provided by exploiting the tridiagonal Toeplitz band structure of $\boldsymbol{\Delta}_L$, see~\cite[Cor. 3.2]{KK2012}.
\end{example}

\begin{example}\label{ex:stiffness_Dir}
    Let $L \in \N$ be such that $L\geq 2$. Let $\widehat{\boldsymbol{A}}_L \in \R^{2^L \times 2^L}$ be the padded stiffness matrix in the $L^2$-normalized basis functions~\eqref{basis}, i.e., 
    \begin{equation*}
        \widehat{\boldsymbol{A}}_L = 2^{2L}
        \begin{pmatrix}
            2 & -1 & 0 & \cdots & 0\\
            -1 & 2 & \ddots & &\vdots\\
            0 & \ddots & \ddots & -1 & \vdots \\
            \vdots & & -1 & 2 & 0\\
            0 & \cdots & \cdots & 0 &0
        \end{pmatrix}.
    \end{equation*}
    Then, $\widehat{\boldsymbol{A}}_L$ admits the following QTT decomposition of $L$ factors and ranks $4$:
    \begin{equation*}
        \widehat{\boldsymbol{A}}_L = 
        \begin{bmatrix}
            \boldsymbol{I}_1 & \boldsymbol{I}_2 & \boldsymbol{J} & \boldsymbol{J}^T
        \end{bmatrix}
        \Join
        \begin{bmatrix}
            \boldsymbol{I} & & \boldsymbol{J} & \boldsymbol{J}^T \\
            \boldsymbol{I}_1 & \boldsymbol{I}_2 & \boldsymbol{J} & \boldsymbol{J}^T\\
            & & \boldsymbol{J}^T & \\
            & & & \boldsymbol{J}
        \end{bmatrix}^{\Join (L-2)}
        \Join 
        \begin{bmatrix}
            2^{2L+1} \boldsymbol{I} - 2^{2L}\boldsymbol{J} - 2^{2L}\boldsymbol{J}^T\\
            2^{2L+1} \boldsymbol{I}\\
            - 2^{2L}\boldsymbol{J}^T\\
            - 2^{2L}\boldsymbol{J}
        \end{bmatrix}.
    \end{equation*}
    As we mentioned earlier, this decomposition can be provided by exploiting the low-rank Kronecker structure of the matrix under consideration. Let us explain the details. Recalling that rescaling a QTT decomposition is equivalent to rescaling one of its factor, it is sufficient to provide a QTT representation of the rescaled matrix:
    \begin{equation*}
        \widetilde{\boldsymbol{A}}_L = 2^{-2L} \widehat{\boldsymbol{A}}_L.
    \end{equation*}
    As in Example~\ref{ex::Laplacian}, for $k \in \N$, we denote with $\boldsymbol{\Delta}_{k} \in \R^{2^k \times 2^k}$ the negative Laplacian with homogeneous Dirichlet boundary conditions on $(0,1)$, discretized by piecewise linear finite elements on a uniform grid with $2^k$ interior nodes.
    For $k \in \{2,\ldots,L\}$, we have the recursive block structure
    \begin{equation*}
        \begin{aligned}
            \widetilde{\boldsymbol{A}}_k &=
        \left( \begin{array}{c | c}
           \boldsymbol{\Delta}_{k-1}  & -\boldsymbol{J}^{T \otimes(k-1)} \\
           \hline
            -\boldsymbol{J}^{\otimes(k-1)} & \widetilde{\boldsymbol{A}}_{k-1}
        \end{array} \right)\\
        &= \boldsymbol{I}_1 \otimes \boldsymbol{\Delta}_{k-1} + \boldsymbol{I}_2 \otimes \widetilde{\boldsymbol{A}}_{k-1} + \boldsymbol{J} \otimes \left(-\boldsymbol{J}^{T \otimes(k-1)} \right) + \boldsymbol{J}^T \otimes \left(-\boldsymbol{J}^{ \otimes(k-1)} \right),
        \end{aligned}
    \end{equation*}
which can be recast by using the strong Kronecker product as follows:
\begin{equation}\label{Ak}
    \widetilde{\boldsymbol{A}}_k = 
        \begin{bmatrix}
            \boldsymbol{I}_1 & \boldsymbol{I}_2 & \boldsymbol{J} & \boldsymbol{J}^T
        \end{bmatrix}
        \Join
        \begin{bmatrix}
            \boldsymbol{\Delta}_{k-1} \\
            \widetilde{\boldsymbol{A}}_{k-1}\\
            -\boldsymbol{J}^{T \otimes(k-1)}\\
            -\boldsymbol{J}^{ \otimes(k-1)}
        \end{bmatrix}.
\end{equation}
By using this representation with the equalities $-\boldsymbol{J}^{T \otimes k} = \boldsymbol{J}^T \otimes \left(-\boldsymbol{J}^{T \otimes(k-1)} \right)$, $-\boldsymbol{J}^{ \otimes k} = \boldsymbol{J} \otimes \left(-\boldsymbol{J}^{\otimes(k-1)} \right)$, and QTT decomposition~\eqref{QTT_Laplacian}, we obtain
\begin{equation}\label{dec}
    \begin{bmatrix}
            \boldsymbol{\Delta}_{k} \\
            \widetilde{\boldsymbol{A}}_{k}\\
            -\boldsymbol{J}^{T \otimes k}\\
            -\boldsymbol{J}^{ \otimes k}
        \end{bmatrix}
        = 
        \begin{bmatrix}
            \boldsymbol{I} & & \boldsymbol{J} & \boldsymbol{J}^T \\
            \boldsymbol{I}_1 & \boldsymbol{I}_2 & \boldsymbol{J} & \boldsymbol{J}^T\\
            & & \boldsymbol{J}^T & \\
            & & & \boldsymbol{J}
        \end{bmatrix}
        \Join
        \begin{bmatrix}
            \boldsymbol{\Delta}_{k-1} \\
            \widetilde{\boldsymbol{A}}_{k-1}\\
            -\boldsymbol{J}^{T \otimes(k-1)}\\
            -\boldsymbol{J}^{ \otimes(k-1)}
        \end{bmatrix}
\end{equation}
for $k \in \{2,\ldots,L\}$. Then, by taking $k = L$ in~\eqref{Ak}, and by applying decomposition~\eqref{dec} to itself recursively, we get the following representation:
\begin{equation*}
    \begin{aligned}
    \widehat{\boldsymbol{A}}_L &= 
    \begin{bmatrix}
            \boldsymbol{I}_1 & \boldsymbol{I}_2 & \boldsymbol{J} & \boldsymbol{J}^T
        \end{bmatrix}
        \Join
        \begin{bmatrix}
            \boldsymbol{\Delta}_{L-1} \\
            \widetilde{\boldsymbol{A}}_{L-1}\\
            -\boldsymbol{J}^{T \otimes(L-1)}\\
            -\boldsymbol{J}^{ \otimes(L-1)}
        \end{bmatrix}\\
        &=  \begin{bmatrix}
            \boldsymbol{I}_1 & \boldsymbol{I}_2 & \boldsymbol{J} & \boldsymbol{J}^T
        \end{bmatrix}
        \Join
        \begin{bmatrix}
            \boldsymbol{I} & & \boldsymbol{J} & \boldsymbol{J}^T \\
            \boldsymbol{I}_1 & \boldsymbol{I}_2 & \boldsymbol{J} & \boldsymbol{J}^T\\
            & & \boldsymbol{J}^T & \\
            & & & \boldsymbol{J}
        \end{bmatrix}
        \Join
        \begin{bmatrix}
            \boldsymbol{\Delta}_{L-2} \\
            \widetilde{\boldsymbol{A}}_{L-2}\\
            -\boldsymbol{J}^{T \otimes(L-2)}\\
            -\boldsymbol{J}^{ \otimes(L-2)}
        \end{bmatrix}\\
        &=
        \begin{bmatrix}
            \boldsymbol{I}_1 & \boldsymbol{I}_2 & \boldsymbol{J} & \boldsymbol{J}^T
        \end{bmatrix}
        \Join 
        \begin{bmatrix}
            \boldsymbol{I} & & \boldsymbol{J} & \boldsymbol{J}^T \\
            \boldsymbol{I}_1 & \boldsymbol{I}_2 & \boldsymbol{J} & \boldsymbol{J}^T\\
            & & \boldsymbol{J}^T & \\
            & & & \boldsymbol{J}
        \end{bmatrix}
        \Join
        \begin{bmatrix}
            \boldsymbol{I} & & \boldsymbol{J} & \boldsymbol{J}^T \\
            \boldsymbol{I}_1 & \boldsymbol{I}_2 & \boldsymbol{J} & \boldsymbol{J}^T\\
            & & \boldsymbol{J}^T & \\
            & & & \boldsymbol{J}
        \end{bmatrix}
        \begin{bmatrix}
            \boldsymbol{\Delta}_{L-3} \\
            \widetilde{\boldsymbol{A}}_{L-3}\\
            -\boldsymbol{J}^{T \otimes(L-3)}\\
            -\boldsymbol{J}^{ \otimes(L-3)}
        \end{bmatrix}\\
        &= \ldots = 
        \begin{bmatrix}
            \boldsymbol{I}_1 & \boldsymbol{I}_2 & \boldsymbol{J} & \boldsymbol{J}^T
        \end{bmatrix}
        \Join 
        \begin{bmatrix}
            \boldsymbol{I} & & \boldsymbol{J} & \boldsymbol{J}^T \\
            \boldsymbol{I}_1 & \boldsymbol{I}_2 & \boldsymbol{J} & \boldsymbol{J}^T\\
            & & \boldsymbol{J}^T & \\
            & & & \boldsymbol{J}
        \end{bmatrix}^{\Join (L-2)}
        \Join
        \begin{bmatrix}
            \boldsymbol{\Delta}_{1} \\
            \widetilde{\boldsymbol{A}}_{1}\\
            -\boldsymbol{J}^{T }\\
            -\boldsymbol{J}
        \end{bmatrix}.
    \end{aligned}
\end{equation*}
We can conclude observing that $\boldsymbol{\Delta}_{1} = 2\boldsymbol{I}-\boldsymbol{J}-\boldsymbol{J}^T$, and $\widetilde{\boldsymbol{A}}_{1} = 2 \boldsymbol{I}_1$.
\end{example}

\section{Solving the GLRK--FEM in the QTT format }\label{sec:tt in space}
In this section, we describe how to compress and solve the iterations presented in Algorithm~\ref{alg:GLRK-FEM} using the QTT decomposition. 

To work in the QTT format, from now on we properly pad with zeros all the matrices and vectors associated with the space discretization so that their sizes are powers of $2$. This means that, letting $L \in \N$ be the number of refinement levels in space, and recalling that the effective number of degrees of freedom is $N_L = 2^L-1$, we now consider $N_L+1 = 2^L$ entries for both the position and velocity variables and force their last entries to be equal to zero.

\subsection{Tensor-structured approximation of the initial data}\label{subsec:initdata}
Let $L \in \N$ be the fixed number of levels in space. Let $\{\widehat{\varphi}_{L,j}\}_{j=1}^{2^L}$ be the $L^2$-normalized nodal basis of hat functions of the space $\{w_L \in \mathcal{S}(\mathcal{T}_L): w_L(0) = 0\}$. These basis functions coincide with the ones defined in~\eqref{basis}, with the additional hat function $\widehat{\varphi}_{L,L}$ corresponding to the endpoint $x_{L,L} = 1$. Furthermore, let us denote with $\widehat{\boldsymbol{f}}_L \in \R^{2^L}$ and $\widehat{\boldsymbol{g}}_L \in \R^{2^L}$ the vectors whose entries are defined by the following inner products:
\begin{equation}\label{RHS}
    \widehat{\boldsymbol{f}}_L[j] := \int_{\Omega} u'_0(x) \hspace{0.05cm} \widehat{\varphi}'_{L,j}(x) \ dx, \qquad \widehat{\boldsymbol{g}}_L[j] := \int_{\Omega} v_0(x) \hspace{0.05cm} \widehat{\varphi}_{L,j}(x) \ dx, 
\end{equation}
for $j \in \{1,\ldots,2^L\}$. As we have discussed in Section~\ref{sec:semidiscrete} and Remark~\ref{FP}, the discrete initial data for position and velocity are defined, respectively, by the elliptic projection~\eqref{ell_proj} and $L^2$ projection~\eqref{l2_proj} of the continuous initial data. As such, the vector representation $\boldsymbol{u}_{L,0} \in \R^{2^L}$ of the discrete initial position in the BPX basis~\eqref{bpx_basis} is QTT-compressed by solving the following uniformly well-conditioned system in the QTT format:
\begin{equation}\label{lin_syst_u0}
    \begin{pmatrix}
        \boldsymbol{C}_L\widehat{\boldsymbol{A}}_L\boldsymbol{C}_L & \\
        & 0
    \end{pmatrix} 
    \boldsymbol{u}_{L,0} = 
    \begin{pmatrix}
        \boldsymbol{C}_L & \\
        & 0
    \end{pmatrix}
    \widehat{\boldsymbol{f}}_L.
\end{equation}
As in~\cite{BK2020}, exact and explicit QTT decompositions, with QTT ranks independent of $L$, can be derived for the system matrix and the padded BPX preconditioner. The vector representation $\widehat{\boldsymbol{v}}_{L,0} \in \R^{2^L}$ of the discrete initial velocity in the $L^2$-normalized basis~\eqref{basis} is QTT-compressed by computing the following matrix-by-vector product in the QTT format:
\begin{equation}\label{linear_syst_v0}
    \widehat{\boldsymbol{v}}_{L,0} =
    \begin{pmatrix}
        \widehat{\boldsymbol{M}}^{-1}_L & \\
        & 0 
    \end{pmatrix}
    \widehat{\boldsymbol{g}}_L,
\end{equation}
where an exact and explicit QTT decomposition of the padded inverse mass matrix with QTT ranks up to $14$ can be obtained using the Sherman--Morrison--Woodbury formula~\cite[Pag. 65]{GVL2013} similarly to as in~\cite{KK2012}. Therefore, the remaining ingredients for providing low-rank QTT approximations of $\boldsymbol{u}_{L,0}$ and $\widehat{\boldsymbol{v}}_{L,0}$ are, respectively, low-rank QTT approximations of the inner products $\widehat{\boldsymbol{f}}_L$ and $\widehat{\boldsymbol{g}}_L$ defined in~\eqref{RHS}, to the construction of which the remainder of this section is devoted. 

As a first step, we provide low-rank approximations of the weak derivative of the initial position $u'_0 \in L^2(\Omega)$, and the initial velocity $v_0 \in L^2(\Omega)$. Let $\widetilde{L} \in \N$ be a fixed number of space levels, and let $\mathcal{T}_{\widetilde{L}}$ be the corresponding mesh on $\Omega$. Let us introduce the space $\mathcal{S}_0(\mathcal{T}_{\widetilde{L}})$ of functions that are piecewise constants on $\mathcal{T}_{\widetilde{L}}$, i.e., 
\begin{equation}\label{piecewise_const}
    \mathcal{S}_0(\mathcal{T}_{\widetilde{L}}) := \{w_{\widetilde{L}} \in L^2(\Omega): w_{\widetilde{L}|K_{\widetilde{L},i}} \in \mathbb{P}_0(K_{\widetilde{L},i}) \ \forall K_{\widetilde{L},i} \in \mathcal{T}_{\widetilde{L}} \}.
\end{equation}
Furthermore, let $\Pi_{\widetilde{L},0}: L^2(\Omega)\rightarrow \mathcal{S}_0(\mathcal{T}_{\widetilde{L}})  $ be the $L^2$ projection through which we approximate $u'_0$ and $v_0$. For a given $w \in L^2(\Omega)$, denoting with $\boldsymbol{\Pi}^{w}_{\widetilde{L},0} \in \R^{2^{\widetilde{L}}}$ the vector representing its $L^2$ projection $\Pi_{\widetilde{L},0}(w) \in \mathcal{S}_0(\mathcal{T}_{\widetilde{L}})$ in the characteristic basis, 
    it holds that \begin{equation}\label{L2_proj_const}
        \boldsymbol{\Pi}^{w}_{\widetilde{L},0}[i] := 2^{\widetilde{L}}\int_{K_{\widetilde{L},i}} w(x) \ dx, \qquad i \in \{1,\ldots,2^{\widetilde{L}}\}.
    \end{equation}
     The following result details sufficient conditions for having low-rank QTT representations of $\boldsymbol{\Pi}^{u'_0}_{\widetilde{L},0} \in \R^{2^{\widetilde{L}}}$ and $\boldsymbol{\Pi}^{v_0}_{\widetilde{L},0} \in \R^{2^{\widetilde{L}}}$.
\begin{theorem}\label{th::low_rank_functions}
    For a fixed number of levels $\widetilde{L} \in \N$, and $w \in L^2(\Omega)$, let $\boldsymbol{\Pi}^{w}_{\widetilde{L},0} \in \R^{2^{\widetilde{L}}}$ be the vector defined in~\eqref{L2_proj_const}. Recall that~$\Omega:=(0,1)$, fix the reference domain $D :=(-1,1)$, and define the affine functions $\phi_0: D \rightarrow \Omega$ and $\phi_{\ell,i_\ell}: D \rightarrow D$, with $i_\ell \in \{1,2\}$ and $\ell \in \{1,\ldots,\widetilde{L}\}$, such that
    \begin{equation*}
        \phi_0(s) := \frac{s+1}{2}, \quad \phi_{\ell,1}(s) := \frac{s-1}{2}, \quad \phi_{\ell,2}(s) := \frac{s+1}{2}, \qquad s \in D.
    \end{equation*}
    Let the following assumptions be satisfied:
    \begin{itemize}
        \item [(i)] There exists a finite-dimensional space $\mathcal{W}_0 \subset L^2(D)$ such that $w \circ \phi_0 \in \mathcal{W}_0$.
        \item [(ii)] There exists a sequence of finite-dimensional spaces $\mathcal{W}_1,\ldots,\mathcal{W}_{\widetilde{L}-1}$ such that, for each $\ell \in \{1,\ldots,\widetilde{L}-1\}$, and $i_\ell \in \{1,2\}$, the following inclusion holds true:
        $$
        \{z \circ \phi_{\ell,i_\ell}: z \in \mathcal{W}_{\ell-1} \} \subseteq \mathcal{W}_{\ell}.
        $$
    \end{itemize}
    Then, $\boldsymbol{\Pi}^{w}_{\widetilde{L},0}$ admits a QTT decomposition with $\widetilde{L}+1$ cores, and each QTT rank $r_\ell$ bounded by the dimension of $\mathcal{W}_{\ell-1}$, for $\ell \in \{1,\ldots,\widetilde{L}\}$.
\end{theorem}
\begin{proof}
    For notation convenience, for each $\ell \in \{1,\ldots,\widetilde{L}\}$, let us define the affine transformation
    \begin{equation}\label{aff_transf}
        \Phi_{\ell,i} := \phi_0 \circ \phi_{1,i_1} \circ \cdots \circ \phi_{\ell,i_\ell},
    \end{equation}
    for $i_1,\ldots,i_\ell \in \{1,2\}$, with the lexicographic ordering:
    \begin{equation*}
        i := \overline{i_1,\ldots,i_\ell} := 1 + \sum_{k=1}^{\ell} 2^{\ell-k}(i_k-1) \in \{1,\ldots,2^\ell\}.
    \end{equation*}
    Let $\pi_{\widetilde{L},0}: L^2(D) \rightarrow \mathbb{P}_0(D)$ be the $L^2$ projection onto the constant functions on $D$. A straightforward computation shows that 
    \begin{equation}\label{proj_zoom}
        \boldsymbol{\Pi}^{w}_{\widetilde{L},0}[i] = \pi_{\widetilde{L},0}(w \circ \Phi_{\widetilde{L},i}),
    \end{equation}
    for each $i \in \{1,\ldots,2^{\widetilde{L}}\}$. Thus, from now on, we work on~$\pi_{\widetilde{L},0}(w \circ \Phi_{\widetilde{L},i})$. For each $\ell \in \{0,\ldots,\widetilde{L}-1\}$, let $\{\psi_{\ell,\alpha_\ell}\}_{\alpha_\ell = 1}^{r_\ell}$ denote a basis of $\mathcal{W}_\ell$. As a result of \emph{(i)}, there exists an array $\boldsymbol{W}_0 \in \R^{1\times 1 \times r_0}$ such that:
    \begin{equation*}
        w \circ \phi_0 = \sum_{\alpha_0 = 1}^{r_0} \boldsymbol{W}_0[1,1,\alpha_0] \psi_{0,\alpha_0}.
    \end{equation*}
    Furthermore, from \emph{(ii)}, we know that, for each $\ell \in \{1,\ldots,\widetilde{L}-1\}$, there exists an array $\boldsymbol{W}_\ell \in \R^{r_{\ell-1} \times 2 \times r_\ell}$ such that
    \begin{equation*}
        \psi_{\ell-1,\alpha_{\ell-1}} \circ \phi_{\ell,i_\ell} = \sum_{\alpha_\ell = 1 }^{r_\ell} \boldsymbol{W}_\ell[\alpha_{\ell-1},i_\ell,\alpha_\ell] \psi_{\ell,\alpha_\ell},
    \end{equation*}
    for each $\alpha_{\ell-1} \in \{1,\ldots,r_{\ell-1}\}$, and each $i_\ell \in \{1,2\}$. Therefore, proceeding recursively, we get the following representation:
    \begin{equation*}
        w \circ \Phi_{\widetilde{L}-1,i_{\widetilde{L}-1}} = \sum_{\alpha_0 = 1}^{r_0} \sum_{\alpha_1 = 1}^{r_1} \ldots \sum_{\alpha_{\widetilde{L}-1}=1}^{r_{\widetilde{L}-1}} \boldsymbol{W}_0[1,1,\alpha_0] \boldsymbol{W}_1[\alpha_0,i_1,\alpha_1] \cdots \boldsymbol{W}_{\widetilde{L}-1}[\alpha_{\widetilde{L}-2},i_{\widetilde{L}-1},\alpha_{\widetilde{L}-1}] \psi_{\widetilde{L}-1,\alpha_{\widetilde{L}-1}}.
    \end{equation*}
    Using the definition of the projection $\pi_{\widetilde{L},0}$, there exists an array $\boldsymbol{W}_{\widetilde{L}} \in \R^{r_{\widetilde{L}-1} \times 2}$ such that
    \begin{equation*}
        \pi_{\widetilde{L},0}(\psi_{\widetilde{L}-1,\alpha_{\widetilde{L}-1}} \circ \phi_{\widetilde{L},i_{\widetilde{L}}}) = \boldsymbol{W}_{\widetilde{L}}[\alpha_{\widetilde{L}-1},i_{\widetilde{L}}]
    \end{equation*}
    for each $\alpha_{\widetilde{L}-1} \in \{1,\ldots,r_{\widetilde{L}-1}\}$ and $i_{\widetilde{L}} \in \{1,2\}$. As a result of this fact, and the linearity of $\pi_{\widetilde{L},0}$, we get the following representation:
    \begin{equation*}
        \pi_{\widetilde{L},0}(w \circ \Phi_{\widetilde{L},i}) = \sum_{\alpha_0 = 1}^{r_0} \sum_{\alpha_1 = 1}^{r_1} \ldots \sum_{\alpha_{\widetilde{L}-1}=1}^{r_{\widetilde{L}-1}} \boldsymbol{W}_0[1,1,\alpha_0] \boldsymbol{W}_1[\alpha_0,i_1,\alpha_1] \cdots \boldsymbol{W}_{\widetilde{L}-1}[\alpha_{\widetilde{L}-2},i_{\widetilde{L}-1},\alpha_{\widetilde{L}-1}] \boldsymbol{W}_{\widetilde{L}}[\alpha_{\widetilde{L}-1},i_{\widetilde{L}}].
    \end{equation*}
    Recalling equality~\eqref{proj_zoom}, the statement is proved.
\end{proof}

\begin{example}\label{example:low_rank_functions}
The following two families of functions satisfy the assumptions of Theorem~\ref{th::low_rank_functions}.
    \begin{itemize}
        \item [(1)] Let $n \in \N$. Each function $w \in \mathbb{P}_n(\Omega)$ satisfies Theorem~\ref{th::low_rank_functions} with QTT ranks bounded by $n+1$.
        \item [(2)] Let $a,b,\kappa \in \R$. Each trigonometric function given by $w(x) = a\cos(2\pi\kappa x) + b\sin(2\pi \kappa x) $, with $x \in \Omega$, satisfies Theorem~\ref{th::low_rank_functions} with QTT ranks bounded by $2$.
    \end{itemize}
\end{example}

In what follows, we show how to exploit Theorem~\ref{th::low_rank_functions} to provide low-rank QTT representations of the inner products in~\eqref{RHS}. 

For $\widetilde{L} \in \N$ such that $\widetilde{L} \geq L$, let us introduce the following matrices:
\begin{subequations}\label{zoom_basis_full}
    \begin{align}
    \widehat{\boldsymbol{D}'}_{\widetilde{L}} &:= 2^{\frac{3}{2}\widetilde{L}}
    \begin{pmatrix}
        1 & & & \\
        -1 & \ddots & & \\
           & \ddots & \ddots & \\
           &        & -1     & 1
    \end{pmatrix} \in \R^{2^{\widetilde{L}} \times 2^{\widetilde{L}}},\label{zoom_basis_full_D'}\\
    \widehat{\boldsymbol{D}}_{\widetilde{L}} &:= 2^{\frac{1}{2}\widetilde{L}-1} 
    \begin{pmatrix}
        1 & & & \\
        1 & \ddots & & \\
           & \ddots & \ddots & \\
           &        & 1     & 1
    \end{pmatrix} 
    \otimes 
    \begin{pmatrix}
        1\\
        0
    \end{pmatrix}
    + 2^{\frac{1}{2}\widetilde{L}-1} 
    \begin{pmatrix}
        1 & & & \\
        -1 & \ddots & & \\
           & \ddots & \ddots & \\
           &        & -1     & 1
    \end{pmatrix} 
    \otimes
    \begin{pmatrix}
        0\\
        1
    \end{pmatrix} \in \R^{2^{\widetilde{L}+1} \times 2^{\widetilde{L}}},\label{zoom_basis_full_D}
    \end{align}
\end{subequations}
which admit the following QTT decompositions:
\begin{subequations}\label{zoom_basis}
    \begin{align}
        \widehat{\boldsymbol{D}'}_{\widetilde{L}} &= 2^{\frac{3}{2}\widetilde{L}} 
        \begin{bmatrix}
            1 & 0 
        \end{bmatrix}
        \Join 
        \begin{bmatrix}
            \boldsymbol{I} & \boldsymbol{J}^T \\
            & \boldsymbol{J}
        \end{bmatrix}^{\Join \widetilde{L}}
        \Join
        \begin{bmatrix}
            1\\
            -1
        \end{bmatrix},\label{zoom_basis_D'}\\
        \widehat{\boldsymbol{D}}_{\widetilde{L}} &= 2^{\frac{1}{2}\widetilde{L}-1}
        \begin{bmatrix}
            1 & 0
        \end{bmatrix}
        \Join 
        \begin{bmatrix}
            \boldsymbol{I} & \boldsymbol{J}^T\\
            & \boldsymbol{J}
        \end{bmatrix}^{\Join \widetilde{L}}
        \Join
        \begin{bmatrix}
            \begin{pmatrix}
                1\\
                1
            \end{pmatrix}
            \\
            \begin{pmatrix}
                1\\
                -1
            \end{pmatrix}
        \end{bmatrix}\label{zoom_basis_D}.
    \end{align}
\end{subequations}
As a simple computation shows, the following equalities are satisfied:
\begin{equation}\label{zoom_basis_2}
    \widehat{\varphi}'_{\widetilde{L},j} \circ \Phi_{\widetilde{L},i} = \widehat{\boldsymbol{D}'}_{\widetilde{L}}[i,j], \qquad \widehat{\varphi}_{\widetilde{L},j} \circ \Phi_{\widetilde{L},i} = \sum_{\alpha=1}^2 \widehat{\boldsymbol{D}}_{\widetilde{L}}[\overline{i\alpha},j] \psi_{\widetilde{L},\alpha},
\end{equation}
for all $ i,j \in \{1,\ldots,2^{\widetilde{L}}\}$, where, for each $s \in D$, $\psi_{\widetilde{L},1}(s) := 1$ and $\psi_{\widetilde{L},2}(s) := s$, and $\Phi_{\widetilde{L},i}$ is the affine transformation defined in~\eqref{aff_transf}. Furthermore, let us consider the extension matrix $\widehat{\boldsymbol{P}}_{L,\widetilde{L} } \in \R^{2^{\widetilde{L}}\times 2^{L}}$ from $\{w_L \in \mathcal{S}(\mathcal{T}_L): w_L(0) = 0\}$ to $\{w_{\widetilde{L}} \in \mathcal{S}(\mathcal{T}_{\widetilde{L}}): w_{\widetilde{L}}(0) = 0\}$ with respect to the basis $\{\widehat{\varphi}_j\}_{j=1}^{2^L}$. This matrix reads as
\begin{equation*}
    \widehat{\boldsymbol{P}}_{L,\widetilde{L} } := 2^{\frac{1}{2}(L-\widetilde{L})}
    \begin{pmatrix}
        1 & & & \\
          & \ddots & & \\
          &        & \ddots &\\
          &        &         & 1\\
    \end{pmatrix}
    \otimes
    \begin{pmatrix}
        2^{L-\widetilde{L}}\\
        2^{1 + L-\widetilde{L}}\\
        \vdots\\
        1 - 2^{L-\widetilde{L}}\\
        1
    \end{pmatrix}
    + 2^{\frac{1}{2}(L-\widetilde{L})}
    \begin{pmatrix}
        0 & & & \\
        1 & \ddots & & \\
          & \ddots & \ddots & \\
          & & 1 & 0\\
    \end{pmatrix}
    \otimes
    \begin{pmatrix}
        1 - 2^{L-\widetilde{L}}\\
        1 - 2^{1 + L-\widetilde{L}}\\
        \vdots\\
        2^{L-\widetilde{L}}\\
        0\\
    \end{pmatrix},
\end{equation*}
and the following QTT decomposition holds:
\begin{equation}\label{QTT-ext_mat}
    \widehat{\boldsymbol{P}}_{L,\widetilde{L} }=2^{\frac{1}{2}(L-\widetilde{L})}
    \begin{bmatrix}
        1 & 0 
    \end{bmatrix}
    \Join
    \begin{bmatrix}
            \boldsymbol{I} & \boldsymbol{J}^T \\
            & \boldsymbol{J}
    \end{bmatrix}^{\Join L}
    \Join
    \begin{bmatrix}
        \begin{pmatrix}
            \frac{1}{2} \\
            1
        \end{pmatrix}
        & 
        \begin{pmatrix}
            0  \\
            \frac{1}{2}
        \end{pmatrix}\\
        \begin{pmatrix}
            \frac{1}{2} \\
            0
        \end{pmatrix}
        &
        \begin{pmatrix}
            1  \\
            \frac{1}{2}
        \end{pmatrix}
    \end{bmatrix}^{\Join (\widetilde{L}-L)}
    \Join
    \begin{bmatrix}
        1\\
        0
    \end{bmatrix}.
\end{equation}
Representations~\eqref{zoom_basis}, \eqref{zoom_basis_2}, and \eqref{QTT-ext_mat} 
guarantee that the inner products in~\eqref{RHS} admit low-rank QTT approximations if the initial data $u'_0$ and $v_0$ satisfy the assumptions of Theorem~\ref{th::low_rank_functions}, as established in the following corollary. 

\begin{corollary}
    Let $L \in \N$ be the refinement level in space, and let $\widetilde{L} \in \N$ be such that $\widetilde{L} \geq L$. Let the initial data $u'_0 \in L^2(\Omega)$ and $v_0 \in L^2(\Omega)$ of wave problem~\eqref{eq_wave_ham} satisfy Theorem~\ref{th::low_rank_functions}. Then, the right-hand sides $\widehat{\boldsymbol{f}}_L$ and $\widehat{\boldsymbol{g}}_L$ defined in~\eqref{RHS} admit QTT representations with QTT ranks independent of $L$ and $\widetilde{L}$.
\end{corollary} 
\begin{proof}
     Let $\Pi_{\widetilde{L},0}: L^2(\Omega)\rightarrow \mathcal{S}_0(\mathcal{T}_{\widetilde{L}})$ be the $L^2$ projection onto the space $\mathcal{S}_0(\mathcal{T}_{\widetilde{L}})$ of piecewise constant functions defined in~\eqref{piecewise_const}. As above, we denote with $\boldsymbol{\Pi}^{u'_0}_{\widetilde{L},0} \in \R^{2^{\widetilde{L}}}$ and $\boldsymbol{\Pi}^{v_0}_{\widetilde{L},0} \in \R^{2^{\widetilde{L}}}$, respectively, the vector
     representations of $\Pi_{\widetilde{L},0}(u'_0) \in \mathcal{S}_0(\mathcal{T}_{\widetilde{L}})$ and $\Pi_{\widetilde{L},0}(v_0) \in \mathcal{S}_0(\mathcal{T}_{\widetilde{L}})$ in the characteristic basis. Furthermore, for each $i \in \{1,\ldots,2^{\widetilde{L}}\}$, let $\Phi_{\widetilde{L},i}$ be the affine transformation defined in~\eqref{aff_transf}, and fix the reference domain $D =(-1,1)$. For each $j \in \{1,\ldots,2^L\}$, the following relations are satisfied:
     \begin{equation}\label{approx_f}
         \begin{aligned}
             \widehat{\boldsymbol{f}}_L[j] &:= \int_{\Omega} u'_0(x) \hspace{0.05cm} \widehat{\varphi}'_{L,j}(x) \ dx = \sum_{k=1}^{2^{\widetilde{L}}} \widehat{\boldsymbol{P}}_{L,\widetilde{L} }[k,j] \int_{\Omega} u'_0(x) \hspace{0.05cm} \widehat{\varphi}'_{\widetilde{L},k}(x) \ dx\\
             &=2^{-(\widetilde{L}+1)}\sum_{k=1}^{2^{\widetilde{L}}} \widehat{\boldsymbol{P}}_{L,\widetilde{L} }[k,j] \sum_{i=1}^{2^{\widetilde{L}}} \int_{D}(u'_0 \circ \Phi_{\widetilde{L},i})(s)  \hspace{0.05cm} (\widehat{\varphi}'_{\widetilde{L},k} \circ \Phi_{\widetilde{L},i})(s) \ ds \\
             &=2^{-(\widetilde{L}+1)} \sum_{k=1}^{2^{\widetilde{L}}} \sum_{i=1}^{2^{\widetilde{L}}} \widehat{\boldsymbol{P}}_{L,\widetilde{L} }[k,j] \widehat{\boldsymbol{D}'}_{\widetilde{L}}[i,k] \int_{D} (u'_0 \circ \Phi_{\widetilde{L},i})(s) \ ds\\
             &= 2^{-\widetilde{L}} \sum_{k=1}^{2^{\widetilde{L}}} \sum_{i=1}^{2^{\widetilde{L}}} \widehat{\boldsymbol{P}}_{L,\widetilde{L} }[k,j] \widehat{\boldsymbol{D}'}_{\widetilde{L}}[i,k]\boldsymbol{\Pi}^{u'_0}_{\widetilde{L},0}[i].
         \end{aligned}
     \end{equation}
     Namely, $\widehat{\boldsymbol{f}}_L$ is represented by the matrix-by-vector product $2^{-\widetilde{L}}(\widehat{\boldsymbol{D}'}_{\widetilde{L}}\widehat{\boldsymbol{P}}_{L,\widetilde{L}})^T\boldsymbol{\Pi}^{u'_0}_{\widetilde{L},0}$, which can be efficiently perfomed in the QTT format by using Properties~\ref{QTT-properties}, Theorem~\ref{th::low_rank_functions}, and representations~\eqref{zoom_basis_D'}, \eqref{QTT-ext_mat}. 
     
     The square matrix $\widehat{\boldsymbol{D}}^{(\alpha=1)}_{\widetilde{L}}$ defined by the odd rows of the matrix $\widehat{\boldsymbol{D}}_{\widetilde{L}}$ defined in~\eqref{zoom_basis_full_D} admits the QTT decomposition
     \begin{equation}\label{odd_D}             \widehat{\boldsymbol{D}}^{(\alpha=1)}_{\widetilde{L}} =  2^{\frac{1}{2}\widetilde{L}-1}
        \begin{bmatrix}
            1 & 0 
        \end{bmatrix}
        \Join 
        \begin{bmatrix}
            \boldsymbol{I} & \boldsymbol{J}^T \\
            & \boldsymbol{J}
        \end{bmatrix}^{\Join \widetilde{L}}
        \Join
        \begin{bmatrix}
            1\\
            1
        \end{bmatrix}.
     \end{equation}
     For~$\widehat{\boldsymbol{g}}_L$, with similar calculations as in~\eqref{approx_f} we obtain
     \begin{equation}\label{approx_g}
         \begin{aligned}
             \widehat{\boldsymbol{g}}_L[j] &:= \int_{\Omega} v_0(x) \hspace{0.05cm} \widehat{\varphi}_{L,j}(x) \ dx = \sum_{k=1}^{2^{\widetilde{L}}} \widehat{\boldsymbol{P}}_{L,\widetilde{L} }[k,j] \int_{\Omega} v_0(x) \hspace{0.05cm} \widehat{\varphi}_{\widetilde{L},k}(x) \ dx\\
             &=2^{-(\widetilde{L}+1)}\sum_{k=1}^{2^{\widetilde{L}}} \widehat{\boldsymbol{P}}_{L,\widetilde{L} }[k,j] \sum_{i=1}^{2^{\widetilde{L}}} \int_{D}(v_0 \circ \Phi_{\widetilde{L},i})(s)  \hspace{0.05cm} (\widehat{\varphi}_{\widetilde{L},k} \circ \Phi_{\widetilde{L},i})(s) \ ds \\
              &=2^{-(\widetilde{L}+1)}\sum_{k=1}^{2^{\widetilde{L}}} \sum_{i=1}^{2^{\widetilde{L}}} \widehat{\boldsymbol{P}}_{L,\widetilde{L} }[k,j] \sum_{\alpha=1}^2\widehat{\boldsymbol{D}}_{\widetilde{L}}[\overline{i\alpha},k]  \int_{D}(v_0 \circ \Phi_{\widetilde{L},i})(s)  \hspace{0.05cm} \psi_{\widetilde{L},\alpha}(s) \ ds,
         \end{aligned}
     \end{equation}
     which can be approximated by $2^{-\widetilde{L}} \sum_{k=1}^{2^{\widetilde{L}}} \sum_{i=1}^{2^{\widetilde{L}}} \widehat{\boldsymbol{P}}_{L,\widetilde{L} }[k,j] \widehat{\boldsymbol{D}}^{(\alpha=1)}_{\widetilde{L}}[i,k]\boldsymbol{\Pi}^{v_0}_{\widetilde{L},0}[i]$, using the orthogonality of Legendre polynomials $\psi_{\widetilde{L},1}(s) = 1$ and $\psi_{\widetilde{L},2}(s) = s$, $s\in D$.
      Namely, the approximation of $v_0 \in L^2(\Omega)$ with $\Pi_{\widetilde{L},0}(v_0) \in \mathcal{S}_0(\mathcal{T}_{\widetilde{L}})$ results in the approximation of $\widehat{\boldsymbol{g}}_L$ by the matrix-by-vector product $2^{-\widetilde{L}}\left(\widehat{\boldsymbol{D}}^{(\alpha=1)}_{\widetilde{L}}\widehat{\boldsymbol{P}}_{L,\widetilde{L}}\right)^{T}\boldsymbol{\Pi}^{v_0}_{\widetilde{L},0}$, which can be efficiently perfomed in the QTT format by using Properties~\ref{QTT-properties}, Theorem~\ref{th::low_rank_functions}, and representations~\eqref{QTT-ext_mat} and \eqref{odd_D}.
\end{proof}

\begin{remark}
Whenever the pull-back of the initial data~$u'_0\circ\phi_0$ and~$v_0\circ\phi_0$ do not belong to a finite dimensional space, one could work with approximations of them in suitable finite dimensional spaces.
\end{remark}

\begin{remark}\label{remark::L2proj_const}
    The inner product $\widehat{\boldsymbol{f}}_L$ for the initial position is explicitly and exactly represented in the QTT format. In contrast, $\widehat{\boldsymbol{g}}_L$ is only approximated. The accuracy in this approximation depends on the magnitude of $\widetilde{L}$, which determines the accuracy of the $L^2$ projection $\Pi_{\widetilde{L},0}: L^2(\Omega)\rightarrow \mathcal{S}_0(\mathcal{T}_{\widetilde{L}})$. Relying on Theorem~\ref{th::low_rank_functions}, we can handle \emph{huge} values for $\widetilde{L}$ (e.g., $\widetilde{L} = 60$, see Section~\ref{subsec:initial_data_computation}), providing a highly accurate QTT approximation of $\widehat{\boldsymbol{g}}_L$.
\end{remark}

\begin{remark}
    While a low-rank QTT approximation of the initial velocity is provided by performing the matrix-by-vector product~\eqref{linear_syst_v0} in the QTT format, a low-rank QTT approximation of the initial position is provided by considering a low-rank QTT compression of the linear system~\eqref{lin_syst_u0}, then solved with a suitable optimization solver. Since the matrix $\boldsymbol{C}_L\widehat{\boldsymbol{A}}_L\boldsymbol{C}_L$ is symmetric, positive definite, and (asymptotically) uniformly well-conditioned, in our computations we solve system~\eqref{lin_syst_u0} in the QTT format by using the Richardson iteration.
\end{remark}

\subsection{Tensor-structured GLRK--FEM iterations}\label{subsec:QTT_GLRK-FEM}
In this section, we describe 
the implementation of the GLRK--FEM iterations in~\eqref{u-v_GLRK-FEM}-\eqref{k1-k2_GLRK-FEM} in the QTT format.

Each matrix involved in the GLRK--FEM iterations~\eqref{u-v_GLRK-FEM}-\eqref{k1-k2_GLRK-FEM} admits an exact and explicit QTT decomposition with ranks independent of $L$. We give explicit rank estimates based on Property~\ref{QTT-properties}, the results of~\cite{BK2020}, and the Sherman--Morrison--Woodbury formula. As discussed at the beginning of this section, in our computations we properly pad with zeros all the matrices associated with the space discretization, so that their sizes are $2^L\times2^L$, see, e.g., equations~\eqref{lin_syst_u0} and \eqref{linear_syst_v0}. In stating the following property, for the sake of simplicity, we avoid introducing additional notation for the padded matrices, and keep the same notation that we have used in the GLRK--FEM iterations~\eqref{u-v_GLRK-FEM}-\eqref{k1-k2_GLRK-FEM}. 

\begin{property}\label{prop::wave-op_rank_bounds}
Let $L \in \N$ be the number of space levels, and $q \in \N$ be the number of GLRK stages. The following rank bounds are satisfied:
\begin{itemize}
    \item[(1)] The BPX preconditioner $\boldsymbol{C}_L$ defined in~\eqref{BPX} admits a QTT decomposition with $L$ cores and ranks up to $13$.
    \item[(2)] The matrix $\boldsymbol{b}^{(q)T }\otimes \boldsymbol{I}_L$ on the right-hand side of~\eqref{GLRK-FEM_uhat} admits a QTT decomposition with $L+1$ cores and ranks up to $1,2,\ldots,2$.
    \item[(3)] The preconditioned stiffness matrix $\boldsymbol{C}_L \widehat{\boldsymbol{A}}_L\boldsymbol{C}_L$ defining system~\eqref{GLRK-FEM_u} admits a QTT decomposition with $L$ cores and ranks up to $169$.
    \item[(4)] The matrix $\boldsymbol{C}_L \widehat{\boldsymbol{A}}_L $ on the right-hand side of~\eqref{GLRK-FEM_u} admits a QTT decomposition with $L$ cores and ranks up to $39$.
    \item[(5)] The matrix $\boldsymbol{A}^{(q)} \otimes \boldsymbol{I}_L$ on the right-hand side of~\eqref{GLRK-FEM_k1} admits a QTT decomposition with $L+1$ cores and ranks up to $1,2,\ldots,2$.
    \item[(6)] The matrix $\boldsymbol{I}^{(q)}\otimes\boldsymbol{I}_L + \tau^2 \boldsymbol{A}^{(q)}\boldsymbol{A}^{(q)} \otimes \widehat{\boldsymbol{M}}_L^{-1}\widehat{\boldsymbol{A}}_L$ defining system~\eqref{GLRK-FEM_k2} admits a QTT decomposition with $L+1$ cores and ranks up to $2,128,\ldots,128$.
    \item[(7)] The matrices $\boldsymbol{A}^{(q)} \otimes \widehat{\boldsymbol{M}}_L^{-1}\widehat{\boldsymbol{A}}_L$ and $\boldsymbol{I}^{(q)} \otimes \widehat{\boldsymbol{M}}_L^{-1}\widehat{\boldsymbol{A}}_L$ on the right-hand side of~\eqref{GLRK-FEM_k2} admit QTT decompositions with $L+1$ cores and ranks up to $1,126,\ldots,126$. 
\end{itemize}
\end{property}

The code for computing the explicit QTT decompositions whose rank bounds are detailed in Property~\ref{prop::wave-op_rank_bounds} is available for verification in~\cite{QTTwave.jl}.

\begin{remark}
The method requires to perform matrix-by-vector multiplications with
     the matrices in Property~\ref{prop::wave-op_rank_bounds}. If, during this process, one applies low-rank re-approximations~\cite[Algo. 2]{O2011}, the (high) QTT rank bounds in Property~\ref{prop::wave-op_rank_bounds} may not affect the resulting inexact scheme.
\end{remark}

From now on, we make the following assumption.
\begin{assumption}\label{assumpt::q}
    Let $q$ be the number of GLRK stages, and $\ceil{\cdot}$ be the ceiling operator. The uniform time step $\tau$ is defined as $\tau=2^{-\ceil{L/q}}$, where $L$ is the number of space levels.
\end{assumption}
Assumption~\ref{assumpt::q} ensures that, if the continuous initial data $u_0$ and $v_0$ satisfy suitable compatibility conditions~\cite[Th. 8]{FP1996}, the fully-discrete solutions~\eqref{fully_discr_sol} approximate the exact solutions 
at each discrete time with accuracy $\mathcal{O}(2^{-2L})$ and $\mathcal{O}(2^{-L})$, respectively, in $\|\cdot\|_{L^2(\Omega)}$ and $\|\cdot\|_{H^1_0(\Omega)}$, for $L \rightarrow \infty$ and $q \rightarrow \infty$. Then, assume that the number of GLRK stages $q$ is sufficiently large. In that case, we can obtain exponential convergence of the GLRK--FEM~\eqref{u-v_GLRK-FEM}-\eqref{k1-k2_GLRK-FEM} in the number of levels $L$. In this setting, the computational complexity reduction provided by the QTT decomposition permits the use of a high number of levels, allowing to benefit from the exponential convergence of the method.

As discussed in Section~\ref{sec::fully_discrete}, the linear systems we solve in the GLRK--FEM iterations~\eqref{u-v_GLRK-FEM}-\eqref{k1-k2_GLRK-FEM} are system~\eqref{GLRK-FEM_u}, whose matrix is uniformly well-conditioned, and system~\eqref{GLRK-FEM_k2}, which defines the intermediate slopes for velocity. According to the conditioning behavior of the latter, we split the remaining discussion of this subsection into two parts: the first part is devoted to the one-stage GLRK--FEM, and the second part addresses the GLRK--FEM with $q>1$ stages.

\subsubsection{One-stage GLRK--FEM}\label{subsec:midpoint_matrix}

Let $L \in \N$ denote the number of space levels. Let us consider the GLRK--FEM iterations~\eqref{u-v_GLRK-FEM}-\eqref{k1-k2_GLRK-FEM} with one stage ($q=1$). Namely, we consider the implicit midpoint method whose Butcher tableau is detailed in Example~\ref{ex:GLRK}. In this setting, the systems defining the intermediate slopes for velocity read as
\begin{equation}\label{syst_k2_q1}
    \left(\boldsymbol{I}_L + \frac{\tau^2}{4}\widehat{\boldsymbol{M}}^{-1}_L\widehat{\boldsymbol{A}}_L \right)\boldsymbol{k}^{(n,q)}_{v,L}= -\frac{\tau}{2} \widehat{\boldsymbol{M}}^{-1}_L\widehat{\boldsymbol{A}}_L \widehat{\boldsymbol{v}}^{(n)}_L - \widehat{\boldsymbol{M}}^{-1}_L\widehat{\boldsymbol{A}}_L \boldsymbol{C}_L \boldsymbol{u}^{(n)}_L \quad \text{for } n= 0,1,\ldots,n_t-1.
\end{equation}
Under Assumption~\ref{assumpt::q}, 
the matrices~$\boldsymbol{I}_L + \frac14 \tau^2\widehat{\boldsymbol{M}}^{-1}_L\widehat{\boldsymbol{A}}_L$ are uniform well-conditioning in $L$, as reported in Table~\ref{tab::cond_q1}. Therefore, the QTT representation of each system in~\eqref{syst_k2_q1} can be solved using a suitable gradient-based optimization method for the least square formulation of~\eqref{syst_k2_q1}, such as gradient descent.
\begin{table}[h!]
    \centering
    \begin{tabular}{ccccc}
        \toprule
        $L$ & $h_L$ & $\sigma_{\min}$ & $\sigma_{\max}$ & $\kappa_2$\\
        \midrule
        1 & 5.0000000e-1 & 1.7500000e0 & 1.7500000e0 & 1.0000000e0\\
        2 & 2.5000000e-1 & 1.1622913e0 & 2.9805659e0 & 2.5643881e0\\
        3 & 1.2500000e-1 & 1.0390511e0 & 3.6816879e0 & 3.5433175e0\\
        4 & 6.2500000e-2 & 1.0096693e0 & 3.9151639e0 & 3.8776696e0\\
        5 & 3.1250000e-2 & 1.0024115e0 & 3.9784351e0 & 3.9688642e0\\
        6 & 1.5625000e-2 & 1.0006025e0 & 3.9945861e0 & 3.9921808e0\\
        7 & 7.8125000e-3 & 1.0001506e0 & 3.9986451e0 & 3.9980429e0\\
        8 & 3.9062500e-3 & 1.0000377e0 & 3.9996612e0 & 3.9995104e0\\
        9 & 1.9531250e-3 & 1.0000094e0 & 3.9999153e0 & 3.9998777e0\\
        10 & 9.7656250e-4 & 1.0000024e0 & 3.9999788e0 & 3.9999692e0\\
    \bottomrule
    \end{tabular}
    \caption{Spectral condition numbers $\kappa_2$ of the system matrix in~\eqref{syst_k2_q1}. We denote with $h_L=2^{-L}$ the uniform mesh size in space, and with $\sigma_{\min}$ and $\sigma_{\max}$, respectively, the minimal and maximal singular values of the considered matrix.}
    \label{tab::cond_q1}
\end{table}

\subsubsection{High-order GLRK--FEM}\label{subsec:high-order_GLRK-FEM}

Under Assumption~\ref{assumpt::q}, and aiming at obtaining exponential convergence of the GLRK--FEM~\eqref{u-v_GLRK-FEM}-\eqref{k1-k2_GLRK-FEM} with respect to $L$, relying on one-stage Runge Kutta iterations to approximate time evolution of wavefronts might be computationally unaffordable. Therefore, considering GLRK--FEMs with $q>1$ stages is more suitable for our setting. 

Let us define the following matrix:
\begin{equation}\label{syst_k2_generalq}
    \boldsymbol{H}^{(q)}_L := \boldsymbol{I}^{(q)}\otimes\boldsymbol{I}_L + \tau^2 \boldsymbol{A}^{(q)}\boldsymbol{A}^{(q)} \otimes \widehat{\boldsymbol{M}}_L^{-1}\widehat{\boldsymbol{A}}_L,
\end{equation}
which is the operator that defines the intermediate slopes for velocity in~\eqref{GLRK-FEM_k2}, and which coincides with the system matrix in~\eqref{syst_k2_q1} when $q=1$. As shown in Tables~\ref{tab::cond_q2},~\ref{tab::cond_q5}, and~\ref{tab::cond_q8}, the matrix in~\eqref{syst_k2_generalq} is  not uniformly well conditioned for $q>1$. Numerically, we observe the spectral condition number to grow as $\mathcal{O}\left(q^{4/3}(2^{L-\ceil{L/q}})^{2}\right)$ for $L \rightarrow \infty$, like the maximal singular value. Therefore, computing discrete solutions for large $L$ requires preconditioning. 

\begin{remark}
    The GLRK--FEM iterations~\eqref{u-v_GLRK-FEM}-\eqref{k1-k2_GLRK-FEM} have been written exploiting the block structure of the matrix~$\boldsymbol{B}_{2L}$ defined in~\eqref{matrix_midp}. We have observed numerically that the coefficient matrix of the linear system in~\eqref{GLRK_FEM}, from which we derive the GLRK--FEM~\eqref{u-v_GLRK-FEM}-\eqref{k1-k2_GLRK-FEM}, presents the same ill-conditioning behavior as the matrix~$\boldsymbol{H}^{(q)}_L$ in~\eqref{syst_k2_generalq} for $L \rightarrow \infty$.
\end{remark}

\begin{table}[h!]
\centering
\caption*{$q = 2$}
\begin{tabular}{cccccc}
    \toprule
        $L$ & $h_L$ & $q^{4/3}\left(h_L^{-1}\tau^{(q)}_L\right)^{2}$ & $\sigma_{\min}$ & $\sigma_{\max}$ & $\kappa_2$\\
    \midrule
    1 & 5.0000000e-1 & 2.5198421e0 & 8.3045635e-1 & 1.5804564e0 & 1.9031179e0\\
        2 & 2.5000000e-1 & 1.0079368e1& 7.9128785e-1 & 9.0899623e0 & 1.1487555e1\\
        3 & 1.2500000e-1 & 1.0079368e1& 7.6896481e-1 & 1.2155002e1 & 1.5806968e1\\
        4 & 6.2500000e-2 & 4.0317474e1& 7.7338419e-1 & 5.1643689e1 & 6.6776241e1\\
        5 & 3.1250000e-2 & 4.0317474e1& 7.6373267e-1 & 5.2757857e1 & 6.9078958e1\\
        6 & 1.5625000e-2 & 1.6126989e2& 7.6373499e-1 & 2.1127753e2 & 2.7663723e2\\
        7 & 7.8125000e-3 & 1.6126989e2& 7.6373499e-1 & 2.1156352e2 & 2.7701169e2\\
        8 & 3.9062500e-3 & 6.4507958e2& 7.6376779e-1 & 8.4569342e2 & 1.1072651e3\\
        9 & 1.9531250e-3 & 6.4507958e2 & 7.6376779e-1& 8.4576504e2& 1.1073589e3\\
        10 & 9.7656250e-4 &  2.5803183e3 & 7.6377799e-1& 3.3822959e3 & 4.4283757e3\\
    \bottomrule
\end{tabular}
    \caption{Spectral condition numbers $\kappa_2$ of matrix $\boldsymbol{H}^{(q)}_L$ defined in~\eqref{syst_k2_generalq} with $q=2$. We denote with $h_L=2^{-L}$ the uniform mesh size in space, $\tau^{(q)}_{L}:=2^{-\ceil{L/q}}$ the uniform step size, and with $\sigma_{\min}$ and $\sigma_{\max}$, respectively, the minimal and maximal singular values of the considered matrix.}
    \label{tab::cond_q2}
\end{table}

\begin{table}[h!]
\centering
    \caption*{$q = 5$}
    \begin{tabular}{cccccc}
        \toprule
            $L$ & $h_L$ & $q^{4/3}\left(h_L^{-1}\tau^{(q)}_L\right)^{2}$ & $\sigma_{\min}$ & $\sigma_{\max}$ & $\kappa_2$\\
        \midrule
        1 & 5.0000000e-1 & 8.5498797e0 & 7.4999052e-1 & 1.6479997e0 & 2.1973607e0\\
        2 & 2.5000000e-1 & 3.4199519e1 & 3.9812187e-1 & 1.0072352e1 & 2.5299670e1\\
        3 & 1.2500000e-1 & 1.3679808e2& 3.6108045e-1 & 5.3102702e1 & 1.4706612e2\\
        4 & 6.2500000e-2 & 5.4719230e2 & 3.6054521e-1 & 2.3009071e2 & 6.3817436e2\\
        5 & 3.1250000e-2 & 2.1887692e3 & 3.5915199e-1 & 9.3963938e2 & 2.6162722e3\\
        6 & 1.5625000e-2 & 2.1887692e3 & 3.5915199e-1 & 9.4473349e2 & 2.6304559e3\\
        7 & 7.8125000e-3 & 8.7550769e3 & 3.5922799e-1 & 3.7833867e3 & 1.0531993e4\\
        8 & 3.9062500e-3 & 3.5020307e4 & 3.5927318e-1 & 1.5138009e4 & 4.2135093e4\\
        9 & 1.9531250e-3 & 1.4008123e5& 3.5928609e-1 & 6.0556501e4 & 1.6854675e5\\
        10 & 9.7656250e-4 & 5.6032492e5 & 3.5928942e-1 & 2.4223047e5 & 6.7419316e5\\
        \bottomrule
    \end{tabular}
        \caption{Spectral condition numbers $\kappa_2$ of matrix $\boldsymbol{H}^{(q)}_L$ defined in~\eqref{syst_k2_generalq} with $q=5$.}
        \label{tab::cond_q5}
\end{table}

\begin{table}[h!]
\centering
\caption*{$q = 8$}
    \begin{tabular}{cccccc}
        \toprule
            $L$ & $h_L$ & $q^{4/3}\left(h_L^{-1}\tau^{(q)}_L\right)^{2}$ & $\sigma_{\min}$ & $\sigma_{\max}$ & $\kappa_2$\\
        \midrule
        1 & 5.0000000e-1 & 1.5999999e1& 7.4258759e-1 & 1.6713376e0 & 2.2506942e0\\
        2 & 2.5000000e-1 & 6.3999999e1& 3.8306418e-1 & 1.0394605e1 & 2.7135414e1\\
        3 & 1.2500000e-1 & 2.5599999e2& 2.1174765e-1 & 5.4874372e1 & 2.5914985e2\\
        4 & 6.2500000e-2 & 1.0239999e3& 2.1399304e-1 & 2.3780644e2 & 1.1112812e3\\
        5 & 3.1250000e-2 & 4.0959999e3& 2.1193086e-1 & 9.7118174e2 & 4.5825405e3\\
        6 & 1.5625000e-2 & 1.6383999e4& 2.1279012e-1 & 3.9051284e3 & 1.8352019e4\\
        7 & 7.8125000e-3 & 6.5535999e4& 2.1312087e-1 & 1.5641029e4 & 7.3390415e4\\
        8 & 3.9062500e-3 & 2.6214399e5& 2.1321093e-1 & 6.2584659e4 & 2.9353401e5\\
        9 & 1.9531250e-3 & 2.6214399e5 & 2.1210756e-1& 6.2589961e4& 2.9508595e5\\
        10 & 9.7656250e-4 & 1.0485759e6 & 2.1209649e-1& 2.5036449e5& 1.1804273e6\\
        \bottomrule
    \end{tabular}
        \caption{Spectral condition numbers $\kappa_2$ of matrix $\boldsymbol{H}^{(q)}_L$ defined in~\eqref{syst_k2_generalq} with $q=8$.}
        \label{tab::cond_q8}
\end{table}

The ill-conditioning of $\boldsymbol{H}^{(q)}_L$ in $L$ is mainly due to the stiffness matrix $\widehat{\boldsymbol{A}}_L$, which is computed with respect to the $L^2$-normalized nodal basis~\eqref{basis}. Recalling that $\boldsymbol{C}_L\widehat{\boldsymbol{A}}_L\boldsymbol{C}_L$, with $\boldsymbol{C}_L$ the BPX preconditioner~\eqref{BPX}, is well conditioned, we employ 
\begin{equation}\label{precond}
    \boldsymbol{Q}^{(q)}_L:= \boldsymbol{I}^{(q)} \otimes \boldsymbol{C}_L,
\end{equation}
as a symmetric preconditioner,
and consider the preconditioned problems: For $n=0,1,\ldots,n_t-1$, find $\widetilde{\boldsymbol{k}}^{(n,q)}_{v,L}$ such that
\begin{equation}\label{GLKR-FEM_k2_precond}
        \begin{split}
    \boldsymbol{Q}^{(q)}_L \boldsymbol{H}^{(q)}_L \boldsymbol{Q}^{(q)}_L \widetilde{\boldsymbol{k}}^{(n,q)}_{v,L} &= -\tau\boldsymbol{Q}^{(q)}_L\left( \boldsymbol{A}^{(q)} \otimes \widehat{\boldsymbol{M}}_L^{-1}\widehat{\boldsymbol{A}}_L\right)\left((1,\ldots,1)^T \otimes \widehat{\boldsymbol{v}}^{(n)}_L\right)\\ 
    &\qquad- \boldsymbol{Q}^{(q)}_L\left( \boldsymbol{I}^{(q)} \otimes \widehat{\boldsymbol{M}}_L^{-1}\widehat{\boldsymbol{A}}_L\right) \left((1,\ldots,1)^T \otimes \boldsymbol{C}_L \boldsymbol{u}^{(n)}_L\right), 
    \end{split} 
\end{equation}
with $(1,\ldots,1)^T \in \R^q$. Then, for each $n=0,1,\ldots,n_t-1$, $\boldsymbol{k}^{(n,q)}_{v,L} := \boldsymbol{Q}^{(q)}_L \widetilde{\boldsymbol{k}}^{(n,q)}_{v,L}$ solves~\eqref{GLRK-FEM_k2}, i.e., 
\begin{equation*}
\begin{split}
\boldsymbol{H}^{(q)}_L \boldsymbol{k}^{(n,q)}_{v,L} &=-\tau\left( \boldsymbol{A}^{(q)} \otimes \widehat{\boldsymbol{M}}_L^{-1}\widehat{\boldsymbol{A}}_L\right)\left((1,\ldots,1)^T \otimes \widehat{\boldsymbol{v}}^{(n)}_L\right)\\ &\qquad-\left( \boldsymbol{I}^{(q)} \otimes \widehat{\boldsymbol{M}}_L^{-1}\widehat{\boldsymbol{A}}_L\right) \left((1,\ldots,1)^T \otimes \boldsymbol{C}_L \boldsymbol{u}^{(n)}_L\right).
\end{split}
\end{equation*}

\begin{remark}
    The matrix $\boldsymbol{Q}^{(q)}_L$ defined in~\eqref{precond} admits a QTT decomposition with $L+1$ cores and QTT ranks up to $1,13,\ldots,13$, as a result of Property~\ref{prop::wave-op_rank_bounds}.
\end{remark}

While not an optimal preconditioner, it can be seen in Figure~\ref{fig:precond_q5} 
that, for~$q=5,8$, the preconditioner~\eqref{precond} reduces the condition number of $\boldsymbol{H}^{(q)}_L $. Indeed, after initial growth, the condition number of $\boldsymbol{Q}^{(q)}_L \boldsymbol{H}^{(q)}_L \boldsymbol{Q}^{(q)}_L$ saturates. In contrast, we have observed that this is not true for $q=2$.

\begin{figure}[h!]
    \centering
    \includegraphics[width=0.48\linewidth]{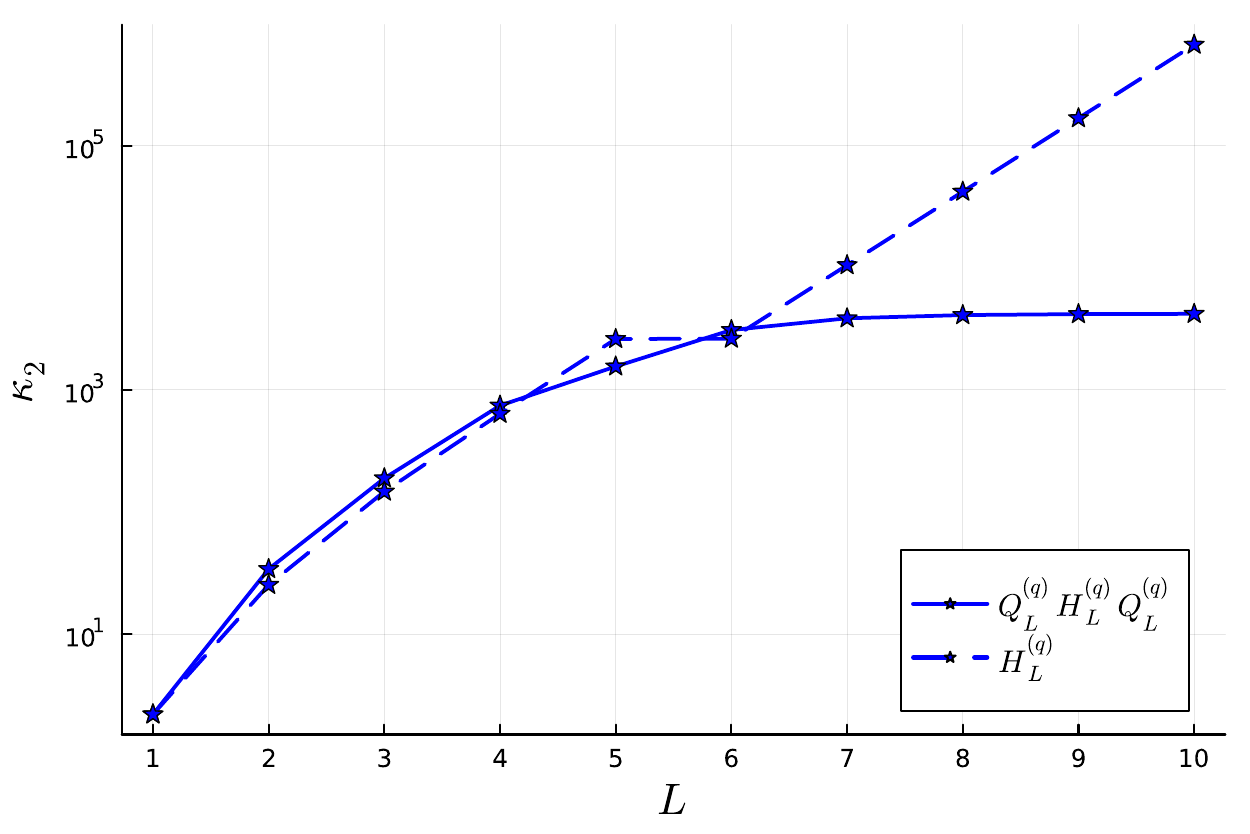}
    \includegraphics[width=0.48\linewidth]{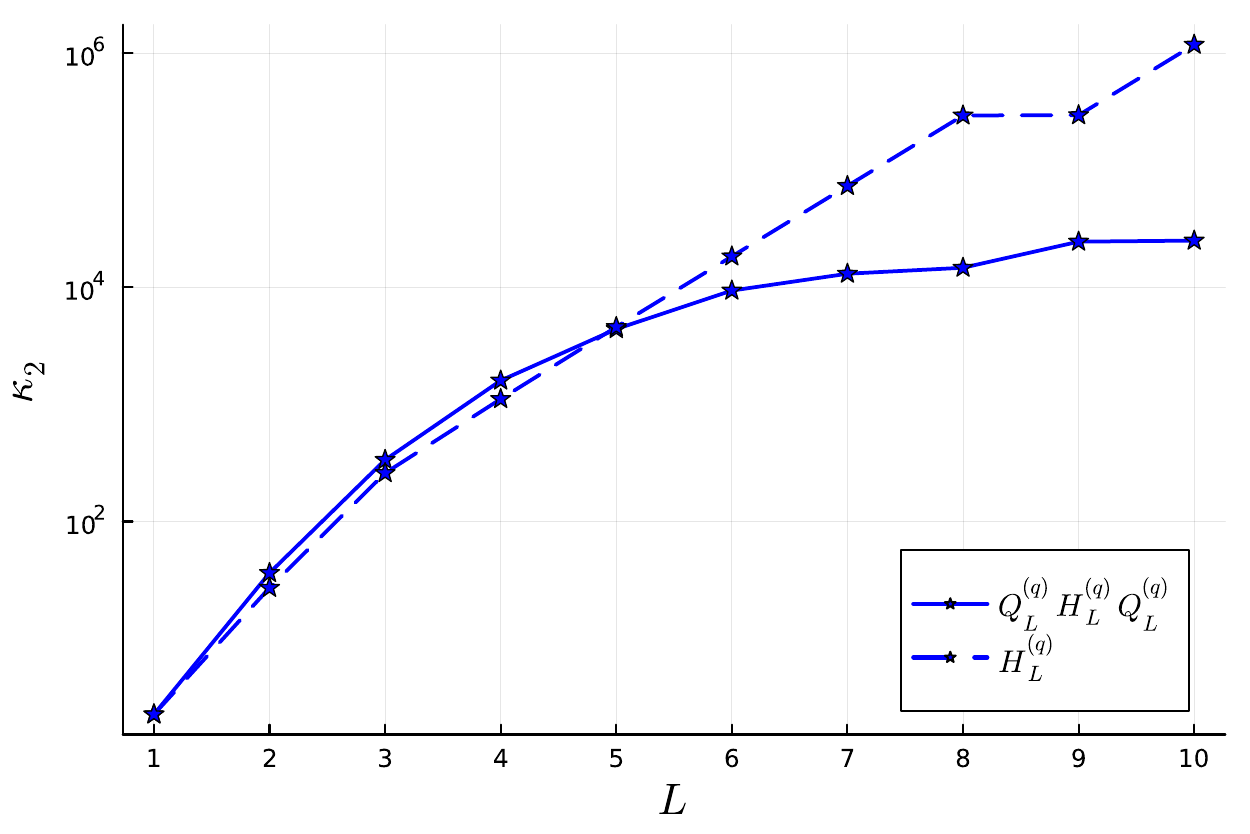}
    
    \caption{Spectral condition numbers $\kappa_2$ of matrix $\boldsymbol{Q}^{(q)}_L \boldsymbol{H}^{(q)}_L \boldsymbol{Q}^{(q)}_L$ (solid line), with $\boldsymbol{Q}^{(q)}_L$ defined in~\eqref{precond}, and $\boldsymbol{H}^{(q)}_L$ defined in~\eqref{syst_k2_generalq} (dashed line). The number of Runge--Kutta stages is $q=5$ (left plot) and $q=8$ (right plot).}
    \label{fig:precond_q5}
\end{figure}

The preconditioned GLRK--FEM iterations are implemented as shown in Algorithm~\ref{alg:GLRK-FEM_precond}. 

\begin{remark}
    If $q=1$, we compress and solve Algorithm~\ref{alg:GLRK-FEM} iterations in the QTT format. For $q>1$, we compress the iterations presented in Algorithm~\ref{alg:GLRK-FEM_precond}.
\end{remark}

\begin{algorithm}
\caption{Preconditioned GLRK--FEM}\label{alg:GLRK-FEM_precond}
\hspace*{\algorithmicindent} \textbf{Input:} $\left(\boldsymbol{u}^{(n)}_L,\widehat{\boldsymbol{v}}^{(n)}_L\right)$\\
\hspace*{\algorithmicindent} \textbf{Output:} $\left(\boldsymbol{u}^{(n+1)}_L,\widehat{\boldsymbol{v}}^{(n+1)}_L\right)$
\begin{algorithmic}[1]
\State $\widehat{\boldsymbol{u}}^{(n)}_L \gets \boldsymbol{C}_L \boldsymbol{u}^{(n)}_L$
\State $\widehat{\boldsymbol{u}}^{(n,q)}_L \gets (1,\ldots,1)^T \otimes \widehat{\boldsymbol{u}}^{(n)}_L$, with $(1,\ldots,1)^T \in \R^q$
\State $\widehat{\boldsymbol{v}}^{(n,q)}_L \gets (1,\ldots,1)^T \otimes \widehat{\boldsymbol{v}}^{(n)}_L$, with $(1,\ldots,1)^T \in \R^q$
\State $\widetilde{\boldsymbol{k}}^{(n,q)}_{v,L} \gets $ Solve \eqref{GLKR-FEM_k2_precond} given $\left(\widehat{\boldsymbol{u}}^{(n,q)}_L, \widehat{\boldsymbol{v}}^{(n,q)}_L\right)$
\State $\boldsymbol{k}^{(n,q)}_{v,L} \gets \boldsymbol{Q}^{(q)}_L \widetilde{\boldsymbol{k}}^{(n,q)}_{v,L}$
\State $\widehat{\boldsymbol{v}}^{(n+1)}_L \gets \widehat{\boldsymbol{v}}^{(n)}_L + \tau \left(\boldsymbol{b}^{(q)  T} \otimes \boldsymbol{I}_L \right) \boldsymbol{k}_{v,L}^{(n,q)}$
\State $\boldsymbol{Q}^{(q)}_L \boldsymbol{k}_{u,L}^{(n,q)} \gets \tau \left(\boldsymbol{A}^{(q)} \otimes \boldsymbol{I}_L\right) \boldsymbol{k}_{v,L}^{(n,q)} + \widehat{\boldsymbol{v}}^{(n,q)}_L $
\State $\widehat{\boldsymbol{u}}^{(n+1)}_L \gets \widehat{\boldsymbol{u}}^{(n)}_L + \tau \left(\boldsymbol{b}^{(q)T }\otimes \boldsymbol{I}_L \right) \boldsymbol{Q}^{(q)}_L\boldsymbol{k}_{u,L}^{(n,q)}$
\State $\boldsymbol{u}^{(n+1)}_L \gets$ Solve~\eqref{GLRK-FEM_u} given $\widehat{\boldsymbol{u}}^{(n+1)}_L$
\end{algorithmic}
\end{algorithm}

\section{Numerical experiments}\label{sec:numerical_experiments}

In this section, we present numerical experiments demonstrating the efficiency of the QTT-compressed GLRK--FEM. All computations are carried out in Julia. The Julia package~\cite{TensorRefinement.jl} is extensively exploited to perform low-rank computations in the QTT format.

\subsection{Example 1: low-frequency oscillations}\label{sec:low-freq}
Throughout this section, the time-harmonic, standing wave
\begin{equation}\label{exact_sol_u}
    u(x,t) = (\cos(\pi t) + \sin(\pi t)) \sin(\pi x), \quad (x,t) \in \Omega \times (0,T)=(0,1)^2
\end{equation}
and its time derivative
\begin{equation}\label{exact_sol_v}
    v(x,t) = \pi (\cos{(\pi t)} - \sin{(\pi t)}) \sin(\pi x), \quad (x,t) \in \Omega \times (0,T)
\end{equation}
are taken as the reference exact solutions of the wave propagation problem~\eqref{eq_wave_ham}. 
The corresponding initial data $(u_0,v_0)$ given by
\begin{subequations}\label{initial_data_trig}
\begin{align}
    u_0(x) &= \sin{(\pi x)}, & x &\in \Omega,\\
    v_0(x) &= \pi \sin{(\pi x)}, & x &\in \Omega,
\end{align}
\end{subequations}
satisfy the assumptions of Theorem~\ref{th::low_rank_functions}, as observed in Example~\ref{example:low_rank_functions}, and the compatibility conditions of~\cite[Th. 8]{FP1996} hold. Furthermore, the function in~\eqref{exact_sol_u} and its time derivative~\eqref{exact_sol_v} are sinusoidal functions with different amplitudes at each fixed time. Therefore, the QTT rank bounds we consider for their QTT approximations do not change over time.

\subsubsection{Construction of the discrete initial data}\label{subsec:initial_data_computation}

Let $L\in \N$ be the number of space levels. With abuse of notation, we denote with $\boldsymbol{u}_{L,0}$ and $\widehat{\boldsymbol{v}}_{L,0}$ the QTT representations of the initial data that we compute solving~\eqref{lin_syst_u0} and \eqref{linear_syst_v0} in the QTT format. To compute $\boldsymbol{u}_{L,0}$ and $\widehat{\boldsymbol{v}}_{L,0}$, the following approach is used.
\begin{itemize}
    \item The inner products in~\eqref{RHS} are approximated with $\widetilde{L} = 60$ levels and the scheme described in Section~\ref{subsec:initdata}.
    \item System~\eqref{lin_syst_u0} is solved using the truncated Richardson iteration with $150$ maximal number of iterations.
    \item TT rank bounds for the discrete position are up to $14$ at each  Richardson iteration. These bounds were obtained heuristically. Accordingly, TT rank bounds for the residuals are up to $28$ at each iteration.
\end{itemize}

As shown in Figure~\ref{fig:QTT_initial_data}, exponential convergence in $L$ is provided for the discrete initial position and the discrete initial velocity. Exploiting the QTT parametrization of a straightforward space discretization allows us to reach a very high
accuracy and affordable computational complexity across the range $\{1,\ldots,20\}$ of the number of levels $L$. Specifically, to compute $\widehat{\boldsymbol{v}}_{L,0}$, we perform the following matrix-by-matrix multiplication in the QTT format:
\begin{equation*}
    \widehat{\boldsymbol{v}}_{L,0} =
    \begin{pmatrix}
        \widehat{\boldsymbol{M}}^{-1}_L & \\
        & 0 
    \end{pmatrix}
    \widehat{\boldsymbol{g}}_L.
\end{equation*}
The cost of this operation is \emph{linear} in $L$, instead of being \emph{exponential}, and is mainly governed by the square of the product between the QTT ranks of the padded inverse mass matrix and the QTT ranks of $\widehat{\boldsymbol{g}}_L$, which are both uniformly bounded with respect to $L$. To compute $\boldsymbol{u}_{L,0}$, we solve the following padded linear system using  Richardson iteration:
\begin{equation*}
    \begin{pmatrix}
        \boldsymbol{C}_L\widehat{\boldsymbol{A}}_L\boldsymbol{C}_L & \\
        & 0
    \end{pmatrix} 
    \boldsymbol{u}_{L,0} = 
    \begin{pmatrix}
        \boldsymbol{C}_L & \\
        & 0
    \end{pmatrix}
    \widehat{\boldsymbol{f}}_L,
\end{equation*}
whose fast convergence is ensured by the uniform well-conditioning behavior in $L$ of the matrix $\boldsymbol{C}_L\widehat{\boldsymbol{A}}_L\boldsymbol{C}_L$. The computational cost of each iteration is mainly governed by the rounding procedure that reduces the ranks of the updated position and residual. For each iteration, this cost is \emph{linear} in $L$, and mainly governed by the cubes of the involved QTT ranks, which are uniformly bounded with respect to $L$. 

\begin{figure}[h!]
    \centering
    \includegraphics[width=0.6\linewidth]{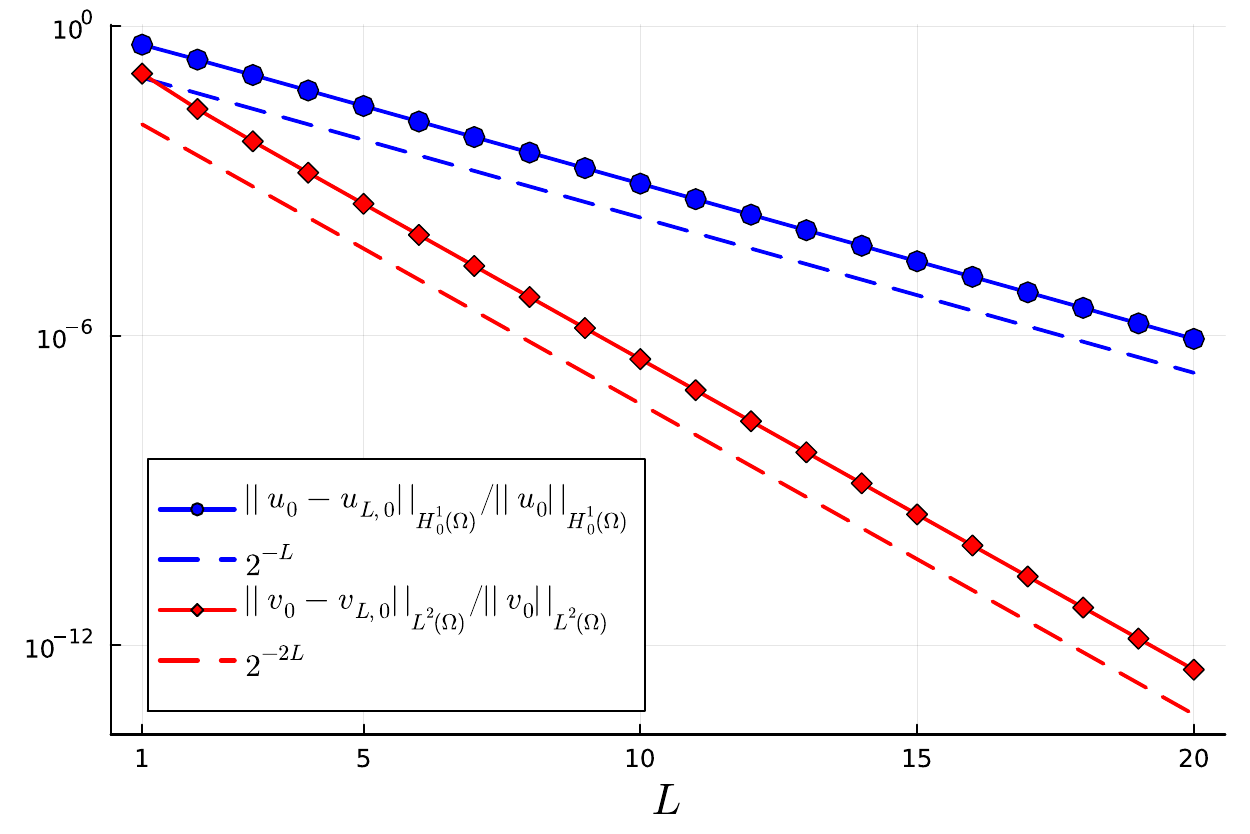}
    \caption{Horizontal axis: number of space levels L. Vertical axis: relative errors between the continuous initial data~\eqref{initial_data_trig} and their approximations from~\eqref{lin_syst_u0} and \eqref{linear_syst_v0} when compressed in the QTT format. The error of the initial position in $\|\cdot\|_{H^1_0(\Omega)}$ (solid-blue line with round markers) decreases as $2^{-L}$ (dashed-blue line). The error of the initial velocity in $\|\cdot\|_{L^2(\Omega)}$ (solid-red line with diamond markers) decreases as $2^{-2L}$ (dashed-red line).}
    \label{fig:QTT_initial_data}
\end{figure}

\begin{remark}\label{remrk:err_initial_data}
    To compute the errors between the discrete and exact initial data we need accurate QTT approximations of the functions defined in~\eqref{initial_data_trig}. Low-rank QTT approximations of these functions are provided following a similar construction as the one in Section~\ref{subsec:initdata}. In this case, instead of considering the $L^2$ projection onto piecewise constant functions, we consider the linear interpolation between the endpoints of $2^{\widetilde{L}}$ subintervals of $\Omega$, with $\widetilde{L} = 60$ as for the inner products~\eqref{RHS}. 
\end{remark}

\subsubsection{Accuracy in the energy norm}\label{subsec:accuracy_test}
In this section, we test the accuracy of the GLRK--FEM when compressed in the QTT format. Let $L \in \N$ be the number of space levels. Let $n_t$ be the number of time steps of the GLRK method with $q \in \N$ stages. With abuse of notation, we denote with $\boldsymbol{u}^{(n)}_L \in \R^{2^L}$ and $\widehat{\boldsymbol{v}}^{(n)}_L \in \R^{2^L}$ the QTT-compressed vectors representing the coefficients of the finite element functions that approximate, respectively, the exact position and exact velocity at each discrete time $t_n = n\tau$, for $n \in \{0,\ldots,n_t\}$. These QTT approximations are provided by applying Algorithm~\ref{alg:GLRK-FEM_precond} in the format if $q>1$, or Algorithm~\ref{alg:GLRK-FEM} if $q=1$. As such, at each discrete time $t_n = n\tau$, the numerical solution is given by
\begin{equation}\label{err_L2H10}
\begin{aligned}
    u^{(n)}_L(x)&=\sum_{i=1}^{N_L} \boldsymbol{u}_L^{(n)}[i] \ \varphi_{L,i}(x), \quad &&x \in \Omega,\\
    v^{(n)}_L(x)&=\sum_{j=1}^{N_L} \widehat{\boldsymbol{v}}_L^{(n)}[j] \ \widehat{\varphi}_{L,j}(x), \quad &&x \in \Omega,
\end{aligned}
\end{equation}
with $N_L=2^L-1$. We
investigate the following errors:
\begin{equation}\label{err_norms}
    \big{\|}u(\cdot,t_{n})-u^{(n)}_L\big{\|}_{H^1_0(\Omega)}, \qquad 
    \big{\|}v(\cdot,t_{n})-v^{(n)}_L\big{\|}_{L^2(\Omega)},
\end{equation}
for $n \in \{0,\ldots,n_t\}$. In computing these errors, QTT representations of the continuous functions at each discrete time are devised as observed in Remark~\ref{remrk:err_initial_data}. Additionally, the GLRK--FEM computations in the format are performed respecting the following assumptions:
\begin{itemize}
    \item A QTT representation of the initial data is provided as described in Section~\ref{subsec:initial_data_computation}.
    \item The step size $\tau$ of the GLRK method satisfies Assumption~\ref{assumpt::q}, i.e., $\tau = 2^{-\ceil{L/q}}$.
    \item System~\eqref{GLKR-FEM_k2_precond} ($q>1$), or~\eqref{GLRK-FEM_k2} ($q=1$), is solved in the QTT format by considering normal equations solved using the truncated gradient descent with $200$ maximal number of iterations.
    \item System~\eqref{GLRK-FEM_u} is solved in the QTT format using the truncated Richardson iteration with $200$ maximal number of iterations.
    \item For each $n \in \{1,\ldots,n_t\}$, the QTT ranks for $\boldsymbol{u}^{(n)}_L \in \R^{2^L}$ are up to $14$. The same rank bounds are used for the intermediate slopes for the velocity when solving system~\eqref{GLRK-FEM_k2}, or its preconditioned version~\eqref{GLKR-FEM_k2_precond}.
\end{itemize}

For $L=15$ space levels and $q=5$ GLRK stages, Table~\ref{tab:err_GLRK-FEM} shows the errors~\eqref{err_L2H10} at each discrete time $t_n = n\tau \in [0,T]=[0,1]$, for $n \in \{0,\ldots,n_t\}$, where $n_t = \tau^{-1} = 2^3$. We can see that there is no error blow-up in time.
However, the convergence rates expected for a smooth solution (Remark~\ref{FP}) are not reached. This might be related to the conditioning behavior of system~\eqref{GLKR-FEM_k2_precond}, whose ill-conditioning compared to system~\eqref{GLRK-FEM_k2} is only mitigated, see~Figure~\ref{fig:precond_q5}. This conditioning behavior slows down the convergence of gradient descent for the normal equations associated with system~\eqref{GLRK-FEM_k2}, meaning that we would need a much larger number of iterations than $200$ to see optimal accuracy for large $L$. Therefore, providing an optimal preconditioner for system~\eqref{GLRK-FEM_k2} is desirable, and its construction is left to future work.

The conditioning behavior of system~\eqref{GLKR-FEM_k2_precond} as the source of suboptimal accuracy, rather than the QTT compression, is corroborated by the results in Figure~\ref{fig:err_GLRK-FEM_q1}, where the relative errors at the final time $T=1$ for $L = 1,\ldots,6$, and $q=1$ are shown. In this case, the matrix defining system~\eqref{GLRK-FEM_k2} is uniformly well-conditioned, as discussed in Section~\ref{subsec:midpoint_matrix}. Therefore, no preconditioning is required, and the iterations compressed in the QTT format are those presented in Algorithm~\ref{alg:GLRK-FEM}. As we can see, optimal convergence rates are achieved for both position and velocity.

\begin{table}[h!]
    \centering
    \caption*{$L = 15$, $q=5$}
    \begin{tabular}{c|c}
    \toprule
    $h_L =$ 3.0517578e-5 & $h_L^2 = $ 9.3132258e-10\\
    \bottomrule
    \end{tabular}
    \\
    \begin{tabular}{lcc}
    \toprule
    $t_n$ & $\big{\|}u(x,t_{n})-u^{(n)}_L\big{\|}_{H^1_0(\Omega)} 
    $ & $\big{\|}v(x,t_{n})-v^{(n)}_L\big{\|}_{L^2(\Omega)}$\\
     \midrule
    0 & 6.1481462e-5 & 7.6097169e-10 \\
    0.125 & 2.5140404e-3 & 3.1326202e-3 \\
    0.25 & 1.5408062e-3 & 7.8311771e-4\\
    0.375 & 2.2194197e-3 & 1.5227397e-3\\
    0.5 & 1.3508461e-3 & 7.1919243e-4 \\
    0.625 & 1.5531260e-3 & 1.1687711e-3\\
    0.75 & 1.8450044e-3 & 1.7954836e-3\\
    0.875 & 2.7934582e-3 & 1.8307384e-3 \\
    1 & 2.6643520e-3 & 1.8919887e-3\\ 
        \bottomrule
    \end{tabular}
    \caption{Time evolution of the errors in $\|\cdot\|_{H^1_0(\Omega)}$ and $\|\cdot\|_{L^2(\Omega)}$for the QTT-compressed discrete position and velocity, respectively, associated with the QTT-compressed iterations described in Algorithm~\ref{alg:GLRK-FEM_precond}, with $L=15$ and $q=5$. We denote with $h_L=2^{-L}$ the uniform mesh size in space.}
    \label{tab:err_GLRK-FEM}
\end{table}

\begin{figure}[h!]
    \centering
    \includegraphics[width=0.6\linewidth]{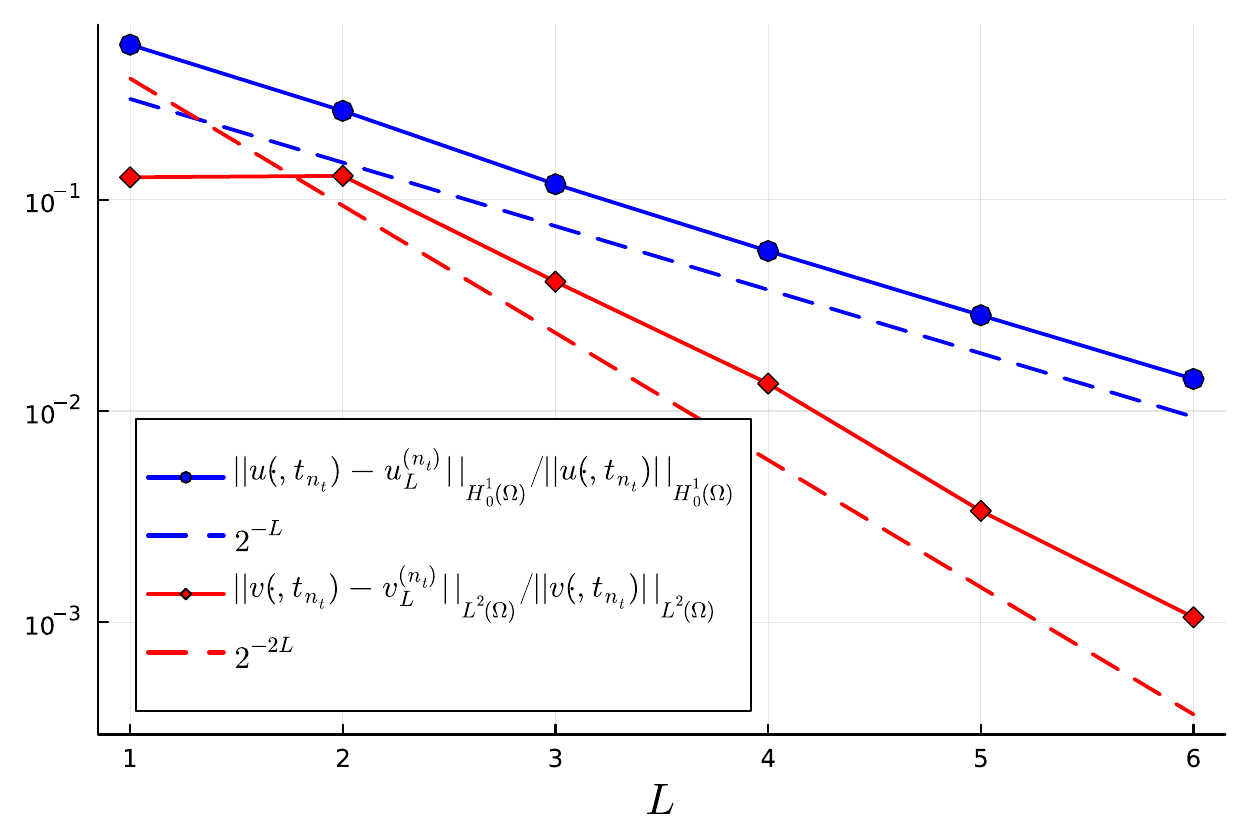}
    \caption{Horizontal axis: number of space levels $L$. Vertical axis: relative errors at the final time $t_{n_t} = T$ between the exact solution~\eqref{exact_sol_u}-\eqref{exact_sol_v} and their discrete approximations computed using Algorithm~\ref{alg:GLRK-FEM}, when compressed in the QTT format. The number $q$ of GLRK stages is $q=1$, i.e., the implicit midpoint method is employed. The position error in $\|\cdot\|_{H^1_0(\Omega)}$ (solid-blue line with round markers) decreases as $2^{-L}$ (dashed-blue line). The velocity error in $\|\cdot\|_{L^2(\Omega)}$ (solid-red line with diamond markers) decreases as $2^{-2L}$ (dashed-red line).}
    \label{fig:err_GLRK-FEM_q1}
\end{figure}

\subsubsection{Energy conservation}\label{subsec:energy_low-freq}
In this section, we demonstrate the energy conservation of the QTT-compressed GLRK--FEM. As previously mentioned, the GLRK--FEM iterations we use are energy-preserving due to the preservation of quadratic invariants guaranteed by Gauss–Legendre integrators. We aim to show that this energy conservation is not affected by the QTT compression.

As in Section~\ref{subsec:accuracy_test}, we denote with $\boldsymbol{u}^{(n)}_L \in \R^{2^L}$ and $\widehat{\boldsymbol{v}}^{(n)}_L \in \R^{2^L}$ the vectors represented in the QTT format that approximate, respectively, the exact position and exact velocity at each discrete time. The computations are carried out under the same assumptions as in Section~\ref{subsec:accuracy_test}. The discrete energy $E^{(n)}_L$~\eqref{E_dis} is computed by performing the following operation in the QTT format:
    \begin{equation*}
        E^{(n)}_L = \frac{1}{2}\boldsymbol{u}^{(n)T}_L 
        \begin{pmatrix}
        \boldsymbol{C}_L\widehat{\boldsymbol{A}}_L\boldsymbol{C}_L &  \\
         & 0\\
         \end{pmatrix}
         \boldsymbol{u}^{(n)}_L + \frac{1}{2}\widehat{\boldsymbol{v}}^{(n)T}_L 
         \begin{pmatrix}
         \widehat{\boldsymbol{M}}_L & \\
         & 0 \\
        \end{pmatrix}
        \widehat{\boldsymbol{v}}^{(n)}_L.
    \end{equation*}
The 
energy $E$~\eqref{Energy} of the exact solution~\eqref{exact_sol_u}-\eqref{exact_sol_v} satisfies $E(t)\equiv\frac{\pi^2}{2}$. 

Let $L=15$ and $q=5$. Assumption~\ref{assumpt::q} implies $n_t = 2^3$. In Figure~\ref{fig:energy}, we show the energy behavior of the GLRK--FEM iterations of Algorithm~\ref{alg:GLRK-FEM_precond} compressed in the QTT format. As one can observe, the energy is preserved with a relative error that remains bounded by $10^{-4}$. Additionally, in the first few time steps, the energy relative error is bounded by the mesh size $2^{-L}$. As we approach the final time, there is a slight increase in the discrete energy (Figure~\ref{fig:energy1}) and the corresponding relative error (Figure~\ref{fig:energy2}), which might be related to the conditioning behavior of system~\eqref{GLKR-FEM_k2_precond}. Although its condition number is lower than that of system~\eqref{GLRK-FEM_k2}, it is still of significant magnitude; see~Figure~\ref{fig:precond_q5}. 

For a comparison, we consider the GLRK--FEM with~$q=1$. In this case, no preconditioning is required and the iterations we compress in the format are those in Algorithm~\ref{alg:GLRK-FEM}. In Figure~\ref{fig:energy_q1}, we plot the relative error in the energy obtained with~$L=6$ levels and observe that it remains bounded by the square of the spatial mesh size, $2^{-2L}$, in contrast with~$2^{-L}$ obtained for~$q>1$.

\begin{figure}[h!]
    \centering
    \begin{subfigure}[b]{0.45\linewidth}
        \centering
        \includegraphics[width=\linewidth]{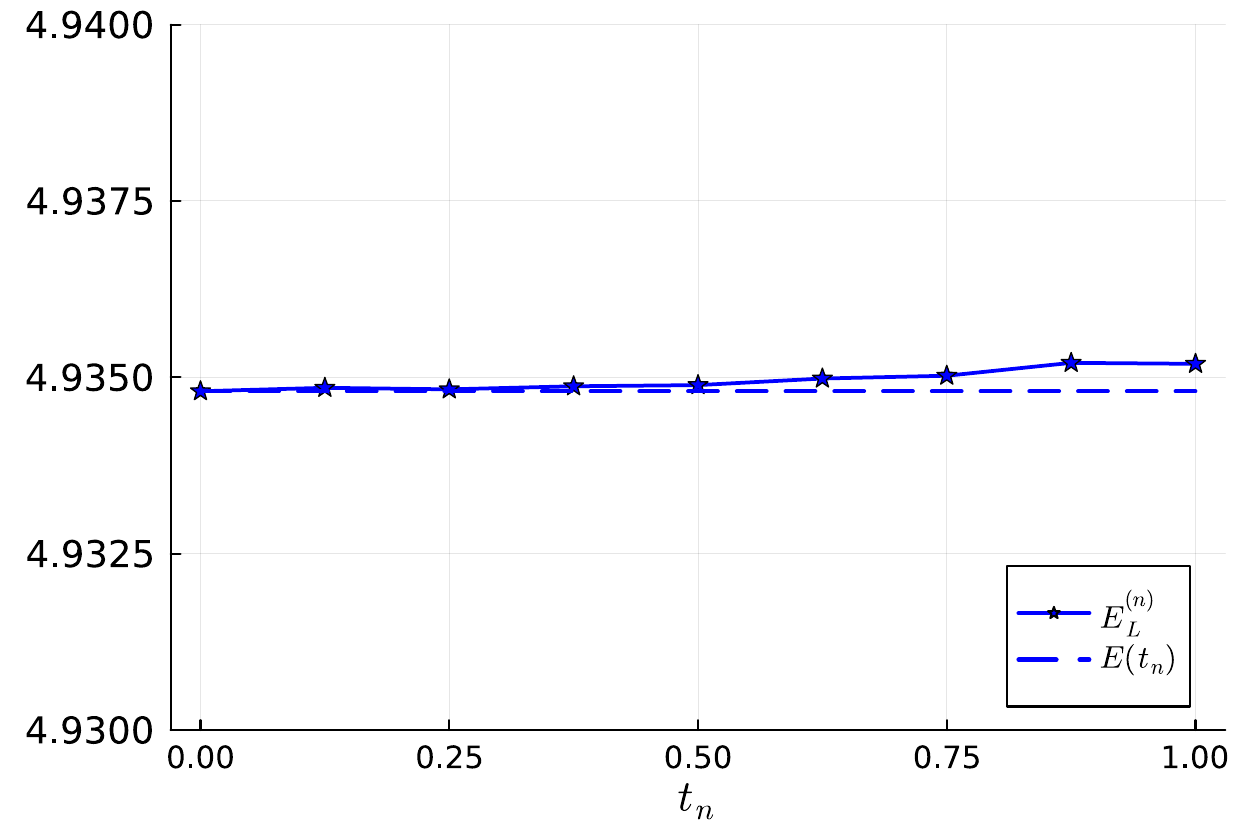}
        \caption{Energy.}
        \label{fig:energy1}
    \end{subfigure}
    \hfill
    \begin{subfigure}[b]{0.45\linewidth}
        \centering
        \includegraphics[width=\linewidth]{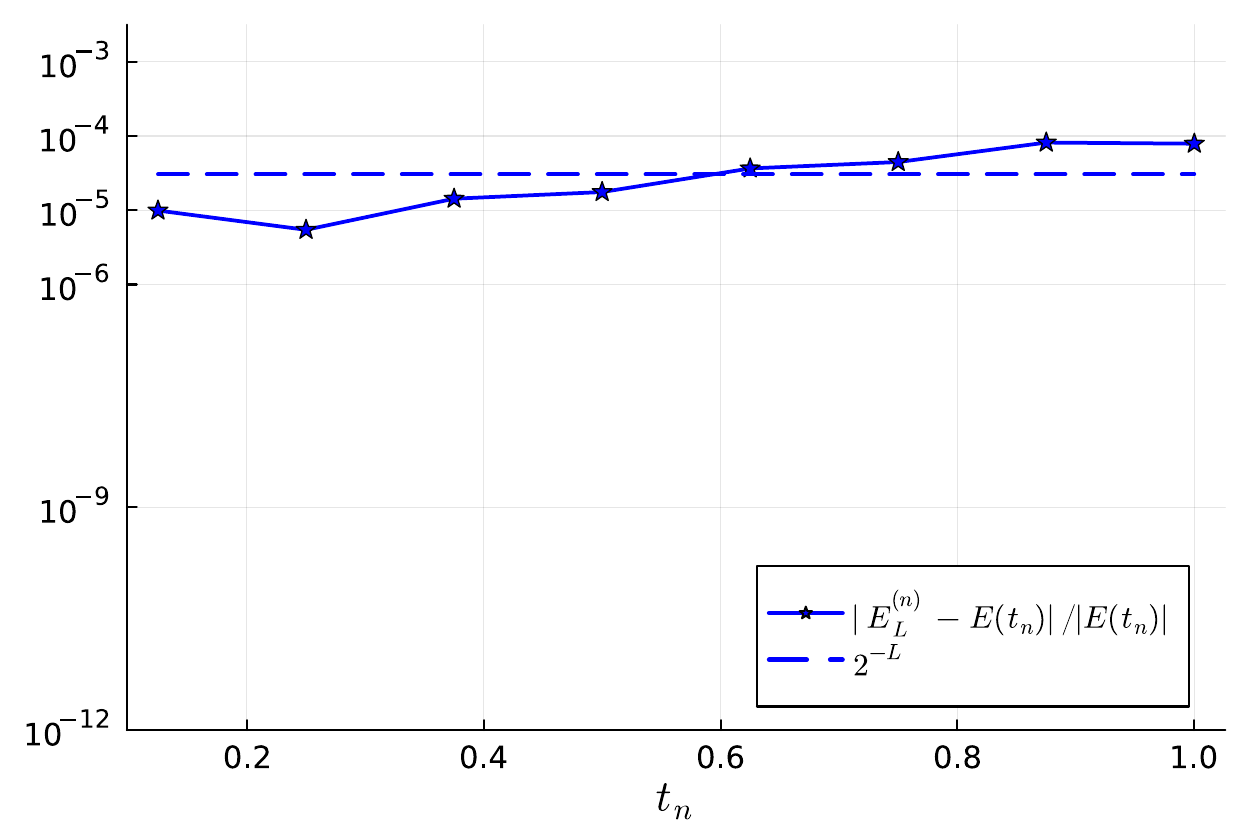}
        \caption{Energy error.}
        \label{fig:energy2}
    \end{subfigure}
    \caption{$L =15$ is the number of space levels, and $q=5$ is the number of GLRK stages. The step size $\tau$ of the GLRK--FEM iterations in Algorithm~\ref{alg:GLRK-FEM_precond} satisfies Assumption~\ref{assumpt::q}. On the horizontal axis, there are the discrete times $t_n = n\tau \in [0,1]$, $n \in \{0,\ldots,n_t\}$ with $n_t = 2^3$. Figure~\ref{fig:energy1} compares the discrete energy $E^{(n)}_L$ (solid line) associated with the QTT-compressed GLRK--FEM iterations detailed in Algorithm~\ref{alg:GLRK-FEM_precond}, and the total energy $E(t_n) = \frac{\pi^2}{2}$ (dashed line) of the exact solution~\eqref{exact_sol_u}-\eqref{exact_sol_v}. Figure~\ref{fig:energy2} shows the relative error (solid line) between the discrete energy and the exact energy. The dashed line of this figure represents the mesh size of the finite element discretization in space.}
    \label{fig:energy}
\end{figure}

\begin{figure}[htb!]
    \centering
    \includegraphics[width=0.45\linewidth]{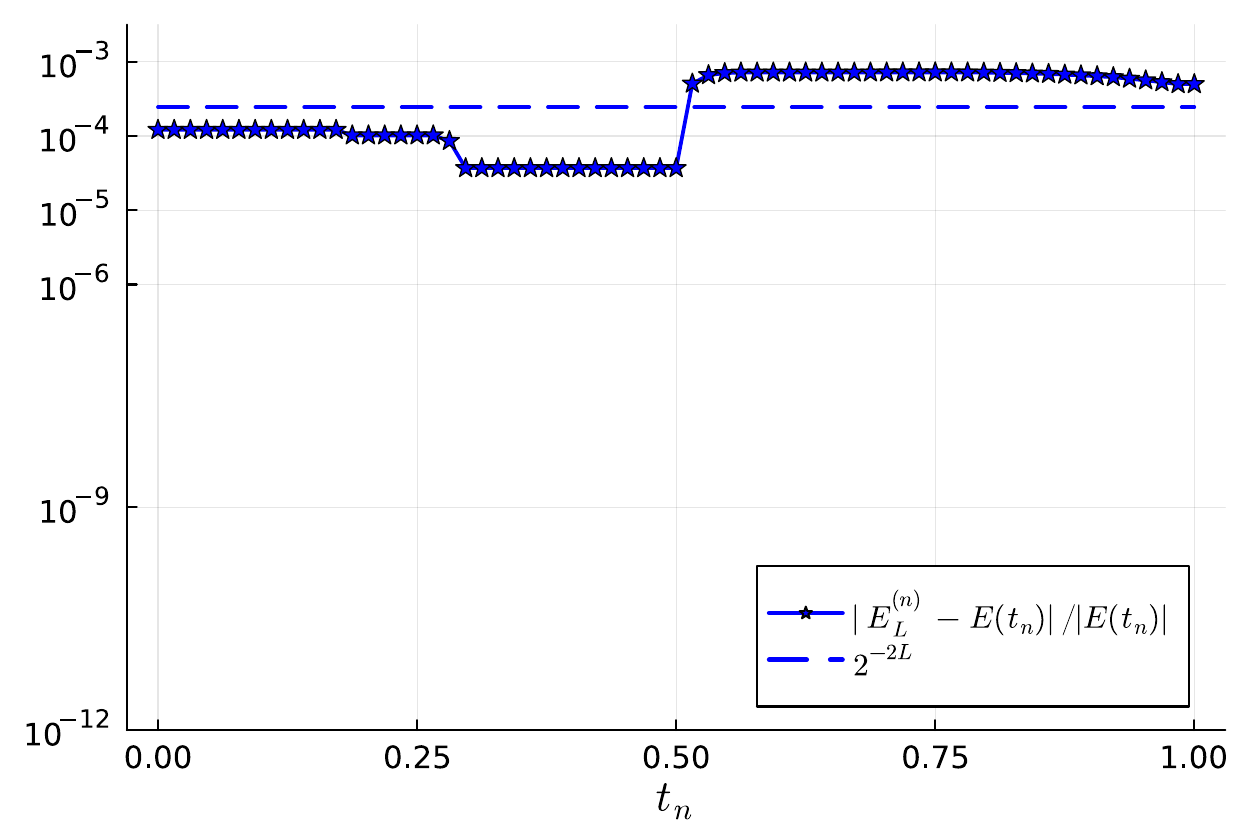}
    \caption{$L =6$ is the number of space levels, and $q=1$ is the number of GLRK stages. Accordingly, the number of time steps $n_t$ is $n_t = 2^{L} = 2^6$. On the horizontal axis, there are the discrete times $t_n = n\tau \in [0,1]$, $n \in \{0,\ldots,n_t\}$.  
    The plot shows the relative error (solid line) between the discrete energy $E^{(n)}_L$
    associated with the QTT-compressed GLRK--FEM iterations detailed in Algorithm~\ref{alg:GLRK-FEM},
    and the total energy~$E(t_n) = \frac{\pi^2}{2}$ of the exact solution~\eqref{exact_sol_u}-\eqref{exact_sol_v}. The dashed line of this figure represents the square of the mesh size of the finite element discretization in space.}
    \label{fig:energy_q1}
\end{figure}

\subsection{Example 2: high-frequency oscillations}
In this section, we test the robustness of the QTT-compressed one-stage GLRK--FEM (see Algorithm~\ref{alg:GLRK-FEM}) for high-frequency oscillations. Specifically, we consider the following solution of the wave propagation problem~\eqref{eq_wave_ham}:
\begin{subequations}\label{exact_sol_high-freq}
\begin{align}
    u(x,t) &= (\cos(k\pi t) + \sin(k\pi t)) \sin(k\pi x), &&(x,t) \in \Omega \times (0,T)=(0,1)^2,\label{exact_sol_u_high-freq}\\
    v(x,t) &= k\pi (\cos{(k\pi t)} - \sin{(k\pi t)}) \sin(k\pi x), &&(x,t) \in \Omega \times (0,T),
    \label{exact_sol_v_high-freq}
\end{align}
\end{subequations}
where $k \in \N$ determines the space and time frequencies of the corresponding wave. 
The corresponding initial data are
\begin{subequations}\label{initial_data_high-freq}
\begin{align}
    u_0(x) &= \sin{(k\pi x)}, &x \in \Omega,\\
    v_0(x) &= k\pi \sin{(k\pi x)}, &x \in \Omega.
\end{align}
\end{subequations}
For large wave numbers $k$, the QTT compression is promising, enabling us to construct finite element spaces on uniformly refined nested meshes that accurately resolve the high-frequencies, whose degrees of freedom are never accessed directly.

\begin{remark}
    In principle, performing the computations of the QTT-compressed GLRK--FEM with $q>1$ stages is a reasonable approach to handle a large number of space discretization levels, which are desirable for large values of $k$. However, the operator defined in~\eqref{precond} does not provide an optimal preconditioner. Consequently, the approximations obtained for $q>1$, and large values of $L$, are suboptimal, as already observed in Table~\ref{tab:err_GLRK-FEM}. For this reason, in this section, we restrict our focus to the results obtained for the QTT-compressed one-stage GLRK--FEM.
\end{remark}

\begin{remark}
Since solution~\eqref{exact_sol_high-freq} is defined by trigonometric functions, the corresponding initial data satisfy Theorem~\ref{th::low_rank_functions} (Example~\ref{example:low_rank_functions}), and the compatibility conditions of~\cite[Th. 8]{FP1996} hold. Furthermore, the same QTT rank bounds of  Section~\ref{sec:low-freq} hold.
\end{remark}

\subsubsection{Construction of the discrete initial data}
Let $L\in \N$ be the number of space levels. We denote with $\boldsymbol{u}_{L,0}$ and $\widehat{\boldsymbol{v}}_{L,0}$ the QTT representations of the initial data that we compute solving~\eqref{lin_syst_u0} and \eqref{linear_syst_v0} in the QTT format. These computations are performed as described in Section~\ref{subsec:initial_data_computation}. Figure~\ref{fig:QTT_initial_data_high-freq} shows the relative errors obtained from these computations when $k = 10$. We remark that the minimum $L$ we consider is chosen depending on~$k$. 
As shown, exponential convergence in $L$ 
is achieved for these QTT-compressed projections of the high-frequency functions defined in~\eqref{initial_data_high-freq}, as in the low-frequency case (see Section~\ref{subsec:initial_data_computation}).

\begin{figure}[h!]
    \centering
    \includegraphics[width=0.6\linewidth]{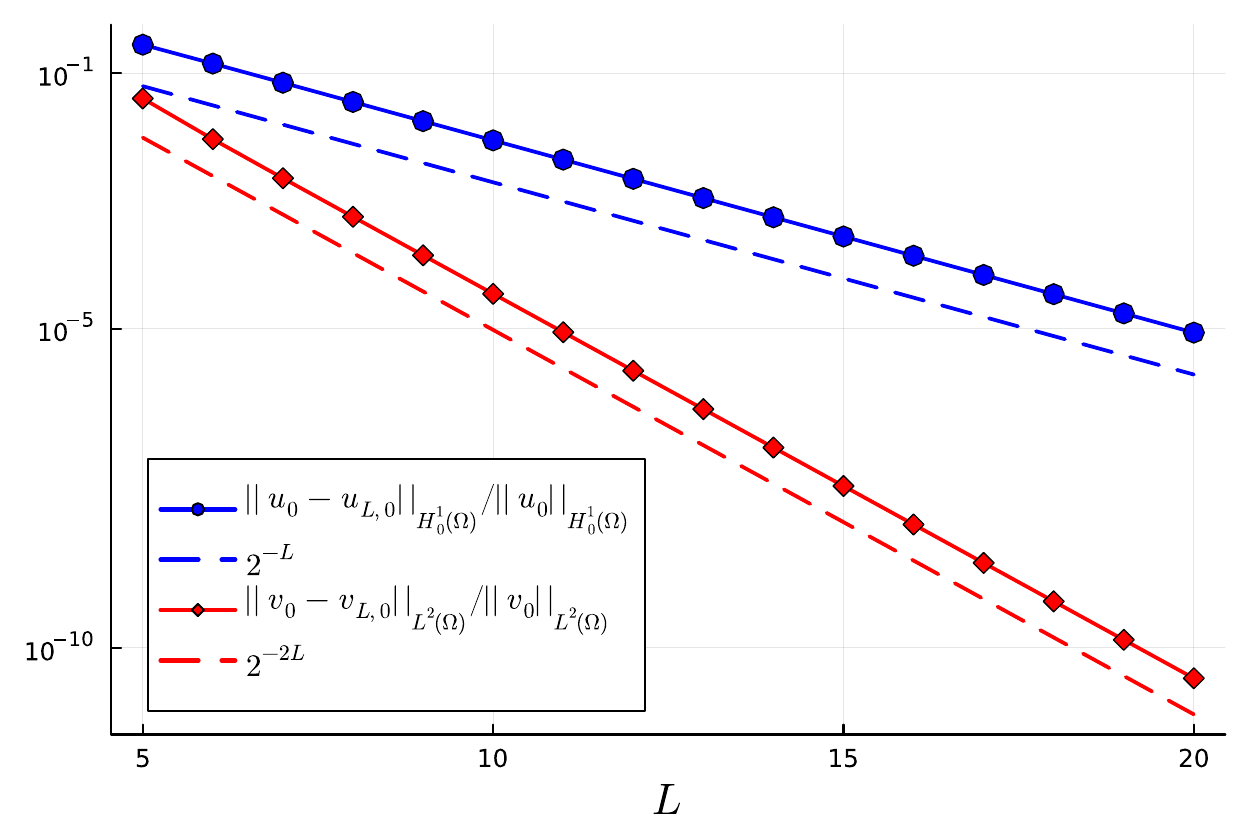}
    \caption{Horizontal axis: number of space levels L. Vertical axis: relative errors between the continuous initial data~\eqref{initial_data_high-freq} with $k=10$, and their approximations from~\eqref{lin_syst_u0} and \eqref{linear_syst_v0} when compressed in the QTT format. The error of the initial position in $\|\cdot\|_{H^1_0(\Omega)}$ (solid-blue line with round markers) decreases as $2^{-L}$ (dashed-blue line). The error of the initial velocity in $\|\cdot\|_{L^2(\Omega)}$(solid-red line with diamond markers) decreases as $2^{-2L}$ (dashed-red line).}
    \label{fig:QTT_initial_data_high-freq}
\end{figure}

\subsubsection{Accuracy in the energy norm}\label{subsec:acc_high-freq}
In this section, we investigate the accuracy of the QTT-compressed one-stage GLRK--FEM approximating the solution~\eqref{exact_sol_high-freq} for different values of $k \in \N$.

Let $L \in \N$ represent the number of space levels, and $n_t$ be the number of time steps in the one-stage GLRK method. At each discrete time $t_n = n\tau$, the numerical solution is given by
\begin{equation}\label{err_L2H10_high-freq}
\begin{aligned}
    u^{(n)}_L(x)&=\sum_{i=1}^{N_L} \boldsymbol{u}_L^{(n)}[i] \ \varphi_{L,i}(x), \quad &&x \in \Omega,\\
    v^{(n)}_L(x)&=\sum_{j=1}^{N_L} \widehat{\boldsymbol{v}}_L^{(n)}[j] \ \widehat{\varphi}_{L,j}(x), \quad &&x \in \Omega,
\end{aligned}
\end{equation}
with $N_L=2^L-1$, where $\boldsymbol{u}^{(n)}_L \in \R^{2^L}$ and $\widehat{\boldsymbol{v}}^{(n)}_L \in \R^{2^L}$ denote the QTT-compressed coefficients of the finite element solutions at each discrete time $t_n = n\tau$ for $n \in \{0,\ldots,n_t\}$. The same assumptions as in Section~\ref{subsec:accuracy_test} hold for the computations. Figure~\ref{fig:err_high-freq_q1} shows the relative errors~\eqref{err_norms} at the final time $t_{n_t} = T$ of the QTT-compressed GLRK--FEM solutions~\eqref{err_L2H10_high-freq}, plotted against the number of levels $L$ for different wave numbers $k \in \{4,6,8,10\}$.
In each experiment, the minimum number of levels is chosen depending on~$k$.
We see that optimal convergence rates are obtained for both position and velocity, for each wave number. 
We remark that employing QTT-compressed GLRK--FEM with $q > 1$ could enable significantly larger values of the space levels $L$, leading to increased accuracy. However, this approach requires a robust preconditioner, which is currently unavailable and is left to future work.

\begin{figure}[htbp!]
    \centering
    \begin{subfigure}[b]{0.45\linewidth}
        \centering
        \includegraphics[width=\textwidth]{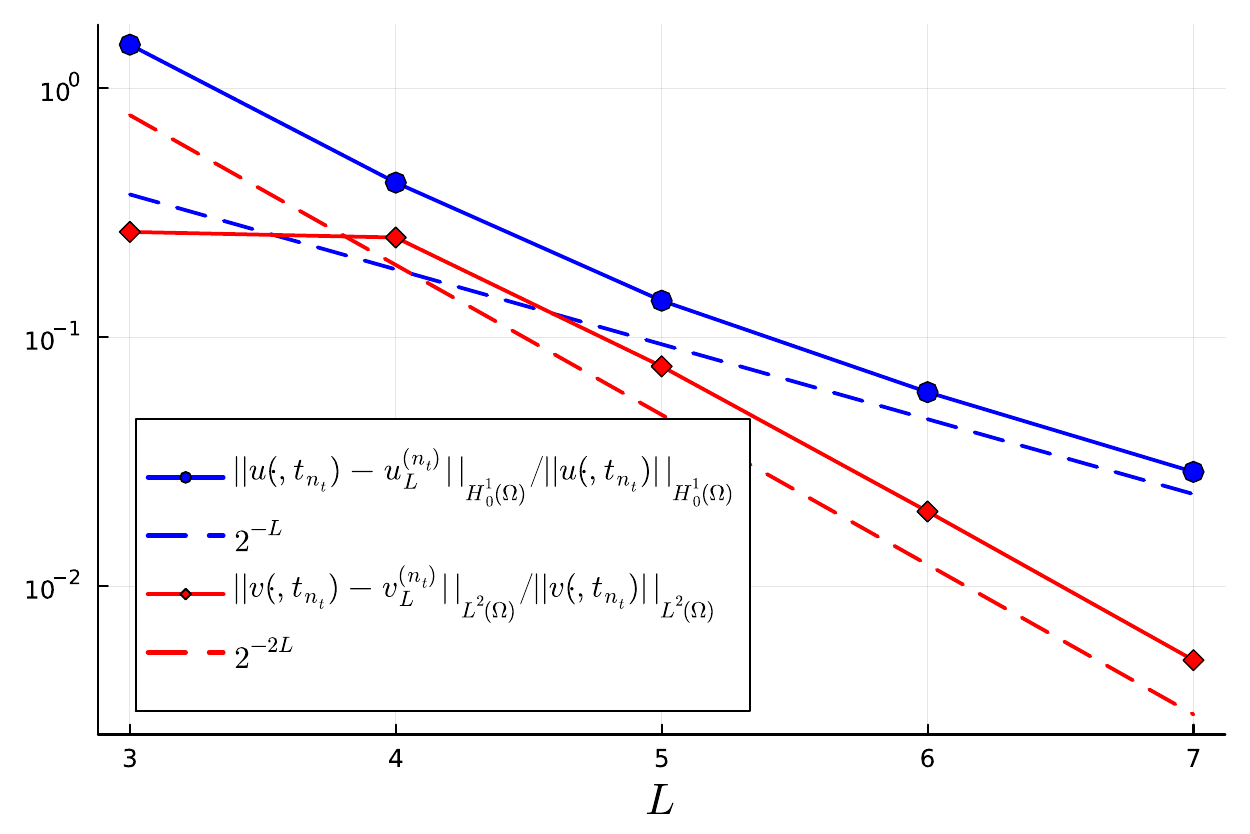}
        \caption{Wave number $k=4$.}
        \label{}
    \end{subfigure}
    \hfill
    \begin{subfigure}[b]{0.45\linewidth}
        \centering
        \includegraphics[width=\textwidth]{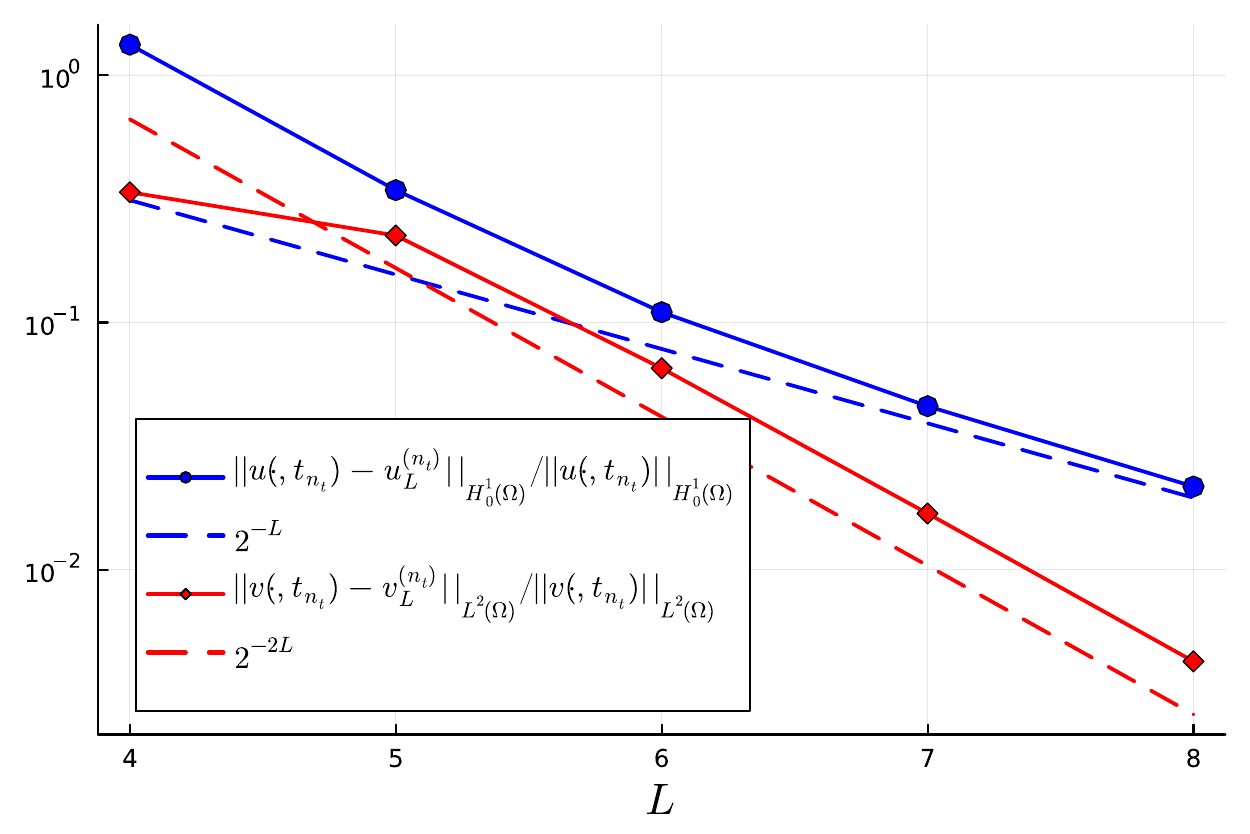}
        \caption{Wave number $k=6$.}
        \label{}
    \end{subfigure} \\[1em]
    
    \begin{subfigure}[b]{0.45\linewidth}
        \centering
        \includegraphics[width=\textwidth]{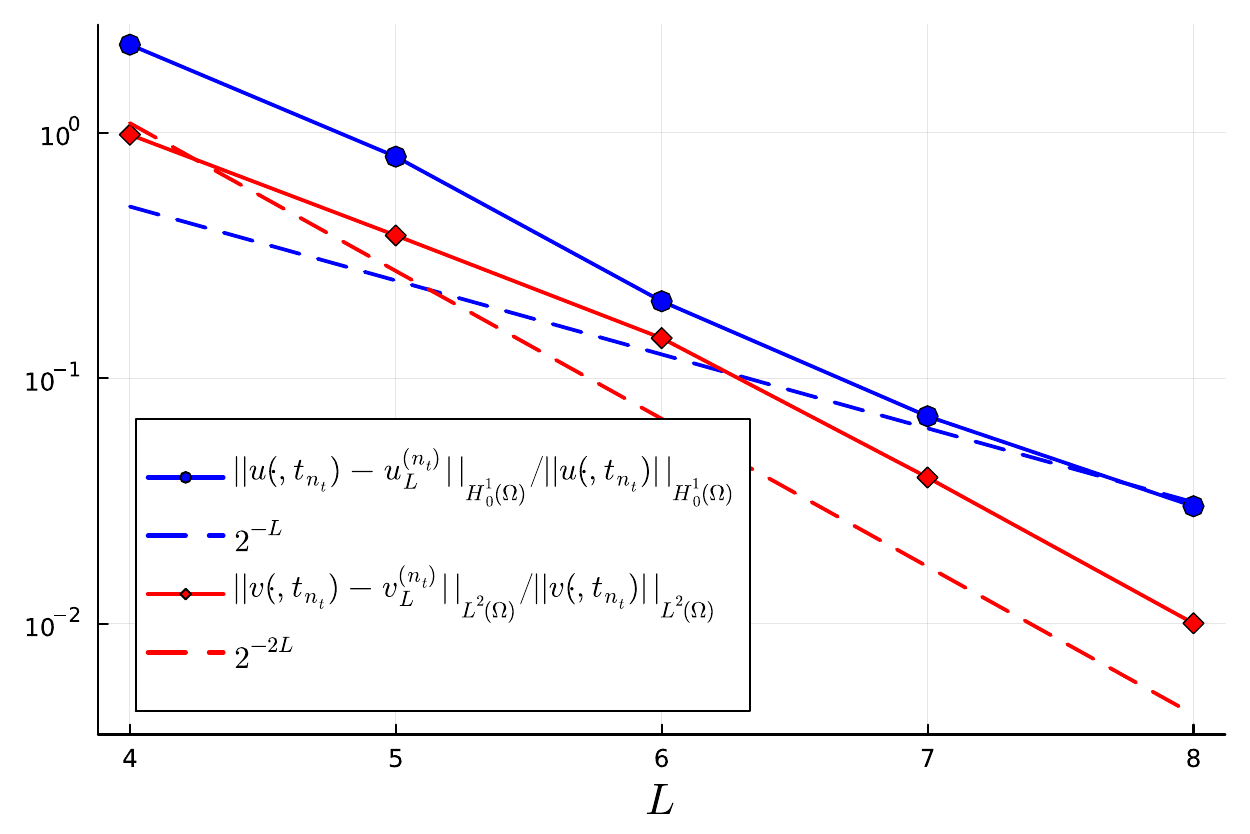}
        \caption{Wave number $k=8$.}
        \label{}
    \end{subfigure}
    \hfill
    \begin{subfigure}[b]{0.45\linewidth}
        \centering
        \includegraphics[width=\textwidth]{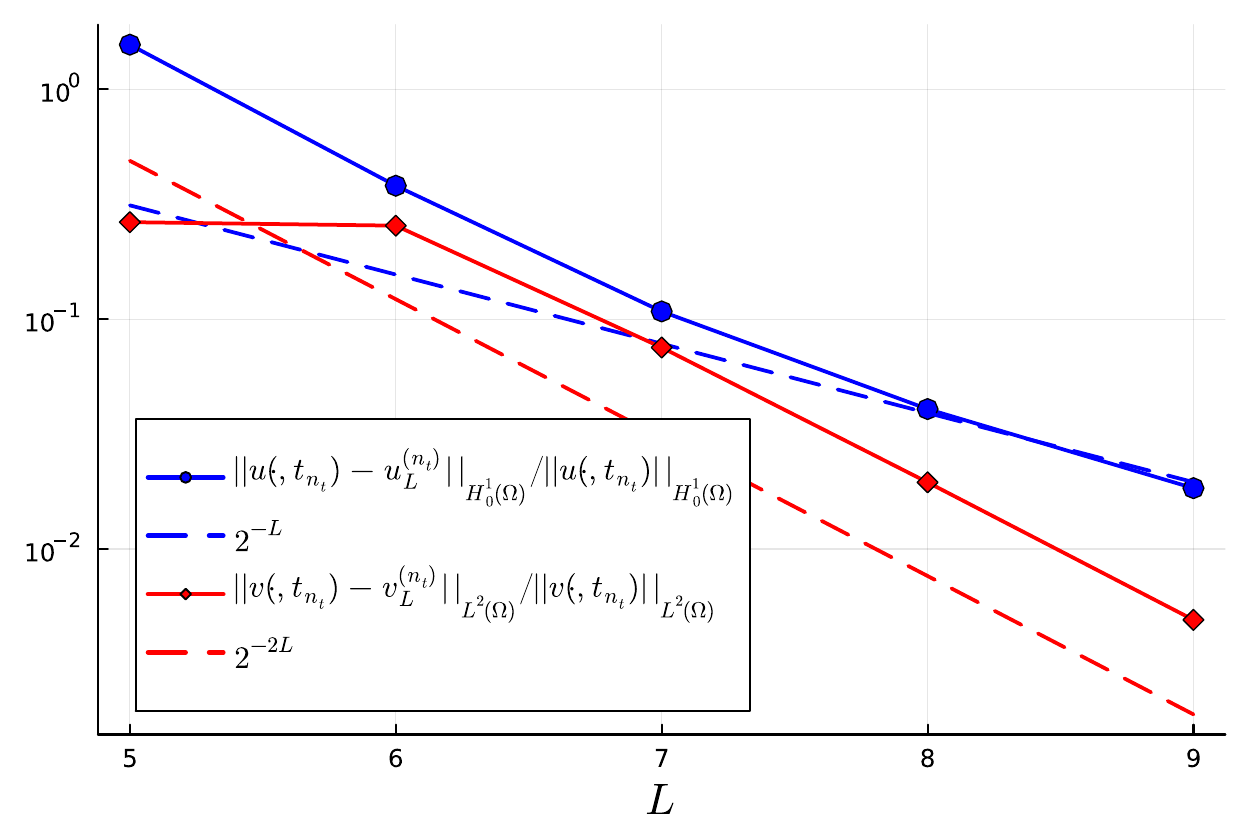}
        \caption{Wave number $k=10$.}
        \label{}
    \end{subfigure} \\[1em]
    \caption{Each plot corresponds to a different wave number $k\in\{4,6,8,10\}$, and illustrates the relative errors at the final time $t_{n_t} = T$ between the exact solution~\eqref{exact_sol_high-freq} and its QTT-compressed approximation, computed using Algorithm~\ref{alg:GLRK-FEM} in the QTT format. The implicit midpoint method is employed. The horizontal axis represents the number of space levels $L$. For each wave number, the convergence rates are as expected. The position error in $\|\cdot\|_{H^1_0(\Omega)}$ (solid-blue line with round markers) decreases as $2^{-L}$ (dashed-blue line), while the velocity error in $\|\cdot\|_{L^2(\Omega)}$ (solid-red line with diamond markers) decreases as $2^{-2L}$ (dashed-red line).}
    \label{fig:err_high-freq_q1}
\end{figure}

\subsubsection{Energy conservation}
In this section, we demonstrate the energy conservation of the QTT-compressed one-stage GLRK--FEM approximating solution~\eqref{exact_sol_high-freq} for different values of $k \in \N$. For each $k \in \N$, the 
energy $E$~\eqref{Energy} of the exact solution~\eqref{exact_sol_high-freq} satisfies $E(t)\equiv\frac{k^2\pi^2}{2}$. 

As in Section~\ref{subsec:acc_high-freq}, we take 
$k \in \{4,6,8,10\}$. We compute the discrete energy as described in Section~\ref{subsec:energy_low-freq}. For each~$k$, the number of space levels $L$ for the corresponding finite element discretization is the highest one
we have considered in Figure~\ref{fig:err_high-freq_q1}. Assumption~\ref{assumpt::q} implies $n_t = 2^L$. We see in Figure~\ref{energy_high-freq_k810} that the discrete energy is preserved without oscillations, except for the ones due to machine precision. Indeed, the relative errors in the discrete energy we have observed for~$k \in \{4,6,8,10\}$ are, approximately,~$10^{-12}$. 

The relative error is bounded from above by the mesh size $2^{-L}$, and from below by its square $2^{-2L}$. No error blow-up occurs while approaching the final time.

\begin{figure}
    \centering
        \begin{subfigure}[b]{0.45\linewidth}
        \centering
        \includegraphics[width=0.8\textwidth]{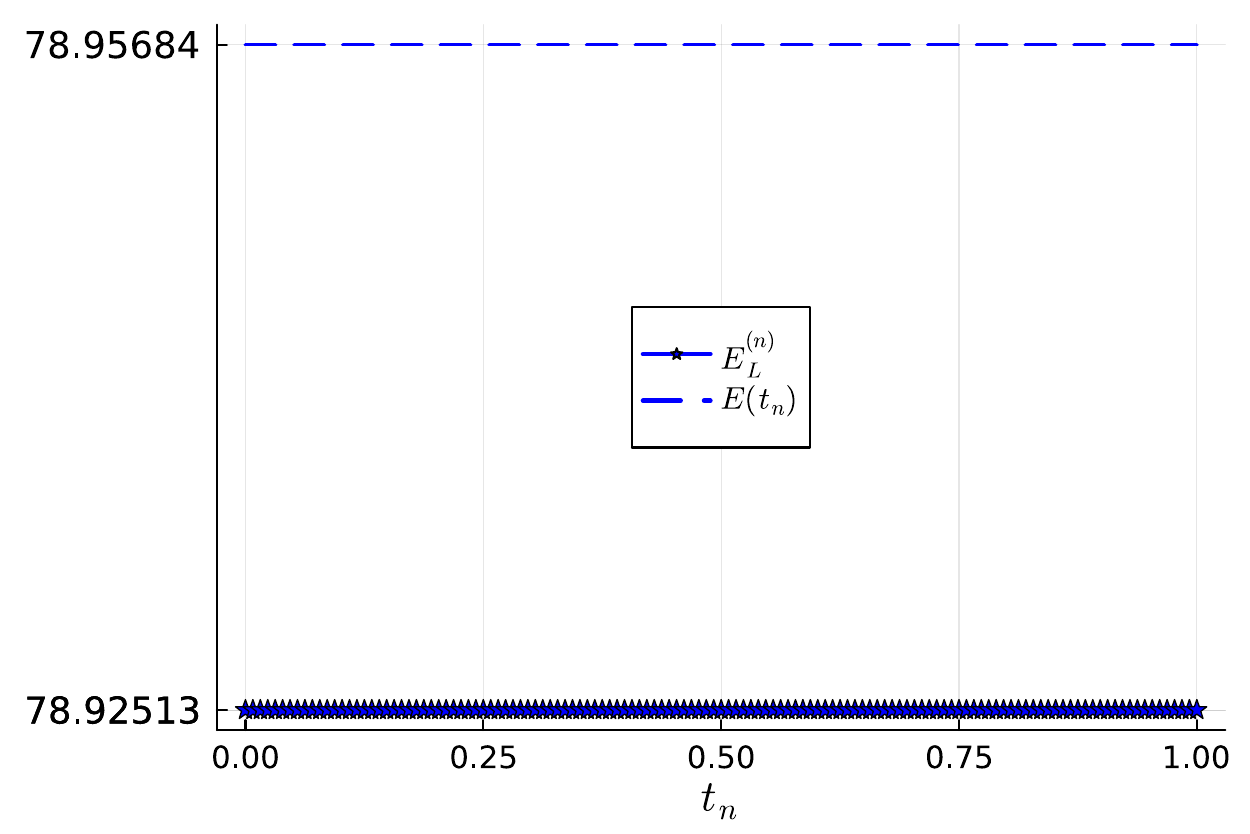}
        \caption{Exact and discrete energy for wave number $k=4$, and space levels $L=7$.}
        \label{}
    \end{subfigure}
    \hfill
    \begin{subfigure}[b]{0.45\linewidth}
        \centering
        \includegraphics[width=0.8\textwidth]{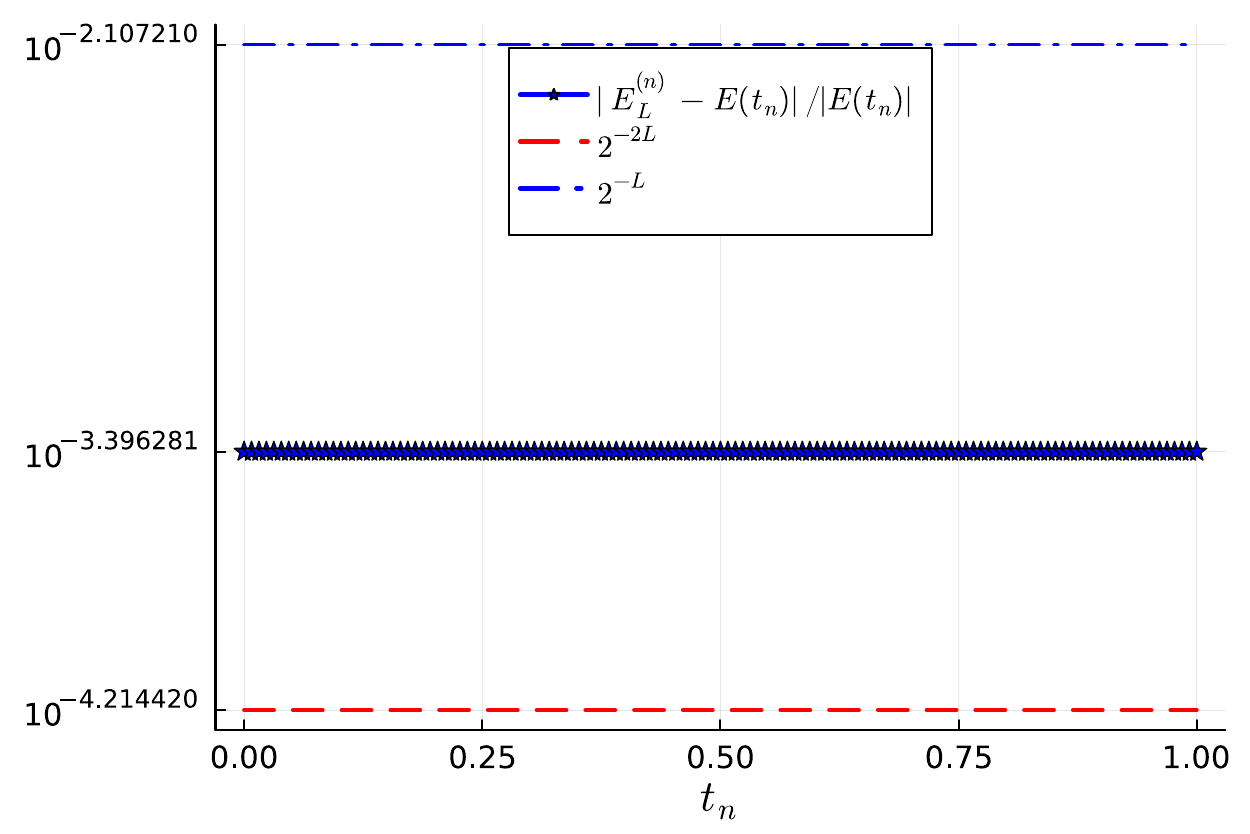}
        \caption{Energy error for wave number $k=4$, and space levels $L=7$.}
        \label{}
    \end{subfigure} \\[1em]
    \begin{subfigure}[b]{0.45\linewidth}
        \centering
        \includegraphics[width=0.8\textwidth]{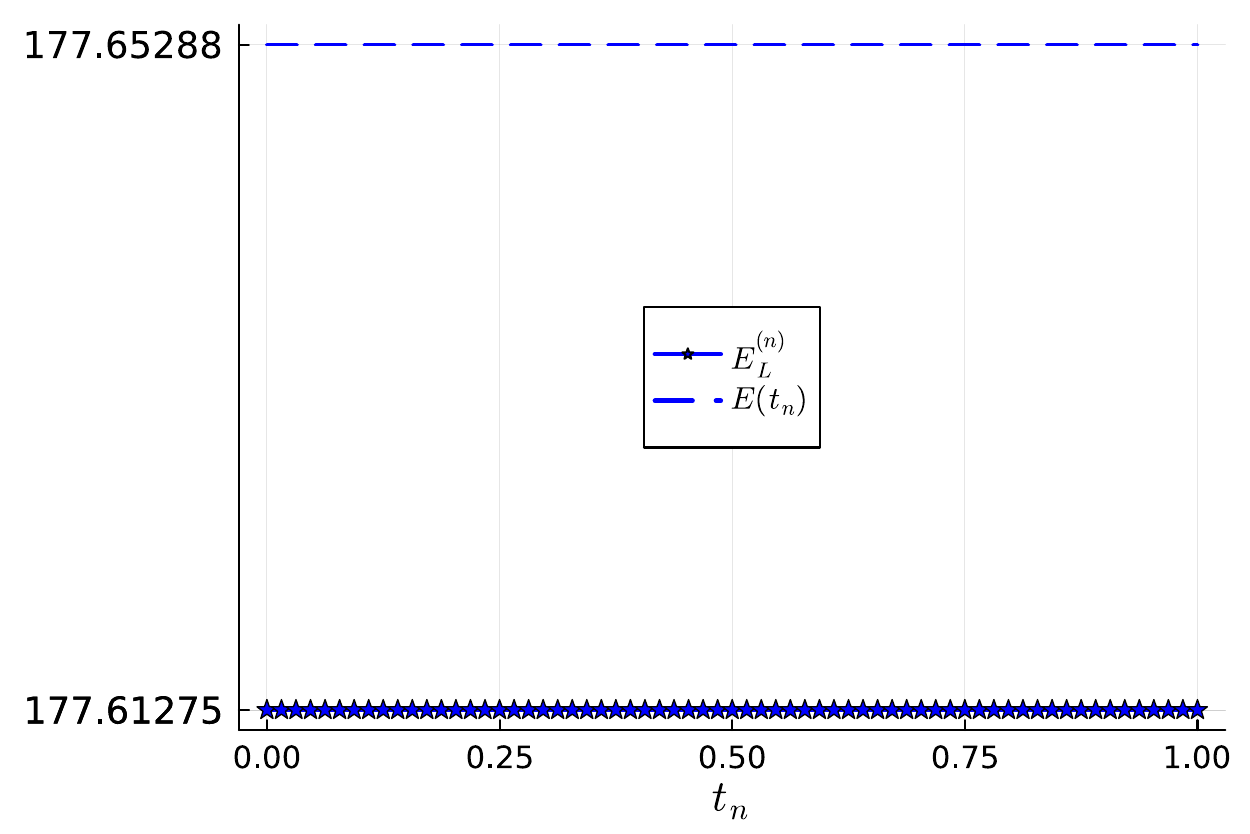}
        \caption{Exact and discrete energy for wave number $k=6$, and space levels $L=8$.}
        \label{}
    \end{subfigure}
    \hfill
    \begin{subfigure}[b]{0.45\linewidth}
        \centering
        \includegraphics[width=0.8\textwidth]{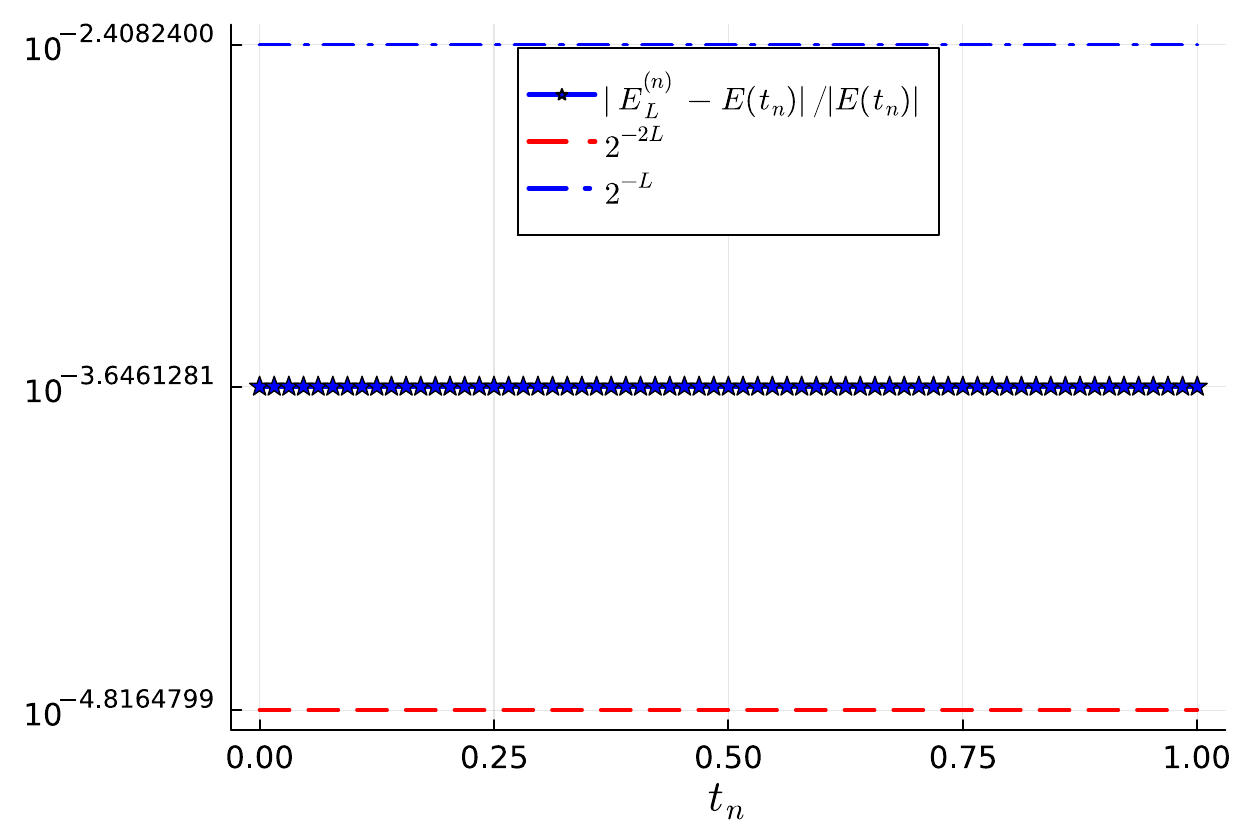}
        \caption{Energy error for wave number $k=6$, and space levels $L=8$.}
        \label{}
    \end{subfigure}\\[1em]

    \begin{subfigure}[b]{0.45\linewidth}
        \centering
        \includegraphics[width=0.8\textwidth]{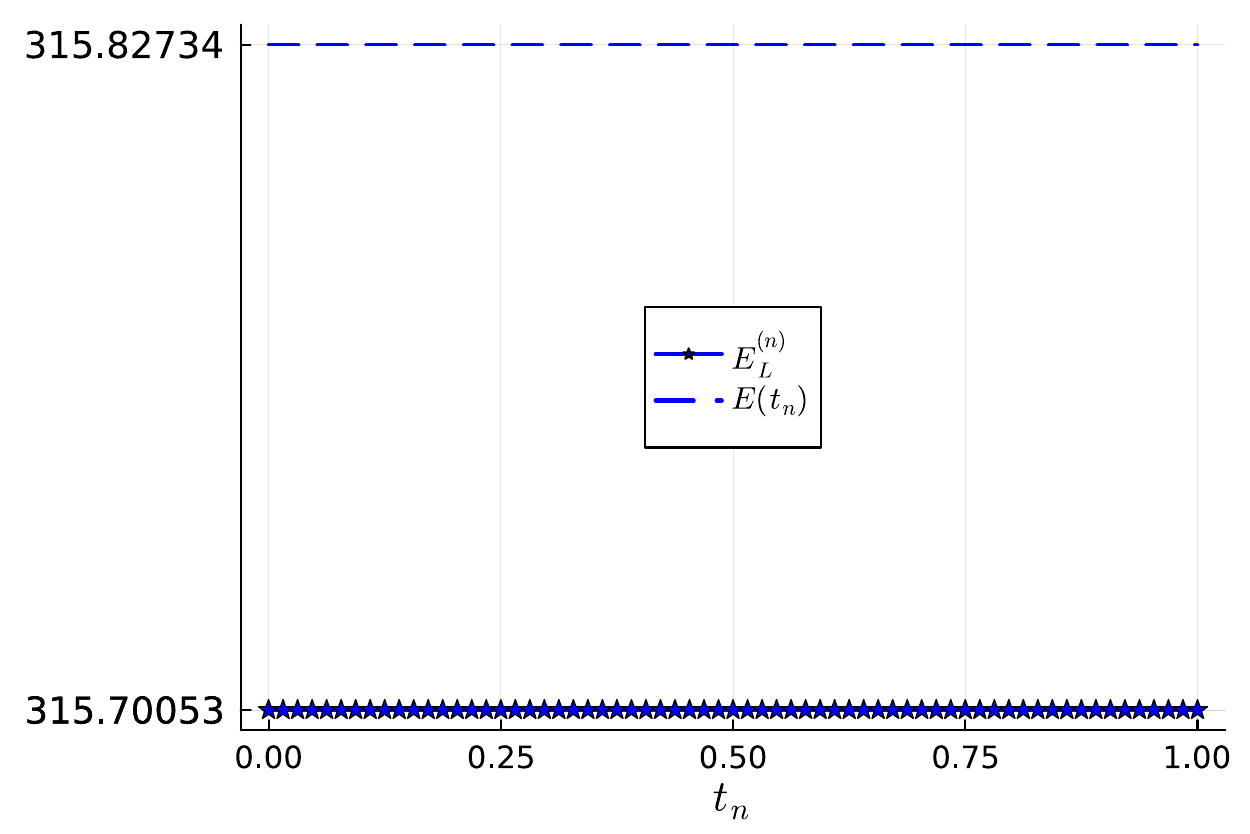}
        \caption{Exact and discrete energy for wave number $k=8$, and space levels $L=8$.}
        \label{}
    \end{subfigure}
    \hfill
    \begin{subfigure}[b]{0.45\linewidth}
        \centering
        \includegraphics[width=0.8\textwidth]{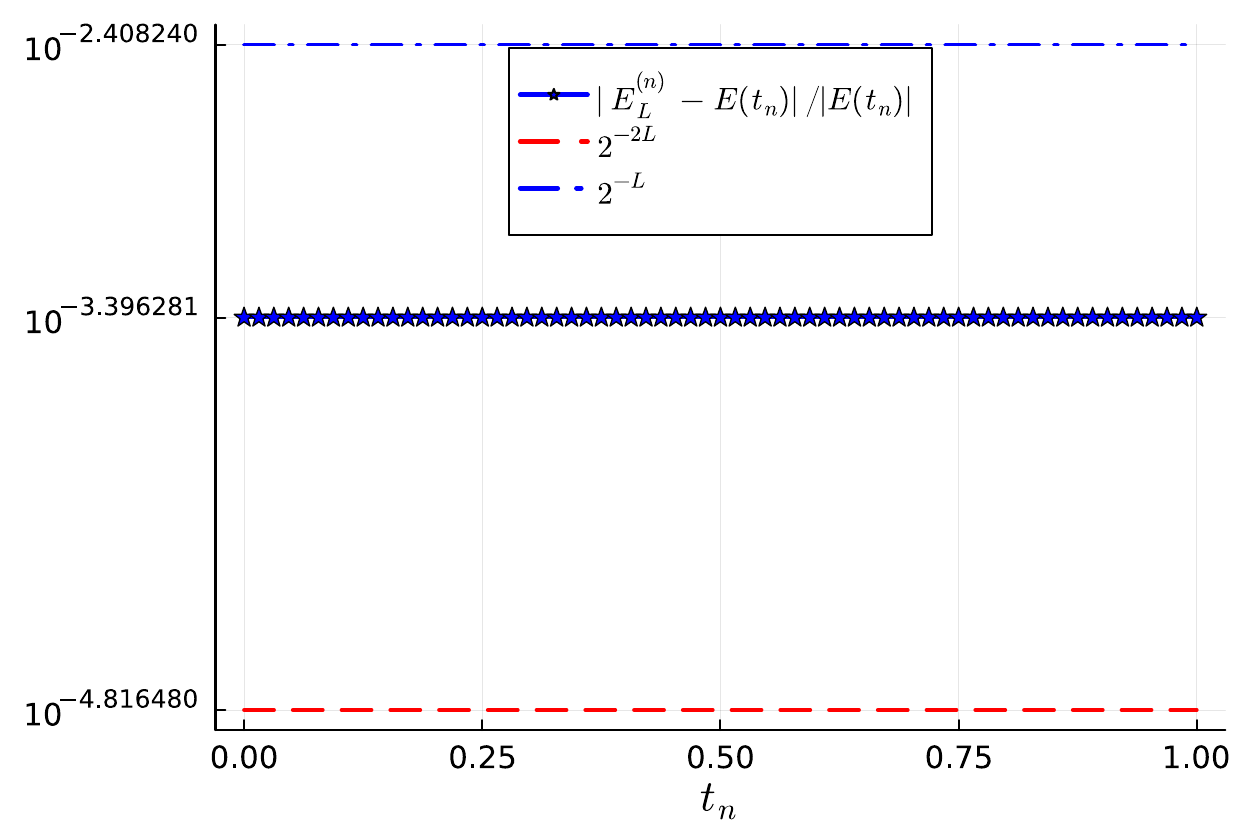}
        \caption{Energy error for wave number $k=8$, and space levels $L=8$.}
        \label{}
    \end{subfigure} \\[1em]

    \begin{subfigure}[b]{0.45\linewidth}
        \centering
        \includegraphics[width=0.8\textwidth]{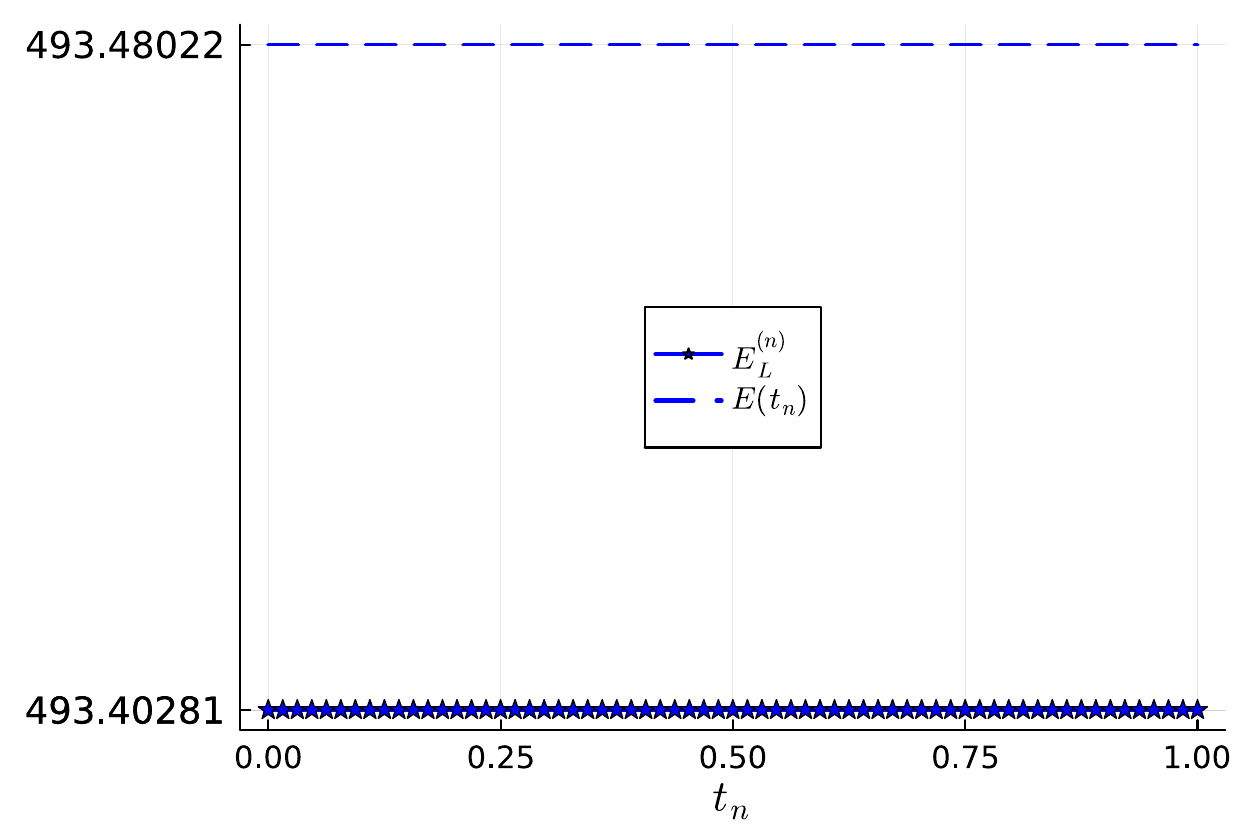}
        \caption{Exact and discrete energy for wave number $k=10$, and space levels $L=9$.}
        \label{}
    \end{subfigure}
    \hfill
    \begin{subfigure}[b]{0.45\linewidth}
        \centering
        \includegraphics[width=0.8\textwidth]{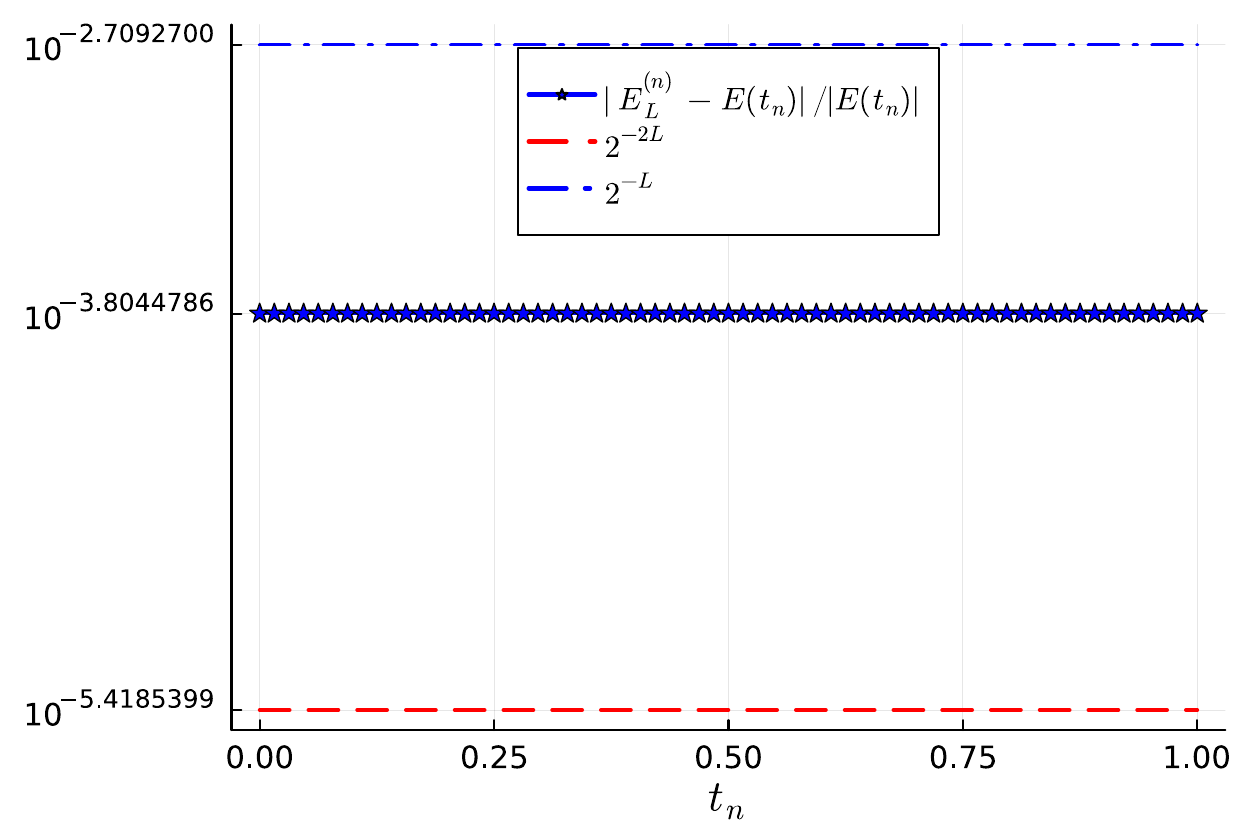}
        \caption{Energy error for wave number $k=10$, and space levels $L=9$.}
        \label{}
    \end{subfigure}
    \caption{On the horizontal axis, there are discrete times $t_n = n\tau \in [0,1]$ of the one-stage GLRK--FEM~\ref{alg:GLRK-FEM}. In the left column, we compare the discrete energy $E^{(n)}_L$ (solid line) and the total energy $E(t_n)$ (dashed line) of the exact solution~\eqref{exact_sol_high-freq} for $k\in\{4,6,8,10\}$. In the right column, we show the relative error (solid line) between the discrete energy and the exact energy. The dashed-dot lines and the dashed lines in the right column are, respectively, the mesh sizes and the square of the mesh sizes of the FEM.}
    \label{energy_high-freq_k810}
\end{figure}

\section{Conclusion}
A symplectic QTT-structured FEM for the one-dimensional acoustic wave equation in the time domain was presented. Our model problem was the first-order-in-time formulation of the wave equation as in~\cite{BL1994,FP1996}. The numerical method we devised combines implicit GLRK time integrators with QTT-compressed, uniform, piecewise linear FEM in space. Numerical tests were conducted for low-frequency and high-frequency, time-harmonic standing waves, confirming the unconditional stability and energy conservation of the method.
The linear system at each time step of the single-stage GLRK scheme (i.e., the midpoint method) is
uniformly well-conditioned in the number of the space discretization levels, and exponential convergence with respect to the number of space levels is
achieved. For higher-order GLRK integrators, we proposed a stabilization technique based on the BPX preconditioner introduced in~\cite{BK2020}. As a result, reasonable approximations were observed, although optimal convergence rates were not attained in this case. Developing an optimal preconditioner is thus essential and is left to future work.

The devised QTT-compressed numerical method was implemented in one space dimension, Dirichlet boundary conditions, and constant wave speed. Our next research step is to extend the QTT-compressed method to wave problems in two space dimensions with 
more general boundary conditions and non-constant coefficients. In this case, the solutions may exhibit spatial point singularities~\cite{BMPS2021}
and leveraging QTT-structured techniques could be beneficial,
as demonstrated in~\cite{KS2018} for elliptic problems.

\section{Acknowledegments}
This research was supported by the Austrian Science Fund (FWF) project \href{https://doi.org/10.55776/F65}{10.55776/F65} (SF, VK, IP) and project \href{https://doi.org/10.55776/P33477}{10.55776/P33477} (SF, IP). SF was also supported by the Vienna School of Mathematics.

\bibliography{mybibliography}{}

\begin{thebibliography}{10}

\bibitem{AEEV2012}
A.~Abdulle, W.~E, B.~Engquist, and E.~Vanden-Eijnden.
\newblock The heterogeneous multiscale method.
\newblock {\em Acta Numer.}, 21:1--87, 2012.

\bibitem{AHP2021}
R.~Altmann, P.~Henning, and D.~Peterseim.
\newblock Numerical homogenization beyond scale separation.
\newblock {\em Acta Numer.}, 30:1--86, 2021.

\bibitem{B2023}
M.~Bachmayr.
\newblock Low-rank tensor methods for partial differential equations.
\newblock {\em Acta Numer.}, 32:1--121, 2023.

\bibitem{BK2020}
M.~Bachmayr and V.~Kazeev.
\newblock Stability of low-rank tensor representations and structured
  multilevel preconditioning for elliptic {PDE}s.
\newblock {\em Found. Comput. Math.}, 20(5):1175--1236, 2020.

\bibitem{BL1994}
L.~Bales and I.~Lasiecka.
\newblock Continuous finite elements in space and time for the nonhomogeneous
  wave equation.
\newblock {\em Comput. Math. Appl.}, 27(3):91--102, 1994.

\bibitem{BMPS2021}
P.~Bansal, A.~Moiola, I.~Perugia, and Ch. Schwab.
\newblock Space-time discontinuous {G}alerkin approximation of acoustic waves
  with point singularities.
\newblock {\em IMA J. Numer. Anal.}, 41(3):2056--2109, 2021.

\bibitem{Bellman:1961}
R.~Bellman.
\newblock {\em Adaptive Control Processes: A Guided Tour}.
\newblock Princeton University Press, Princeton, NJ, 1961.

\bibitem{BPX1990}
J.~H. Bramble, J.~E. Pasciak, and J.~Xu.
\newblock Parallel multilevel preconditioners.
\newblock {\em Math. Comp.}, 55(191):1--22, 1990.

\bibitem{TensorRefinement.jl}
TensorRefinement Contributors.
\newblock Tensor{R}efinement.jl, a {J}ulia package for low-rank computations
  with functions.
\newblock \url{https://github.com/TensorRefinement/TensorRefinement.jl}, 2020.

\bibitem{DL1992}
R.~Dautray and J-L. Lions.
\newblock {\em Mathematical analysis and numerical methods for science and
  technology. {V}ol. 5}.
\newblock Springer-Verlag, Berlin, 1992.

\bibitem{dLS1994}
W.~de~Launey and J.~Seberry.
\newblock The strong {K}ronecker product.
\newblock {\em J. Combin. Theory Ser. A}, 66(2):192--213, 1994.

\bibitem{DKO2012}
S.~V. Dolgov, B.~N. Khoromskij, and I.~V. Oseledets.
\newblock Fast solution of parabolic problems in the tensor train/quantized
  tensor train format with initial application to the {F}okker-{P}lanck
  equation.
\newblock {\em SIAM J. Sci. Comput.}, 34(6):A3016--A3038, 2012.

\bibitem{QTTwave.jl}
S.~Fraschini.
\newblock {QTT}wave.jl.
\newblock \url{https://github.com/sfrasch/QTTwave}, 2024.

\bibitem{FP1996}
D.~A. French and T.~E. Peterson.
\newblock A continuous space-time finite element method for the wave equation.
\newblock {\em Math. Comp.}, 65(214):491--506, 1996.

\bibitem{GVL2013}
G.~H. Golub and C.~F. Van~Loan.
\newblock {\em Matrix computations}.
\newblock Johns Hopkins Studies in the Mathematical Sciences. Johns Hopkins
  University Press, Baltimore, MD, fourth edition, 2013.

\bibitem{G2010}
L.~Grasedyck.
\newblock {\em Polynomial approximation in hierarchical Tucker format by
  vector-tensorization}.
\newblock Inst. f{\"u}r Geometrie und Praktische Mathematik, 2010.

\bibitem{GKT2013}
L.~Grasedyck, D.~Kressner, and C.~Tobler.
\newblock A literature survey of low-rank tensor approximation techniques.
\newblock {\em GAMM-Mitt.}, 36(1):53--78, 2013.

\bibitem{GB1986a}
B.~Guo and I.~Babu{\v{s}}ka.
\newblock The hp version of the finite element method - {P}art 1: {T}he basic
  approximation results.
\newblock {\em Comput. Mech.}, 1(1):21--41, 1986.

\bibitem{GB1986b}
B.~Guo and I.~Babu{\v{s}}ka.
\newblock The hp version of the finite element method - {P}art 2: {G}eneral
  results and applications.
\newblock {\em Comput. Mech.}, 1(3):203--220, 1986.

\bibitem{H2014}
W.~Hackbusch.
\newblock Numerical tensor calculus.
\newblock {\em Acta Numer.}, 23:651--742, 2014.

\bibitem{H2019}
W.~Hackbusch.
\newblock {\em Tensor spaces and numerical tensor calculus}, volume~56 of {\em
  Springer Series in Computational Mathematics}.
\newblock Springer, Cham, second edition, 2019.

\bibitem{HLW2006}
E.~Hairer, C.~Lubich, and G.~Wanner.
\newblock {\em Geometric numerical integration}, volume~31 of {\em Springer
  Series in Computational Mathematics}.
\newblock Springer-Verlag, Berlin, second edition, 2006.

\bibitem{HH1955}
P.~C. Hammer and J.~W. Hollingsworth.
\newblock Trapezoidal methods of approximating solutions of differential
  equations.
\newblock {\em Math. Tables Aids Comput.}, 9:92--96, 1955.

\bibitem{K2015}
V.~Kazeev.
\newblock {\em Quantized tensor structured finite elements for second-order
  elliptic PDEs in two dimensions}.
\newblock Ph{D} thesis, ETH Z{\"u}rich, 2015.

\bibitem{KKNS2014}
V.~Kazeev, M.~Khammash, M.~Nip, and Ch. Schwab.
\newblock Direct solution of the chemical master equation using quantized
  tensor trains.
\newblock {\em {PLOS} computational biology}, 10(3):742--758, 2014.

\bibitem{KK2012}
V.~Kazeev and B.~Khoromskij.
\newblock Low-rank explicit {QTT} representation of the {L}aplace operator and
  its inverse.
\newblock {\em SIAM J. Matrix Anal. Appl.}, 33(3):742--758, 2012.

\bibitem{KORS2017}
V.~Kazeev, I.~Oseledets, M.~Rakhuba, and Ch. Schwab.
\newblock Q{TT}-finite-element approximation for multiscale problems {I}: model
  problems in one dimension.
\newblock {\em Adv. Comput. Math.}, 43(2):411--442, 2017.

\bibitem{KORS2022}
V.~Kazeev, I.~Oseledets, M.~V. Rakhuba, and Ch. Schwab.
\newblock Quantized tensor {FEM} for multiscale problems: diffusion problems in
  two and three dimensions.
\newblock {\em Multiscale Model. Simul.}, 20(3):893--935, 2022.

\bibitem{KRS2013}
V.~Kazeev, O.~Reichmann, and Ch. Schwab.
\newblock Low-rank tensor structure of linear diffusion operators in the {TT}
  and {QTT} formats.
\newblock {\em Linear Algebra Appl.}, 438(11):4204--4221, 2013.

\bibitem{KS2015}
V.~Kazeev and Ch. Schwab.
\newblock Tensor approximation of stationary distributions of chemical reaction
  networks.
\newblock {\em SIAM J. Matrix Anal. Appl.}, 36(3):1221--1247, 2015.

\bibitem{KS2018}
V.~Kazeev and Ch. Schwab.
\newblock Quantized tensor-structured finite elements for second-order elliptic
  {PDE}s in two dimensions.
\newblock {\em Numer. Math.}, 138(1):133--190, 2018.

\bibitem{KKT2013}
V.~A. Kazeev, B.~N. Khoromskij, and E.~E. Tyrtyshnikov.
\newblock Multilevel {T}oeplitz matrices generated by tensor-structured vectors
  and convolution with logarithmic complexity.
\newblock {\em SIAM J. Sci. Comput.}, 35(3):A1511--A1536, 2013.

\bibitem{K2011}
B.~N. Khoromskij.
\newblock {$O(d\log N)$}-quantics approximation of {$N$}-{$d$} tensors in
  high-dimensional numerical modeling.
\newblock {\em Constr. Approx.}, 34(2):257--280, 2011.

\bibitem{K2018}
B.~N. Khoromskij.
\newblock {\em Tensor numerical methods in scientific computing}, volume~19 of
  {\em Radon Series on Computational and Applied Mathematics}.
\newblock De Gruyter, Berlin, 2018.

\bibitem{KO2010}
B.~N. Khoromskij and I.~V. Oseledets.
\newblock {DMRG+ QTT} approach to computation of the ground state for the
  molecular schr{\"o}dinger operator.
\newblock {\em Tech. Rep. 69, MPI MIS Leipzig}, 2010.

\bibitem{KB2009}
T.~G. Kolda and B.~W. Bader.
\newblock Tensor decompositions and applications.
\newblock {\em SIAM Rev.}, 51(3):455--500, 2009.

\bibitem{LM1972}
J.-L. Lions and E.~Magenes.
\newblock {\em Non-homogeneous boundary value problems and applications. {V}ol.
  {I}}, volume Band 181 of {\em Die Grundlehren der mathematischen
  Wissenschaften}.
\newblock Springer-Verlag, New York-Heidelberg, 1972.

\bibitem{MRS2022}
C.~Marcati, M.~Rakhuba, and Ch. Schwab.
\newblock Tensor rank bounds for point singularities in {$\Bbb{R}^3$}.
\newblock {\em Adv. Comput. Math.}, 48(3):Paper No. 18, 57, 2022.

\bibitem{MRU2022}
C.~Marcati, M.~Rakhuba, and J.~E.~M. Ulander.
\newblock Low-rank tensor approximation of singularly perturbed boundary value
  problems in one dimension.
\newblock {\em Calcolo}, 59(1):Paper No. 2, 32, 2022.

\bibitem{MTO2021}
L.~Markeeva, I.~Tsybulin, and I.~Oseledets.
\newblock Q{TT}-isogeometric solver in two dimensions.
\newblock {\em J. Comput. Phys.}, 424:Paper No. 109835, 23, 2021.

\bibitem{O2009b}
I.~V. Oseledets.
\newblock Approximation of matrices with logarithmic number of parameters.
\newblock {\em Dokl. Math.}, 80(2):653--654, 2009.

\bibitem{O2009a}
I.~V. Oseledets.
\newblock On a new tensor decomposition.
\newblock {\em Dokl. Akad. Nauk}, 427(2):168--169, 2009.

\bibitem{O2010}
I.~V. Oseledets.
\newblock Approximation of {$2^d\times 2^d$} matrices using tensor
  decomposition.
\newblock {\em SIAM J. Matrix Anal. Appl.}, 31(4):2130--2145, 2010.

\bibitem{O2011}
I.~V. Oseledets.
\newblock Tensor-train decomposition.
\newblock {\em SIAM J. Sci. Comput.}, 33(5):2295--2317, 2011.

\bibitem{OT2009a}
I.~V. Oseledets and E.~E. Tyrtyshnikov.
\newblock Breaking the curse of dimensionality, or how to use {SVD} in many
  dimensions.
\newblock {\em SIAM J. Sci. Comput.}, 31(5):3744--3759, 2009.

\bibitem{OT2009b}
I.~V. Oseledets and E.~E. Tyrtyshnikov.
\newblock Recursive decomposition of multidimensional tensors.
\newblock {\em Dokl. Akad. Nauk}, 427(1):14--16, 2009.

\bibitem{Schollwoeck:2011:DMRG-MPS}
Ulrich Schollw{\"o}ck.
\newblock The density-matrix renormalization group in the age of matrix product
  states.
\newblock {\em Annals of Physics}, 326(1):96--192, 2011.
\newblock January 2011 Special Issue.

\bibitem{Verstraete:2006:MPS}
F.~Verstraete and J.~I. Cirac.
\newblock Matrix product states represent ground states faithfully.
\newblock {\em Physical Review B}, 73:094423, 2006.

\bibitem{V2003}
G.~Vidal.
\newblock Efficient classical simulation of slightly entangled quantum
  computations.
\newblock {\em Phys. Rev. Lett.}, 91:147902, 2003.

\bibitem{W1993}
S.~R. White.
\newblock Density-matrix algorithms for quantum renormalization groups.
\newblock {\em Phys. Rev. B}, 48:10345--10356, 1993.

\bibitem{W1987}
J.~Wloka.
\newblock {\em Partial differential equations}.
\newblock Cambridge University Press, Cambridge, 1987.

\end{thebibliography}
\bibliographystyle{plain}

\end{document}